%% file: VDMv3.tex
\documentclass[final]{siamltex}

\usepackage{amsmath,amssymb}
\usepackage{latexsym}
\usepackage{graphicx}
\usepackage{color}
\usepackage{subfigure}
\usepackage{psfrag}

\newcommand{\vu}{\boldsymbol{u}}
\newcommand{\vx}{\boldsymbol{x}}

\newcommand{\NN}{\mathbb{N}}

\newcommand{\RR}{\mathbb{R}}
\newcommand{\CC}{\mathbb{C}}
\newcommand{\EE}{\mathbb{E}}
\newcommand{\MM}{\mathcal{M}}

\newcommand{\tr}{\mbox{tr}}
\newcommand{\Ric}{\mbox{Ric}}

\newcommand{\ud}{\textup{d}}
\newcommand{\fg}{\mathfrak{g}}
\newcommand{\fk}{\mathfrak{k}}

\newcommand{\res}{\mbox{res}}
\newcommand{\Hom}{\mbox{Hom}}

\newtheorem{thm}{Theorem}[section]
\newtheorem{lem}[thm]{Lemma}
\newtheorem{cor}[thm]{Corollary}

\newcommand{\argmin}{\operatornamewithlimits{argmin}}

\title{Vector Diffusion Maps and the Connection Laplacian}
\author{
A.~Singer%
\thanks{Department of Mathematics and PACM, Princeton University, Fine Hall, Washington Road, Princeton NJ 08544-1000 USA, email: amits@math.princeton.edu}
\and
H.-T.~Wu%
\thanks{Department of Mathematics, Princeton University, Fine Hall, Washington Road, Princeton NJ 08544-1000 USA, email: hauwu@math.princeton.edu}
}

\date{}

\begin{document}
\maketitle
\begin{center}
\large{\emph{Dedicated to the Memory of Partha Niyogi}}
\vspace{0.2cm}
\end{center}

\begin{abstract}
We introduce {\em vector diffusion maps} (VDM), a new mathematical framework for organizing and
analyzing massive high dimensional data sets, images and shapes.
VDM is a mathematical and algorithmic generalization of diffusion maps and other non-linear dimensionality reduction methods, such as LLE, ISOMAP and Laplacian eigenmaps. While existing methods are either directly or indirectly related to the heat kernel for functions over the data, VDM is based on the heat kernel
for vector fields.
VDM provides tools for organizing complex data sets, embedding them in a low dimensional space, and interpolating and regressing vector fields over the data. In particular, it equips the data with a metric, which we refer to as the {\em vector diffusion distance}.
In the manifold learning setup, where the data set is distributed on (or near) a low dimensional manifold $\MM^d$ embedded in $\RR^{p}$, we prove the relation between VDM and the connection-Laplacian operator for vector fields over the manifold.
\end{abstract}


\section{Introduction}

A popular way to describe the affinities between data points is using a weighted graph, whose vertices correspond to the data points, edges that connect data points with large enough affinities and weights that quantify the affinities.
In the past decade we have been witnessed to the emergence of non-linear dimensionality reduction methods, such as locally linear embedding (LLE) \cite{SamRoweis2000}, ISOMAP \cite{Tenenbaum12222000}, Hessian LLE \cite{Donoho2003}, Laplacian eigenmaps \cite{belkinniyogi2003} and diffusion maps \cite{Coifman20065}. These methods use the local affinities in the weighted graph to learn its global features. They provide invaluable tools for organizing complex networks and data sets, embedding them in a low dimensional space, and studying and regressing functions over graphs. Inspired by recent developments in the mathematical theory of cryo-electron microscopy \cite{amit20093,hadani20091} and synchronization \cite{amit2009,amit2010}, in this paper we demonstrate that in many applications, the representation of the data set can be vastly improved by attaching to every edge of the graph not only a weight but also a linear orthogonal transformation (see Figure \ref{fig:1}).

\begin{figure}
\begin{center}
\includegraphics[width=0.6\columnwidth]{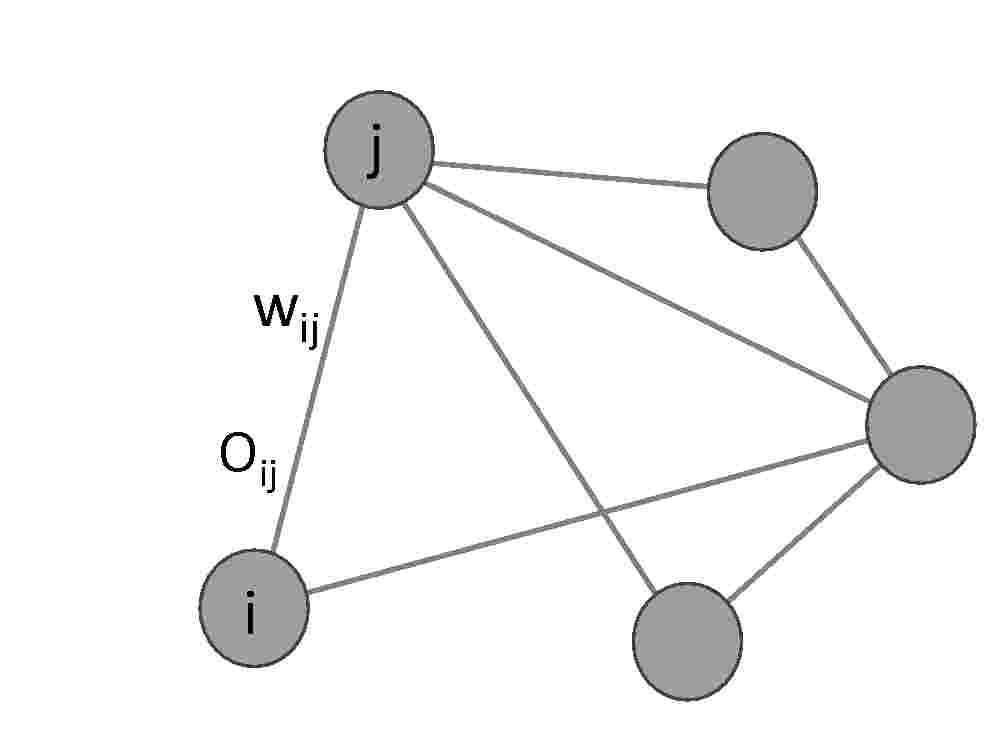}
\end{center}
\caption{In VDM, the relationships between data points are represented as a weighted graph, where the weights $w_{ij}$ are accompanied by linear orthogonal transformations $O_{ij}$.}\label{fig:1}
\end{figure}

\begin{figure}
\begin{center}
\subfigure[$I_i$]{
\includegraphics[width=0.25\textwidth]{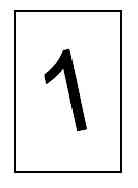}
}
\subfigure[$I_j$]{
\includegraphics[width=0.25\textwidth]{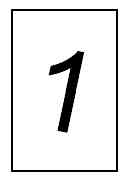}
}
\subfigure[$I_k$]{
\includegraphics[width=0.255\textwidth]{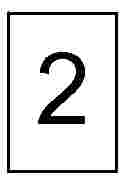}
}
\end{center}
\caption{An example of a weighted graph with orthogonal transformations: $I_i$ and $I_j$ are two different images of the digit one, corresponding to nodes $i$ and $j$ in the graph. $O_{ij}$ is the $2\times 2$ rotation matrix that rotationally aligns $I_j$ with $I_i$ and $w_{ij}$ is some measure for the affinity between the two images when they are optimally aligned. The affinity $w_{ij}$ is large, because the images $I_i$ and $O_{ij}I_j$ are actually the same. On the other hand, $I_k$ is an image of the digit two, and the discrepancy between $I_k$ and $I_i$ is large even when these images are optimally aligned. As a result, the affinity $w_{ik}$ would be small, perhaps so small that there is no edge in the graph connecting nodes $i$ and $k$. The matrix $O_{ik}$ is clearly not as meaningful as $O_{ij}$. If there is no edge between $i$ and $k$, then $O_{ik}$ is not represented in the weighted graph.}
\label{fig:digits}
\end{figure}

Consider, for example, a data set of images, or small patches extracted from images (see, e.g., \cite{Lee,Carlsson}). While weights are usually derived from the pairwise comparison of the images in their original representation, we instead associate the weight $w_{ij}$ to the similarity between image $i$ and image $j$ when they are optimally rotationally aligned. The dissimilarity between images when they are optimally rotationally aligned is sometimes called the rotationally invariant distance \cite{Penczek1996}. We further define the linear transformation $O_{ij}$ as the $2\times 2$ orthogonal transformation that registers the two images (see Figure \ref{fig:digits}). Similarly, for data sets consisting of three-dimensional shapes, $O_{ij}$ encodes the optimal $3\times 3$ orthogonal registration transformation. In the case of manifold learning, the linear transformations can be constructed using local principal component analysis (PCA) and alignment, as discussed in Section \ref{sec:manifold}.

While diffusion maps and other non-linear dimensionality reduction methods are either directly or indirectly related to the heat kernel for functions over the data, our VDM framework is based on the heat kernel
for vector fields. We construct this kernel from the weighted graph and the orthogonal transformations. Through the spectral decomposition of this kernel, VDM defines an embedding of the data in a Hilbert space. In particular, it defines a metric for the data, that is, distances between data points that we call vector diffusion distances. For some applications, the vector diffusion metric is more meaningful than currently used metrics, since it takes into account the linear transformations, and as a result, it provides a better organization of the data. In the manifold learning setup, we prove a convergence theorem illuminating the relation between VDM and the connection-Laplacian operator for vector fields over the manifold.

The paper is organized in the following way: In Section \ref{sec:manifold} we describe the manifold learning setup and a procedure to extract the orthogonal transformations from a point cloud scattered in a high dimensional Euclidean space using local PCA and alignment. In Section \ref{sec:VDM} we specify the vector diffusion mapping of the data set into a finite dimensional Hilbert space. At the heart of the vector diffusion mapping construction lies a certain symmetric matrix that can be normalized in slightly different ways. Different normalizations lead to different embeddings, as discussed in Section \ref{sec:normalizations}. These normalizations resemble the normalizations of the graph Laplacian in spectral graph theory and spectral clustering algorithms. In the manifold learning setup, it is known that when the point cloud is uniformly sampled from a low dimensional Riemannian manifold, then the normalized graph Laplacian approximates the Laplace-Beltrami operator for scalar functions. In Section \ref{convergetoconnlap} we formulate a similar result, stated as Theorem \ref{summary}, for the convergence of the appropriately normalized vector diffusion mapping matrix to the connection-Laplacian operator for vector fields \footnote{One of the main considerations in the way this paper is presented was to make it as accessible as possible, also to readers who are not familiar with differential geometry. Although the connection-Laplacian is essential to the understanding of the mathematical framework that underlies VDM, and differential geometry is extensively used in Appendix \ref{proof} for the proof of Theorem \ref{summary}, we do not assume knowledge of differential geometry in Sections \ref{sec:manifold}-\ref{sec:summary} (except for some parts of Section \ref{gdd}) that detail the algorithmic framework. The concepts of differential geometry that are required for achieving basic familiarity with the connection-Laplacian are explained in Appendix \ref{setupbackground}.}. The proof of Theorem \ref{summary} appears in Appendix \ref{proof}. We verified Theorem \ref{summary} numerically for  spheres of different dimensions, as reported in Section \ref{numerical} and Appendix \ref{app-A}. We also used other surfaces to perform numerical comparisons between the vector diffusion distance, the diffusion distance, and the geodesic distance. In Section \ref{extrapolation} we briefly discuss out-of-sample extrapolation of vector fields via the Nystr\"om extension scheme. The role played by the heat kernel of the connection-Laplacian is discussed in Section \ref{gdd}. We use the well known short time asymptotic expansion of the heat kernel to show the relationship between vector diffusion distances and geodesic distances for nearby points. In Section \ref{applications} we briefly discuss the application of VDM to cryo-electron microscopy, as a prototypical multi-reference rotational alignment problem. We conclude in Section \ref{sec:summary} with a summary followed by a discussion of some other possible applications and extensions of the mathematical framework.

\section{Data sampled from a Riemannian manifold}
\label{sec:manifold}

One of the main objectives in the analysis of a high dimensional large data set is to learn its geometric and topological structure.
Even though the data itself is parameterized as a point cloud in a high dimensional ambient space $\RR^{p}$, the correlation between parameters often suggests the popular ``manifold assumption" that the data points are distributed on (or near) a low dimensional Riemannian manifold $\MM^d$ embedded in $\RR^{p}$, where $d$ is the dimension of the manifold and $d \ll p$.
Suppose that the point cloud consists of $n$ data points $x_1,x_2,\ldots, x_n$ that are viewed as points in $\mathbb{R}^p$ but are restricted to the manifold.
We now describe how the orthogonal transformations $O_{ij}$ can be constructed from the point cloud using local PCA and alignment.

{\bf Local PCA.}
For every data point $x_i$ we suggest to estimate a basis to the tangent plane $T_{x_i}\MM$ to the manifold at $x_i$ using the following procedure which we refer to as local PCA. We fix a scale parameter $\epsilon_{\text{PCA}} > 0$ and define $\mathcal{N}_{x_i,\epsilon_{\text{PCA}}}$ as the neighbors of $x_i$ inside a ball of radius $\sqrt{\epsilon_{\text{PCA}}}$ centered at $x_i$:
$$\mathcal{N}_{x_i,\epsilon_{\text{PCA}}} = \{x_j : 0 < \|x_j-x_i\|_{\mathbb{R}^p} < \sqrt{\epsilon_{\text{PCA}}} \}.$$
Denote the number of neighboring points of $x_i$ by\footnote{Since $N_i$ depends on $\epsilon_{\text{PCA}}$, it should be denoted as $N_{i,\epsilon_{\text{PCA}}}$, but since $\epsilon_{\text{PCA}}$ is kept fixed it is suppressed from the notation, a convention that we use except for cases in which confusion may arise.} $N_i$, that is, $N_i = |\mathcal{N}_{x_i,\epsilon_{\text{PCA}}}|$, and denote the neighbors of $x_i$ by $x_{i_1}, x_{i_2},\ldots, x_{i_{N_i}}$. We assume that $\epsilon_{\text{PCA}}$ is large enough so that $N_i \geq d$, but at the same time $\epsilon_{\text{PCA}}$ is small enough such that $N_i \ll n$. In Theorem \ref{localpcatheorem} we show that a satisfactory choice for $\epsilon_{\text{PCA}}$ is given by $\epsilon_{\text{PCA}} = O(n^{-\frac{2}{d+1}})$, so that $N_i = O(n^{\frac{1}{d+1}})$. In fact, it is even possible to choose $\epsilon_{\text{PCA}} = O(n^{-\frac{2}{d+2}})$ if the manifold does not have a boundary.

Observe that the neighboring points are located near $T_{x_i}\MM$, where deviations are possible either due to curvature or due to neighboring data points that lie slightly off the manifold.
Define $X_i$ to be a $p\times N_i$ matrix whose $j$'th column is the vector $x_{i_j}-x_i$, that is,
$$X_i = \left[\begin{array}{cccc} x_{i_1}-x_i & x_{i_2}-x_i & \ldots & x_{i_{N_i}}-x_i\end{array}  \right].$$
In other words, $X_i$ is the data matrix of the neighbors shifted to be centered at the point $x_i$. Notice, that while it is more common to shift the data for PCA by the mean $\mu_i = \frac{1}{N_i}\sum_{j=1}^{N_i} x_{i_j}$, here we shift the data by $x_i$. Shifting the data by $\mu_i$ is also possible for all practical purposes, but has the slight disadvantage of complicating the proof for the convergence of the local PCA step (see Appendix \ref{localpcatheorem}).

Let $K$ be a $C^2$ positive monotonic decreasing function with support on the interval $[0,1]$, for example, the Epanechnikov kernel $K(u) = (1-u^2)\chi_{[0,1]}$, where $\chi$ is the indicator function \footnote{In fact, $K$ can be chosen in a more general fashion, for example, monotonicity is not required for all theoretical purposes. However, in practice, a monotonic decreasing $K$ leads to a better behavior of the PCA step.}
Let $D_i$ be an $N_i\times N_i$ diagonal matrix with $$D_i(j,j) = \sqrt{K\left(\frac{\|x_i-x_{i_j}\|_{\mathbb{R}^p}}{\sqrt{\epsilon_{\text{PCA}}}}\right)},\quad j=1,2,\ldots,N_i.$$
Define the $p\times N_i$ matrix $B_i$ as
$$B_i = X_i D_i.$$
That is, the $j$'th column of $B_i$ is the vector $(x_{i_j}-x_i)$ scaled by $D_i(j,j)$. The purpose of the scaling is to give more emphasis to nearby points over far away points (recall that $K$ is monotonic decreasing). We denote the singular values of $B_i$ by $\sigma_{i,1} \geq \sigma_{i,2} \geq \cdots \geq \sigma_{i,N_i}$.

In many cases, the intrinsic dimension $d$ is not known in advance and needs to be estimated directly from the data.
If the neighboring points in $\mathcal{N}_{x_i,\epsilon_{\text{PCA}}}$ are located exactly on $T_{x_i}\mathcal{M}$, then $\operatorname{rank}X_i = \operatorname{rank}B_i = d$, and there are only $d$ non-vanishing singular values (i.e., $\sigma_{i,d+1} = \ldots = \sigma_{i,N_i} = 0$). In such a case, the dimension can be estimated as the number of non-zero singular values. In practice, however, due to the curvature effect, there may be more than $d$ non-zero singular values. A common practice is to estimate the dimension as the number of singular values that account for high enough percentage of the variability of the data. That is, one sets a threshold $\gamma$ between 0 and 1 (usually closer to 1 than to 0), and estimates the dimension as the smallest integer $d_i$ for which $$\frac{\sum_{j=1}^{d_i} \sigma_{i,j}^2}{\sum_{j=1}^{N_i} \sigma_{i,j}^2} > \gamma.$$  For example, setting $\gamma=0.9$ means that $d_i$ singular values account for at least $90\%$ variability of the data, while $d_i-1$ singular values account for less than $90\%$. We refer to the smallest integer $d_i$ as the estimated local dimension of $\MM$ at $x_i$. One possible way to estimate the dimension of the manifold would be to use the mean of the estimated local dimensions $d_1,\ldots,d_n$, that is, $\hat{d} = \frac{1}{n}\sum_{i=1}^n d_i$ (and then round it to the closest integer). The mean estimator minimizes the sum of squared errors $\sum_{i=1}^n (d_i-\hat{d})^2$.
We estimate the intrinsic dimension of the manifold by the median value of all the $d_i$'s, that is, we define the estimator $\hat{d}$ for the intrinsic dimension $d$ as $$\hat{d} = \operatorname{median} \{d_1,d_2,\ldots,d_n \}.$$ The median has the property that it minimizes the sum of absolute errors $\sum_{i=1}^n |d_i-\hat{d}|$. As such, estimating the intrinsic dimension by the median is more robust to outliers compared to the mean estimator. In all proceeding steps of the algorithm we use the median estimator $\hat{d}$, but in order to facilitate the notation we write $d$ instead of $\hat{d}$.

Suppose that the singular value decomposition (SVD) of $B_i$ is given by $$B_i = U_i\Sigma_i V_i^T.$$
The columns of the $p\times N_i$ matrix $U_i$ are orthonormal and are known as the left singular vectors
$$U_i = \left[\begin{array}{cccc} u_{i_1} & u_{i_2} & \cdots & u_{i_{N_i}} \end{array} \right].$$
We define the $p\times d$ matrix $O_i$ by the first $d$ left singular vectors (corresponding to the largest singular values):
\begin{equation}
\label{Oi}
O_i = \left[\begin{array}{cccc} u_{i_1} & u_{i_2} & \cdots & u_{i_{d}} \end{array} \right].
\end{equation}
The $d$ columns of $O_i$ are orthonormal, i.e., $O_i^T O_i = I_{d\times d}$. The columns of $O_i$ represent an orthonormal basis to a $d$-dimensional subspace of $\mathbb{R}^p$. This basis is a numerical approximation to an orthonormal basis of the tangent plane $T_{x_i}\MM$. The order of the approximation (as a function of $\epsilon_{\text{PCA}}$ and $n$) is established later, using the fact that the columns of $O_i$ are also the eigenvectors (corresponding to the $d$ largest eigenvalues) of the $p\times p$ covariance matrix $\Xi_i$ given by
\begin{equation}
\label{Xi1}
\Xi_i=\sum_{j=1}^{N_i} K\left(\frac{\|x_i-x_{i_j}\|_{\mathbb{R}^p}}{\sqrt{\epsilon_{\text{PCA}}}}\right)(x_{i_j}-x_i)(x_{i_j}-x_i)^T.
\end{equation}
Since $K$ is supported on the interval $[0,1]$ the covariance matrix $\Xi_i$ can also be represented as
\begin{equation}
\label{Xi2}
\Xi_i = \sum_{j=1}^{n} K\left(\frac{\|x_i-x_j\|_{\mathbb{R}^p}}{\sqrt{\epsilon_{\text{PCA}}}}\right)(x_j-x_i)(x_j-x_i)^T.
\end{equation}
We emphasize that the covariance matrix is never actually formed due to its excessive storage requirements, and all computations are performed with the matrix $B_i$.
We remark that it is also possible to estimate the intrinsic dimension $d$ and the basis $O_i$ using the multiscaled PCA algorithm \cite{Little2010} that uses several different values of $\epsilon_{\text{PCA}}$ for a given $x_i$, but here we try to make our approach as simple as possible while being able to later prove convergence theorems.

{\bf Alignment.}
Suppose $x_i$ and $x_j$ are two nearby points whose Euclidean distance satisfies $\|x_i-x_j \|_{\mathbb{R}^p} < \sqrt{\epsilon}$, where $\epsilon>0$ is a scale parameter different from the scale parameter $\epsilon_{\text{PCA}}$. In fact, $\epsilon$ is much larger than $\epsilon_{\text{PCA}}$ as we later choose $\epsilon = O(n^{-\frac{2}{d+4}})$, while, as mentioned earlier, $\epsilon_{\text{PCA}}=O(n^{-\frac{2}{d+1}})$ (manifolds with boundary) or $\epsilon_{\text{PCA}}=O(n^{-\frac{2}{d+2}})$ (manifolds with no boundary). In any case, $\epsilon$ is small enough so that the tangent spaces $T_{x_i}\MM$ and $T_{x_j}\MM$ are also close.\footnote{In the sense that their Grassmannian distance given approximately by the operator norm $\|O_iO_i^T - O_jO_j^T \|$ is small.}
Therefore, the column spaces of $O_i$ and $O_j$ are almost the same. If the subspaces were to be exactly the same, then the matrices $O_i$ and $O_j$ would have differ by a $d\times d$ orthogonal transformation $O_{ij}$ satisfying $O_i O_{ij} = O_j$, or equivalently $O_{ij} = O_i^T O_j$. In that case, $O_i^T O_j$ is the matrix representation of the operator that transport vectors from $T_{x_j}\MM$ to $T_{x_i}\MM$, viewed as copies of $\mathbb{R}^d$. The subspaces, however, are usually not exactly the same, due to curvature. As a result, the matrix $O_i^T O_j$ is not necessarily orthogonal, and we define $O_{ij}$ as its closest orthogonal matrix, i.e.,
\begin{equation}
\label{HS}
O_{ij} = \argmin_{O \in O(d)} \|O - O_i^T O_j\|_{HS},
\end{equation}
where $\| \cdot \|_{HS}$ is the Hilbert-Schmidt norm (given by $\|A\|_{HS}^2 = \operatorname{Tr}(AA^T)$ for any real matrix $A$) and $O(d)$ is the set of orthogonal $d\times d$ matrices. This minimization problem has a simple solution \cite{Fan1955,Keller1975,Higham1986,arun} via the SVD of $O_i^T O_j$. Specifically, if
\begin{equation*}
O_i^T O_j = U\Sigma V^T
\end{equation*}
is the SVD of $O_i^T O_j$, then $O_{ij}$ is given by
\begin{equation*}
O_{ij} = UV^T.
\end{equation*}

\begin{figure}
\begin{center}
\subfigure{
\includegraphics[width=0.75\textwidth]{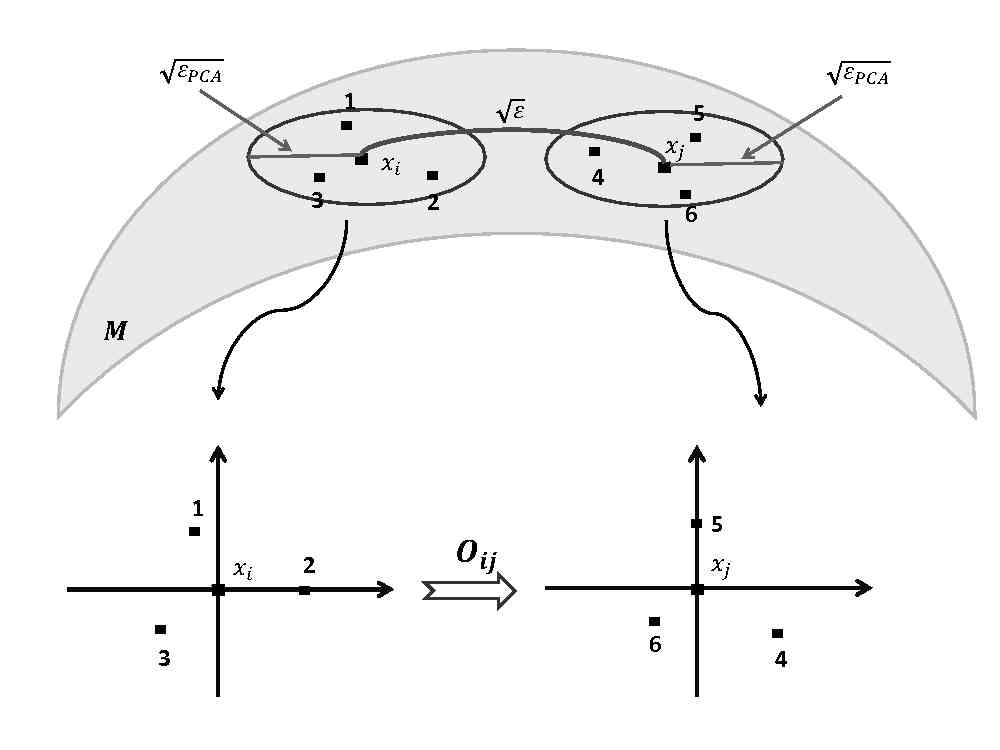}
}
\subfigure{
\includegraphics[width=0.75\textwidth]{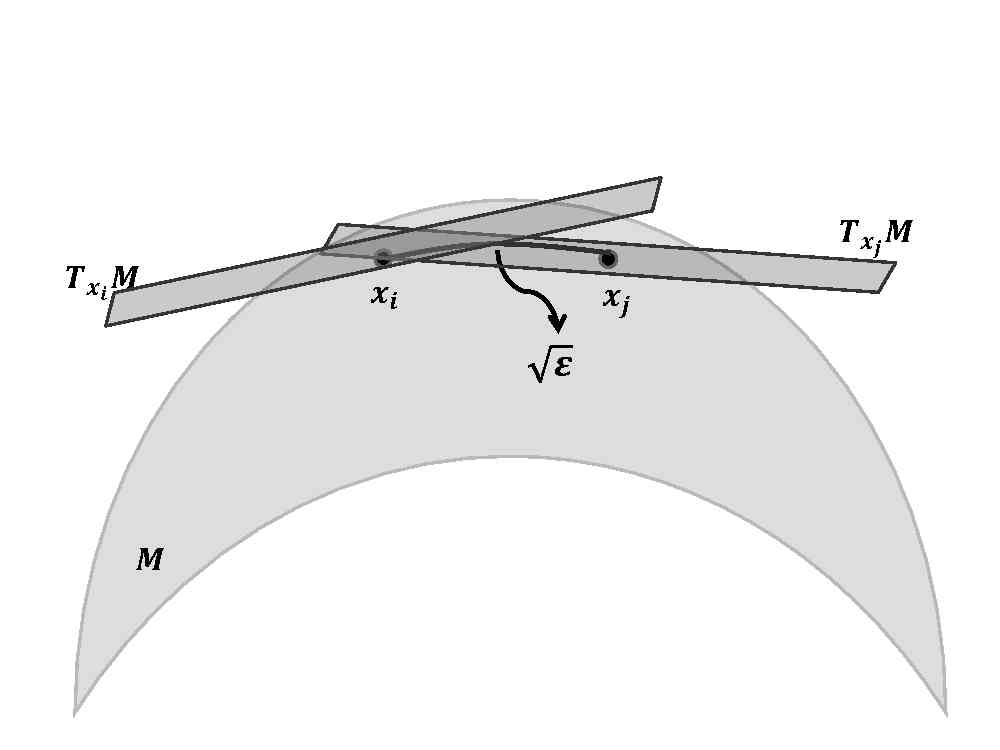}
}
\end{center}
\caption{The orthonormal basis of the tangent plane $T_{x_i}\MM$ is determined by local PCA using data points inside a Euclidean ball of radius $\sqrt{\epsilon_\text{PCA}}$ centered at $x_i$. The bases for $T_{x_i}\MM$ and $T_{x_j}\MM$ are optimally aligned by an orthogonal transformation $O_{ij}$ that can be viewed as a mapping from $T_{x_i}\MM$ to $T_{x_j}\MM$.}
\label{rotref}
\end{figure}

We refer to the process of finding the optimal orthogonal transformation between bases as alignment. Later we show that the matrix $O_{ij}$ is an approximation to the parallel transport operator (see Appendix \ref{setupbackground}) from $T_{x_j}\mathcal{M}$ to $T_{x_i}\mathcal{M}$ whenever $x_i$ and $x_j$ are nearby.

Note that not all bases are aligned; only the bases of nearby points are aligned. We set $E$ to be the edge set of the undirected graph over $n$ vertices that correspond to the data points, where an edge between $i$ and $j$ exists iff their corresponding bases are aligned by the algorithm\footnote{We do not align a basis with itself, so the edge set $E$ does not contain self loops of the form $(i,i)$.} (or equivalently, iff $0<\|x_i-x_j\|_{\mathbb{R}^p} < \sqrt{\epsilon}$).
The weights $w_{ij}$ are defined using a kernel function $K$ as \footnote{Notice that the weights are only a function of the Euclidean distance between data points; another possibility, which we do not consider in this paper, is to include the Grassmannian distance $\|O_iO_i^T - O_jO_j^T \|_{2}$ into the definition of the weight.}
\begin{equation}\label{wij}
w_{ij}=K\left(\frac{\|x_i-x_j \|_{\mathbb{R}^p}}{\sqrt{\epsilon}}\right),
\end{equation}
where we assume that $K$ is supported on the interval $[0,1]$.
For example, the Gaussian kernel $K(u) = \exp\{-u^2\}\chi_{[0,1]}$ leads to weights of the form $w_{ij} = \exp\{-\frac{\|x_i-x_j \|^2}{\epsilon}\}$ for $0 < \|x_i-x_j\| < \sqrt{\epsilon}$ and $0$ otherwise. We emphasize that the kernel $K$ used for the definition of the weights $w_{ij}$ could be different than the kernel used for the previous step of local PCA.

\section{Vector diffusion mapping}
\label{sec:VDM}

We construct the following matrix $S$:
\begin{equation}
\label{S}
S(i,j) = \left\{\begin{array}{ccc}
                  w_{ij}O_{ij} &  & (i,j)\in E, \\
                  0_{d\times d} &  & (i,j)\notin E.
                \end{array}
 \right.
\end{equation}
That is, $S$ is a block matrix, with $n\times n$ blocks, each of which is of size $d\times d$. Each block is either a $d\times d$ orthogonal transformation $O_{ij}$ multiplied by the scalar weight $w_{ij}$, or a zero $d\times d$ matrix.\footnote{As mentioned in the previous footnote, the edge set does not contain self-loops, so $w_{ii}=0$ and $S(i,i)=0_{d\times d}$.} The matrix $S$ is symmetric since $O_{ij}^T = O_{ji}$ and $w_{ij}=w_{ji}$, and its overall size is $nd \times nd$. We define a diagonal matrix $D$ of the same size, where the diagonal blocks are scalar matrices given by
\begin{equation}\label{Dvdm}
D(i,i) = \operatorname{deg}(i) I_{d\times d},
\end{equation}
and
\begin{equation}
\operatorname{deg}(i) = \sum_{j:(i,j)\in E} w_{ij}
\end{equation}
is the weighted degree of node $i$. The matrix $D^{-1}S$ can be applied to vectors $v$ of length $nd$, which we regard as $n$ vectors of length $d$, such that $v(i)$ is a vector in $\mathbb{R}^d$ viewed as a vector in $T_{x_i}\mathcal{M}$. The matrix $D^{-1}S$ is an averaging operator for vector fields, since
\begin{equation}\label{Cmatrix}
(D^{-1}Sv)(i) = \frac{1}{\operatorname{deg}(i)} \sum_{j: (i,j)\in E} w_{ij}O_{ij}v(j).
\end{equation}
This implies that the operator $D^{-1}S$ transport vectors from the tangent spaces $T_{x_j}\mathcal{M}$ (that are nearby to $T_{x_i}\mathcal{M}$) to $T_{x_i}\mathcal{M}$  and then averages the transported vectors in $T_{x_i}\mathcal{M}$.

Notice that diffusion maps and other non-linear dimensionality reduction methods make use of the weight matrix $W=(w_{ij})_{i,j=1}^n$, but not of the transformations $O_{ij}$. In diffusion maps, the weights are used to define a discrete random walk over the graph, where the transition probability $a_{ij}$ in a single time step from node $i$ to node $j$ is given by
\begin{equation}
a_{ij} = \displaystyle{\frac{w_{ij}}{\operatorname{deg}(i)}}.
\end{equation}
The Markov transition matrix $A=(a_{ij})_{i,j=1}^n$ can be written as
\begin{equation}\label{Amatrix}
A=\mathcal{D}^{-1}W,
\end{equation}
where $\mathcal{D}$ is $n\times n$ diagonal matrix with
\begin{equation}
\label{Dii}
\mathcal{D}(i,i) = \operatorname{deg}(i).
\end{equation}
While $A$ is the Markov transition probability matrix in a single time step, $A^t$ is the transition matrix for $t$ steps. In particular, $A^t(i,j)$ sums the probabilities of all paths of length $t$ that start at $i$ and end at $j$. Coifman and Lafon \cite{Coifman20065,LafonPHD} showed that $A^t$ can be used to define an inner product in a Hilbert space. Specifically, the matrix $A$ is similar to the symmetric matrix $\mathcal{D}^{-1/2}W\mathcal{D}^{-1/2}$ through $A = \mathcal{D}^{-1/2} (\mathcal{D}^{-1/2}W\mathcal{D}^{-1/2}) \mathcal{D}^{1/2}$. It follows that $A$ has a complete set of real eigenvalues and eigenvectors $\{\mu_l\}_{l=1}^n$ and $\{\phi_l\}_{l=1}^n$, respectively, satisfying $A\phi_l = \mu_l \phi_l$. Their diffusion mapping $\Phi_t$ is given by
\begin{equation}\label{diffusionmap}
\Phi_t(i) = (\mu_1^t \phi_1(i), \mu_2^t \phi_2(i), \ldots, \mu_n^t \phi_n(i)),
\end{equation}
where $\phi_l(i)$ is the $i$'th entry of the eigenvector $\phi_l$. The mapping $\Phi_t$ satisfies
\begin{equation}
\label{eq:dm}
\sum_{k=1}^n \frac{A^t(i,k)}{\sqrt{\operatorname{deg}(k)}} \frac{A^t(j,k)}{\sqrt{\operatorname{deg}(k)}} = \langle \Phi_t(i), \Phi_t(j) \rangle,
\end{equation}
where $\langle \cdot , \cdot \rangle$ is the usual dot product over Euclidean space.
The metric associated to this inner product is known as the {\em diffusion distance}. The diffusion distance ${d}_{{\text{DM}},t}(i,j)$ between $i$ and $j$ is given by
\begin{equation}\label{dmd}
d_{{\text{DM}},t}^2(i,j) = \sum_{k=1}^n \frac{(A^t(i,k) - A^t(j,k))^2}{\operatorname{deg}(k)}  = \langle \Phi_t(i), \Phi_t(i) \rangle + \langle \Phi_t(j), \Phi_t(j) \rangle - 2 \langle \Phi_t(i), \Phi_t(j) \rangle.
\end{equation}
Thus, the diffusion distance between $i$ and $j$ is the weighted-$\ell_2$ proximity between the probability clouds of random walkers starting at $i$ and $j$ after $t$ steps.

In the VDM framework, we define the affinity between $i$ and $j$ by considering all paths of length $t$ connecting them, but instead of just summing the weights of all paths, we {\em sum the transformations}. A path of length $t$ from $j$ to $i$ is some sequence of vertices $j_0,j_1,\ldots,j_t$ with $j_0=j$ and $j_t=i$ and its corresponding orthogonal transformation is obtained by multiplying the orthogonal transformations along the path in the following order:
\begin{equation}
O_{j_t,j_{t-1}}\cdots O_{j_2,j_1}O_{j_1,j_0}.
\end{equation}
Every path from $j$ to $i$ may therefore result in a different transformation. This is analogous to the parallel transport operator from differential geometry that depends on the path connecting two points whenever the manifold has curvature (e.g., the sphere). Thus, when adding transformations of different paths, cancelations may happen. We would like to define the affinity between $i$ and $j$ as the consistency between these transformations, with higher affinity expressing more agreement among the transformations that are being averaged. To quantify this affinity, we consider again the matrix $D^{-1}S$ which is similar to the symmetric matrix
\begin{equation}
\label{tilde-S}
\tilde{S} = D^{-1/2}SD^{-1/2}
\end{equation}
through $D^{-1}S = D^{-1/2} \tilde{S} D^{1/2}$ and define the affinity between $i$ and $j$ as $\|\tilde{S}^{2t}(i,j)\|^2_{HS}$, that is, as the squared HS norm of the $d\times d$ matrix $\tilde{S}^{2t}(i,j)$, which takes into account all paths of length $2t$, where $t$ is a positive integer. In a sense, $\|\tilde{S}^{2t}(i,j)\|^2_{HS}$ measures not only the number of paths of length $2t$ connecting $i$ and $j$ but also the amount of agreement between their transformations. That is, for a fixed number of paths, $\|\tilde{S}^{2t}(i,j)\|^2_{HS}$ is larger when the path transformations are in agreement, and is smaller when they differ.

Since $\tilde{S}$ is symmetric, it has a complete set of eigenvectors $v_1,v_2,\ldots,v_{nd}$ and eigenvalues $\lambda_1, \lambda_2, \ldots, \lambda_{nd}$. We order the eigenvalues in decreasing order of magnitude $|\lambda_1| \geq |\lambda_2| \geq \ldots \geq |\lambda_{nd}|$. The spectral decompositions of $\tilde{S}$ and $\tilde{S}^{2t}$ are given by
\begin{equation}
\tilde{S}(i,j) = \sum_{l=1}^{nd} \lambda_l v_l(i) v_l(j)^T, \quad \mbox{and}\quad
\tilde{S}^{2t}(i,j) = \sum_{l=1}^{nd} \lambda_l^{2t} v_l(i) v_l(j)^T,
\end{equation}
where $v_l(i)\in \mathbb{R}^d$ for $i=1,\ldots,n$ and $l=1,\ldots,nd$.
The HS norm of $\tilde{S}^{2t}(i,j)$ is calculated using the trace:
\begin{equation}
\label{HS-norm}
\|\tilde{S}^{2t}(i,j)\|^2_{HS} = \operatorname{Tr} \left[ \tilde{S}^{2t}(i,j) \tilde{S}^{2t}(i,j)^T \right] = \sum_{l,r=1}^{nd} (\lambda_l \lambda_r)^{2t} \langle v_l(i), v_r(i) \rangle \langle v_l(j), v_r(j) \rangle.
\end{equation}
It follows that the affinity $\|\tilde{S}^{2t}(i,j)\|^2_{HS}$ is an inner product for the finite dimensional Hilbert space $\mathbb{R}^{(nd)^2}$ via the mapping $V_t$:
\begin{equation}
\label{vec-map}
V_t: i \mapsto \left( (\lambda_l \lambda_r)^{t} \langle v_l(i), v_r(i) \rangle\right)_{l,r=1}^{nd}.
\end{equation}
That is,
\begin{equation}
\|\tilde{S}^{2t}(i,j)\|^2_{HS} = \langle V_t(i), V_t(j) \rangle.
\end{equation}
Note that in the manifold learning setup, the embedding $i \mapsto V_t(i)$ is invariant to the choice of basis for $T_{x_i}\mathcal{M}$ because the dot products $\langle v_l(i), v_r(i)\rangle$ are invariant to orthogonal transformations. We refer to $V_t$ as the vector diffusion mapping.

From the symmetry of the dot products $\langle v_l(i), v_r(i) \rangle = \langle v_r(i), v_l(i) \rangle$, it is clear that $\|\tilde{S}^{2t}(i,j)\|^2_{HS}$ is also an inner product for the finite dimensional Hilbert space $\mathbb{R}^{nd(nd+1)/2}$ corresponding to the mapping $$i \mapsto \left( c_{lr}(\lambda_l \lambda_r)^{t} \langle v_l(i), v_r(i) \rangle\right)_{1 \leq l\leq r \leq nd},$$
where
$$c_{lr} = \left\{\begin{array}{ccc}
                    \sqrt{2} &  & l < r, \\
                    1 &  & l=r.
                  \end{array}\right.$$
We define the symmetric vector diffusion distance $d_{\text{VDM},t}(i,j)$ between nodes $i$ and $j$ as
\begin{equation}
\label{Dt}
d_{\text{VDM},t}^2(i,j) = \langle V_t(i), V_t(i) \rangle + \langle V_t(j), V_t(j) \rangle - 2\langle V_t(i), V_t(j) \rangle.
\end{equation}

The matrices $I-\tilde{S}$ and $I+\tilde{S}$ are positive semidefinite due to the following identity:
\begin{equation}
v^T (I \pm D^{-1/2} S D^{-1/2}) v = \sum_{(i,j)\in E} \left\|\frac{v(i)}{\sqrt{\operatorname{deg}(i)}} \pm \frac{w_{ij}O_{ij}v(j)}{\sqrt{\operatorname{deg}(j)}} \right\|^2 \geq 0,
\end{equation}
for any $v\in \mathbb{R}^{nd}$. As a consequence, all eigenvalues $\lambda_l$ of $\tilde{S}$ reside in the interval $[-1,1]$. In particular, for large enough $t$, most terms of the form $(\lambda_l \lambda_r)^{2t}$ in (\ref{HS-norm}) are close to 0, and $\|\tilde{S}^{2t}(i,j)\|^2_{HS}$ can be well approximated by using only the few largest eigenvalues and their corresponding eigenvectors.
This lends itself into an efficient approximation of the vector diffusion distances $d_{\text{VDM},t}(i,j)$ of (\ref{Dt}), and it is not necessary to raise the matrix $\tilde{S}$ to its $2t$ power (which usually results in dense matrices). Thus, for any $\delta > 0$, we define the truncated vector diffusion mapping $V_t^\delta$ that embeds the data set in $\mathbb{R}^{m^2}$ (or equivalently, but more efficiently in $\mathbb{R}^{m(m+1)/2}$) using the eigenvectors $v_1,\ldots,v_{m}$ as
\begin{equation}
\label{Vtdelta}
V_t^{\delta}: i \mapsto \left( (\lambda_l \lambda_r)^{t} \langle v_l(i), v_r(i) \rangle\right)_{l,r=1}^m
\end{equation}
where $m=m(t,\delta)$ is the largest integer for which $\left(\displaystyle{\frac{\lambda_m}{\lambda_1}}\right)^{2t} > \delta$ and $\left(\displaystyle{\frac{\lambda_{m+1}}{\lambda_1}}\right)^{2t} \leq \delta$.

We remark that we define $V_t$ through $\|\tilde{S}^{2t}(i,j)\|^2_{HS}$ rather than through $\|\tilde{S}^{t}(i,j)\|^2_{HS}$, because we cannot guarantee that in general all eigenvalues of $\tilde{S}$ are non-negative. In Section \ref{gdd}, we show that in the continuous setup of the manifold learning problem all eigenvalues are non-negative. We anticipate that for most practical applications that correspond to the manifold assumption, all negative eigenvalues (if any) would be small in magnitude (say, smaller than $\delta$). In such cases, one can use any real $t>0$ for the truncated vector diffusion map $V_t^\delta$.

\section{Normalized Vector Diffusion Mappings}
\label{sec:normalizations}

It is also possible to obtain slightly different vector diffusion mappings using different normalizations of the matrix $S$. These normalizations are similar to the ones used in the diffusion map framework \cite{Coifman20065}. For example, notice that
\begin{equation}
\label{wv}
w_l=D^{-1/2}v_l
\end{equation}
are the right eigenvectors of $D^{-1}S$, that is, $D^{-1}S w_l = \lambda_l w_l$. We can thus define another vector diffusion mapping, denoted $V'_t$, as
\begin{equation}
\label{vec-map2}
V'_t: i \mapsto \left( (\lambda_l \lambda_r)^{t} \langle w_l(i), w_r(i) \rangle\right)_{l,r=1}^{nd}.
\end{equation}
From (\ref{wv}) it follows that $V'_t$ and $V_t$ satisfy the relations
\begin{equation}
V'_t(i) = \frac{1}{\text{deg}(i)}V_t(i),
\end{equation}
and
\begin{equation}
\langle V'_t(i), V'_t(j) \rangle = \frac{\langle V_t(i), V_t(j) \rangle }{\text{deg}(i) \text{deg}(j)}.
\end{equation}
As a result,
\begin{equation}
\label{vec-map2-innerproduct}
\langle V'_t(i),V'_t(j)\rangle = \frac{\|\tilde{S}^{2t}(i,j)\|^2_{HS}}{\text{deg}(i)\text{deg}(j)} = \frac{\|(D^{-1}S)^{2t}(i,j)\|^2_{HS}}{\text{deg}(j)^2}.
\end{equation}
In other words, the Hilbert-Schmidt norm of the matrix $D^{-1}S$ leads to an embedding of the data set in a Hilbert space only upon proper normalization by the vertex degrees (similar to the normalization by the vertex degrees in (\ref{eq:dm}) and (\ref{dmd}) for the diffusion map). We define the associated vector diffusion distances as
\begin{equation}
\label{Dt2}
{d^2_{\text{VDM}',t}}(i,j) = \langle V'_t(i), V'_t(i) \rangle + \langle V'_t(j), V'_t(j) \rangle - 2\langle V'_t(i), V'_t(j) \rangle.
\end{equation}
The distances are related by ${d^2_{\text{VDM}',t}}(i,j) = \frac{d^2_{\text{VDM},t}(i,j)}{\text{deg}(i)\text{deg}(j)}$.

We comment that the normalized mappings $i \mapsto \frac{V_t(i)}{\| V_t(i) \|}$ and $i \mapsto \frac{V'_t(i)}{\|V'_t(i) \|}$ that map the data points to the unit sphere are equivalent, that is,
\begin{equation}
\frac{V'_t(i)}{\|V'_t(i) \|} = \frac{V_t(i)}{\| V_t(i) \|}.
\end{equation}
This means that the angles between pairs of embedded points are the same for both mappings. For diffusion map, it has been observed that in some cases the distances $\|\frac{\Phi_t(i)}{\|\Phi_t(i) \|} - \frac{\Phi_t(i)}{\|\Phi_t(i) \|} \|$ are more meaningful than $\|\Phi_t(i) - \Phi_t(j) \|$ (see, for example, \cite{Goldberg}).
This may also suggest the usage of the distances $\|\frac{V_t(i)}{\|V_t(i) \|} - \frac{V_t(i)}{\|V_t(i) \|} \|$ in the VDM framework.

Another important family of normalized diffusion mappings is obtained by the following procedure. Suppose $0\leq \alpha \leq 1$, and define the symmetric matrices $W_\alpha$ and $S_\alpha$ as
\begin{equation}
\label{Walpha}
W_\alpha = \mathcal{D}^{-\alpha}W \mathcal{D}^{-\alpha},
\end{equation}
and
\begin{equation}
\label{Salpha}
S_\alpha = D^{-\alpha}S D^{-\alpha}.
\end{equation}
We define the weighted degrees $\operatorname{deg}_\alpha(1),\ldots,\operatorname{deg}_\alpha(n)$ corresponding to $W_\alpha$ by
$$\operatorname{deg}_\alpha(i) = \sum_{j=1}^n W_\alpha(i,j),$$
the $n\times n$ diagonal matrix $\mathcal{D}_\alpha$ as
\begin{equation}
\label{Dalpha}
\mathcal{D}_\alpha(i,i) = \operatorname{deg}_\alpha(i),
\end{equation}
and the $n\times n$ block diagonal matrix $D_\alpha$ (with blocks of size $d\times d$) as
\begin{equation}
\label{mDalpha}
D_\alpha(i,i) = \operatorname{deg}_\alpha(i) I_{d\times d}.
\end{equation}
We can then use the matrices $S_\alpha$ and $D_\alpha$ (instead of $S$ and $D$) to define the vector diffusion mappings $V_{\alpha,t}$ and $V_{\alpha,t}'$. Notice that for $\alpha=0$ we have $S_0=S$ and $D_0=D$, so that $V_{0,t}=V_t$ and $V'_{0,t}=V'_t$. The case $\alpha=1$ turns out to be especially important as discussed in the next Section.

\section{Convergence to the connection-Laplacian}
\label{convergetoconnlap}
\input{convergence}

\section{Numerical simulations}
\label{numerical}
\input{numerical}

\section{Out-of-sample extension of vector fields}
\label{extrapolation}
\input{extrapolation}

\section{The continuous case: heat kernels}
\label{gdd}
\input{heatkernel}

\section{Application of VDM to Cryo-Electron Microscopy}
\label{applications}
\input{realapplication}

\section{Summary and Discussion}
\label{sec:summary}
\input{summary}

\section{Acknowledgements}
A. Singer was partially supported by Award Number DMS-0914892 from the NSF, by Award Number FA9550-09-1-0551 from AFOSR, by Award Number R01GM090200 from the NIGMS and by the Alfred P. Sloan Foundation. H.-T. Wu acknowledges support by FHWA grant DTFH61-08-C-00028. The authors would like to thank Charles Fefferman for various discussions regarding this work. They also express gratitude to the audiences of the seminars at Tel Aviv University, the Weizmann Institute of Science, Princeton University, Yale University, Stanford University and the Air Force, where parts of this work have been presented during 2010-2011.

\bibliographystyle{plain}
\bibliography{connLap}

\appendix

\section{Some Differential Geometry Background}
\label{setupbackground}
\input{background}

\section{Proof of Theorem \ref{summary}, Theorem \ref{connlapbdry} and Theorem \ref{summaryheatkernel}}
\label{proof}

\input{proof6}

\section{Multiplicities of eigen-1-forms of Connection Laplacian over $S^n$}
\label{app-A}
\input{Snspectrum}

\end{document}

%% file: convergence.tex
For diffusion maps, the discrete random walk over the data points converges to a continuous diffusion process over that manifold in the limit $n\to \infty$ and $\epsilon \to 0$. This convergence can be stated in terms of the normalized graph Laplacian $L$ given by $$L=\mathcal{D}^{-1}W-I.$$ In the case where the data points $\left\{{x}_i\right\}_{i=1}^n$ are
sampled independently from the uniform distribution over $\mathcal M^d$, the graph
Laplacian converges pointwise to the Laplace-Beltrami operator, as we have the
following proposition \cite{LafonPHD,Belkin3,Singer2006a,Hein}: If $f:\mathcal M^d \rightarrow \mathbb{R}$ is a smooth function (e.g., $f\in C^3(\mathcal{M})$),
then with high probability
\begin{equation}
\label{eq:bias1} \frac{1}{\epsilon}\sum_{j=1}^N L_{ij} f({x}_j) = \frac{1}{2}\Delta_{\mathcal{M}} f({x}_i) + O\left(\epsilon + \frac{1}{n^{1/2}\epsilon^{1/2+d/4}}\right),
\end{equation}
where $\Delta_{\mathcal{M}}$ is the Laplace-Beltrami operator on $\mathcal{M}^d$. The error consists of two terms: a bias term $O(\epsilon)$ and a variance term that decreases as $1/\sqrt{n}$, but also depends on $\epsilon$. Balancing the two terms may lead to an optimal choice of the parameter $\epsilon$ as a function of the number of points $n$. In the case of uniform sampling, Belkin and Niyogi \cite{BelkinNiyogi2007} have shown that the eigenvectors of the graph Laplacian converge to the eigenfunctions of the Laplace-Beltrami operator on the manifold, which is stronger than the pointwise convergence given in (\ref{eq:bias1}).

In the case where the data points $\left\{{x}_i\right\}_{i=1}^n$ are independently sampled from a probability density function $p(x)$ whose support is a $d$-dimensional manifold $\MM^d$ and satisfies some mild conditions, the graph Laplacian converges pointwise to the Fokker-Planck operator as stated in following proposition \cite{LafonPHD,Belkin3,Singer2006a,Hein}: If $f\in C^3(\mathcal{M})$,
then with high probability
\begin{equation}
\label{eq:bias} \frac{1}{\epsilon}\sum_{j=1}^N L_{ij} f({x}_j) = \frac{1}{2}\Delta_{\mathcal{M}} f({x}_i) + \nabla U({x}_i)\cdot \nabla f({x_i})
+ O\left(\epsilon + \frac{1}{n^{1/2}\epsilon^{1/2+d/4}}\right),
\end{equation}
where the potential term $U$ is given by $U({x})=-2\log p({x})$. The error is interpreted in the same way as in the uniform sampling case. In \cite{Coifman20065} it is shown that it is possible to recover the Laplace-Beltrami operator also for non-uniform sampling processes using $W_1$ and $\mathcal{D}_1$ (that correspond to $\alpha=1$ in (\ref{Walpha}) and (\ref{mDalpha})). The matrix $\mathcal{D}_1^{-1}W_1-I$ converges to the Laplace-Beltrami operator independently of the sampling density function $p(x)$.

For VDM, we prove in Appendix \ref{proof} the following theorem, Theorem \ref{summary}, that states that the matrix $D_\alpha^{-1}S_\alpha-I$, where $0\leq \alpha \leq1$, converges to the connection-Laplacian operator (defined via the covariant derivative,  see Appendix \ref{setupbackground} and \cite{petersen}) plus some potential terms depending on $p(x)$. In particular, $D_1^{-1}S_1-I$ converges to the connection-Laplacian operator, without any additional potential terms. Using the terminology of spectral graph theory, it may thus be appropriate to call $D_1^{-1}S_1-I$ the {\em connection-Laplacian of the graph}.

The main content of Theorem \ref{summary} specifies the way in which VDM generalizes diffusion maps: while diffusion mapping is based on the heat kernel and the Laplace-Beltrami operator over scalar functions, VDM is based on the heat kernel and the connection-Laplacian over vector fields. While for diffusion maps, the computed eigenvectors are discrete approximations of the Laplacian eigenfunctions, for VDM, the $l$-th eigenvector $v_l$ of $D_1^{-1}S_1-I$ is a discrete approximation of the $l$-th eigen-vector field $X_l$ of the connection-Laplacian $\nabla^2$ over $\MM$, which satisfies $\nabla^2X_l=-\lambda_lX_l$ for some $\lambda_l\geq0$.

In the formulation of the Theorem \ref{summary}, as well as in the remainder of the paper, we slightly change the notation used so far in the paper, as we denote the observed data points in $\mathbb{R}^p$ by $\iota(x_1),\iota(x_2),\ldots,\iota(x_n)$, where $\iota:\mathcal{M}\hookrightarrow\RR^p$ is the embedding of the Riemannian manifold $\mathcal{M}$ in $\mathbb{R}^p$. Furthermore, we denote by $\iota_* T_{x_i}\MM$ the $d$-dimensional subspace of $\mathbb{R}^p$ which is the embedding of $T_{x_i}\MM$ in $\mathbb{R}^p$. It is important to note that in the manifold learning setup, the manifold $\mathcal{M}$, the embedding $\iota$ and the points $x_1,x_2,\ldots,x_n \in \mathcal{M}$ are assumed to exist but cannot be directly observed.

%

\begin{thm}\label{summary}
Let $\iota:\mathcal{M}\hookrightarrow\RR^p$ be a smooth d-dim closed Riemannian manifold embedded in $\RR^p$, with metric $g$ induced from the canonical metric on $\RR^p$. Let $K\in C^2([0,1))$ be a positive function. For $\epsilon>0$, let $K_\epsilon\left(x_i,x_j\right)=K\left(\frac{\|\iota(x_i)-\iota(x_j)\|_{\mathbb{R}^p}}{\sqrt{\epsilon}}\right)$ for $0<\|\iota(x_i)-\iota(x_j)\|<\sqrt{\epsilon}$, and $K_\epsilon\left(x_i,x_j\right)=0$ otherwise. Let the data set $\{x_i\}_{i=1,...,n}$ be independently distributed according to the probability density function $p(x)$ supported on $\mathcal{M}$, where $p$ is uniformly bounded from below and above, that is, $0<p_m\leq p(x)\leq p_M<\infty$. Define \textit{the estimated probability density distribution} by
\[
p_\epsilon(x_i)=\sum_{j=1}^n K_\epsilon\left(x_i,x_j\right)
\]
and for $0\leq \alpha\leq 1$ define \textit{the $\alpha$-normalized kernel $K_{\epsilon,\alpha}$} by
\[
K_{\epsilon,\alpha}(x_i,x_j)=\frac{K_{\epsilon}(x_i,x_j)}{p^\alpha_\epsilon(x_i)p^\alpha_\epsilon(x_j)}.
\]
Then, using $\epsilon_{\text{PCA}}=O(n^{-\frac{2}{d+2}})$ for $X\in C^3(T\MM)$ and for all $x_i$ with high probability (w.h.p.)
\begin{eqnarray}
&&\frac{1}{\epsilon}\left[\frac{\sum_{j=1}^n K_{\epsilon,\alpha}\left(x_i,x_j\right)O_{ij}\bar{X}_j}{\sum_{j=1}^n K_{\epsilon,\alpha}\left(x_i,x_j\right)} - \bar{X}_i\right]\nonumber \\
&=&\frac{m_2}{2dm_0}\left(\left\langle \iota_*\left\{\nabla^2X(x_i)+d\frac{\int_{S^{d-1}} \nabla_{\theta}X(x_i)\nabla_\theta(p^{1-\alpha})(x_i)\ud\theta}{p^{1-\alpha}(x_i)}\right\}, u_l(x_i)\right\rangle\right)_{l=1}^d \label{summary:eq}\\
&&+ O\left(\epsilon^{1/2} + \epsilon^{-1}n^{-\frac{3}{d+2}} + n^{-1/2}\epsilon^{-(d/4 + 1/2)}\right)\nonumber \\
&=& \frac{m_2}{2dm_0}\left(\left\langle \iota_*\left\{\nabla^2X(x_i)+d\frac{\int_{S^{d-1}} \nabla_{\theta}X(x_i)\nabla_\theta(p^{1-\alpha})(x_i)\ud\theta}{p^{1-\alpha}(x_i)}\right\}, e_l(x_i)\right\rangle\right)_{l=1}^d \nonumber \\
&&+ O\left(\epsilon^{1/2} + \epsilon^{-1}n^{-\frac{3}{d+2}} + n^{-1/2}\epsilon^{-(d/4 + 1/2)}\right)\nonumber
\end{eqnarray}
where $\nabla^2$ is the connection-Laplacian, $\bar{X}_i \equiv \left(\langle \iota_*X(x_i),u_l(x_i)\rangle\right)^d_{l=1} \in \mathbb{R}^d$ for all $i$, $\{u_l(x_i)\}_{l=1,...,d}$ is an orthonormal basis for a $d$-dimensional subspace of $\mathbb{R}^p$ determined by local PCA (i.e., the columns of $O_i$), $\{e_l(x_i)\}_{l=1,...,d}$ is an orthonormal basis for $\iota_*T_{x_i}\mathcal{M}$,
$m_l=\int_{\mathbb{R}^d}\|x\|^l K(\|x\|)\ud x$, and $O_{ij}$ is the optimal orthogonal transformation determined by the alignment procedure. In particular, when $\alpha=1$ we have
\begin{eqnarray}\label{summary:eq2}
\frac{1}{\epsilon}\left[\frac{\sum_{j=1}^n K_{\epsilon,1}\left(x_i,x_j\right)O_{ij}\bar{X}_j}{\sum_{j=1}^n K_{\epsilon,1}\left(x_i,x_j\right)} - \bar{X}_i\right]&=&\frac{m_2}{2dm_0}\left(\langle \iota_*\nabla^2 X(x_i),e_l(x_i)\rangle\right)_{l=1}^d \\
&&+O\left(\epsilon^{1/2} + \epsilon^{-1}n^{-\frac{3}{d+2}} + n^{-1/2}\epsilon^{-(d/4 + 1/2)}\right).\nonumber
\end{eqnarray}
Furthermore, for $\epsilon = O(n^{-\frac{2}{d+4}})$, almost surely,
\begin{eqnarray}
&&\lim_{n\to \infty} \frac{1}{\epsilon}\left[\frac{\sum_{j=1}^n K_{\epsilon,\alpha}\left(x_i,x_j\right)O_{ij}\bar{X}_j}{\sum_{j=1}^n K_{\epsilon,\alpha}\left(x_i,x_j\right)} - \bar{X}_i\right]\\
&=&\frac{m_2}{2dm_0}\left(\left\langle \iota_*\left\{\nabla^2X(x_i)+d\frac{\int_{S^{d-1}} \nabla_{\theta}X(x_i)\nabla_\theta(p^{1-\alpha})(x_i)\ud\theta}{p^{1-\alpha}(x_i)}\right\}, e_l(x_i)\right\rangle\right)_{l=1}^d,\nonumber
\end{eqnarray}
and
\begin{equation}
\lim_{n\to \infty}\frac{1}{\epsilon}\left[\frac{\sum_{j=1}^n K_{\epsilon,1}\left(x_i,x_j\right)O_{ij}\bar{X}_j}{\sum_{j=1}^n K_{\epsilon,1}\left(x_i,x_j\right)} - \bar{X}_i\right]=\frac{m_2}{2dm_0}\left(\langle \iota_*\nabla^2 X(x_i),e_l(x_i)\rangle\right)_{l=1}^d.
\end{equation}
When $\epsilon_{\text{PCA}}=O(n^{-\frac{2}{d+1}})$ then the same almost surely convergence results above hold but with a slower convergence rate.   
\end{thm}

When the manifold is compact with boundary, (\ref{summary:eq2}) does not hold at the boundary. However, we have the following result for the convergence behavior near the boundary:
\begin{thm}\label{connlapbdry}
Let $\iota:\mathcal{M}\hookrightarrow\RR^p$ be a smooth d-dim compact Riemannian manifold with smooth boundary $\partial\MM$ embedded in $\RR^p$, with metric $g$ induced from the canonical metric on $\RR^p$. Let $\{x_i\}_{i=1,...,n}$, $p(x)$, $K_{\epsilon,1}\left(x_i,x_j\right)$, $p_\epsilon(x_i)$, $\{e_l(x_i)\}_{l=1,...,d}$ and $\bar{X}_i$ be defined in the same way as in Theorem \ref{summary}. Choose $\epsilon_{\text{PCA}}=O(n^{-\frac{2}{d+1}})$. Denote $\MM_{\sqrt{\epsilon}}=\{x\in\MM:~\min_{y\in\partial\MM}d(x,y)\leq \sqrt{\epsilon}\}$, where $d(x,y)$ is the geodesic distance between $x$ and $y$. When $x_i\in\MM_{\sqrt{\epsilon}}$, we have
\begin{eqnarray}
\frac{\sum_{j=1}^n K_{\epsilon,1}\left(x_i,x_j\right)O_{ij}\bar{X}_j}{\sum_{j=1}^n K_{\epsilon,1}\left(x_i,x_j\right)}&=&\left(\left\langle \iota_*P_{x_i,x_0}\left(X(x_0)+\frac{m^\epsilon_1}{m^\epsilon_0}\nabla_{\partial_d}X(x_0)\right),e_l(x_i)\right\rangle\right)_{l=1}^d \nonumber \\
&&+O\left(\epsilon + n^{-\frac{3}{2(d+1)}} + n^{-1/2}\epsilon^{-(d/4 - 1/2)}\right),\label{boundary-error}
\end{eqnarray}
where $x_0=\argmin_{y\in\partial\MM}d(x_i,y)$, $P_{x_i,x_0}$ is the parallel transport from $x_0$ to $x_i$ along the geodesic linking them, $m^\epsilon_1$ and $m^\epsilon_0$ are constants defined in (\ref{meps0}) and (\ref{meps1}), and $\partial_d$ is the normal direction to the boundary at $x_0$.
\end{thm}

For the choice $\epsilon=O(n^{-\frac{2}{d+4}})$ (as in Theorem \ref{summary}), the error appearing in (\ref{boundary-error}) is $O(\epsilon^{3/4})$ which is asymptotically smaller than $O(\sqrt{\epsilon})$, which is the order of $\frac{m^\epsilon_1}{m^\epsilon_0}$.
A consequence of Theorem \ref{summary}, Theorem \ref{connlapbdry} and the above discussion about the error terms is that the eigenvectors of $D_1^{-1}S_1-I$ are discrete approximations of the eigen-vector-fields of the connection-Laplacian operator with homogeneous Neumann boundary condition that satisfy
\begin{equation}
\left\{\begin{array}{ll}
\nabla^2 X(x) = -\lambda X(x), & \mbox{for } x \in \mathcal{M}, \\
\nabla_{\partial_d}X(x) = 0, & \mbox{for } x \in \partial \mathcal{M}.
\end{array}\right.
\end{equation}
We remark that the Neumann boundary condition also emerges for the choice $\epsilon_{\text{PCA}}=O(n^{-\frac{2}{d+2}})$. This is due to the fact that the error in the local PCA term is $O(\epsilon_{\text{PCA}}^{1/2}) = O(n^{-\frac{1}{d+2}})$, which is asymptotically smaller than $O(\epsilon^{1/2}) = O(n^{-\frac{1}{d+4}})$ error term.

Finally, Theorem \ref{summaryheatkernel} details the way in which the algorithm approximates the continuous heat kernel of the connection-Laplacian:
\begin{thm}\label{summaryheatkernel}
Let $\iota:\mathcal{M}\hookrightarrow\RR^p$ be a smooth d-dim compact Riemannian manifold embedded in $\RR^p$, with metric $g$ induced from the canonical metric on $\RR^p$ and $\nabla^2$ be the connection Laplacian. For $0\leq \alpha\leq 1$, define
\[
T_{\epsilon,\alpha} X(x)=\frac{\int_\MM K_{\epsilon,\alpha}(x,y) P_{x,y}X(y)\ud V(y)}{\int_\MM K_{\epsilon,\alpha}(x,y)\ud V(y)},
\]
where $P_{x,y} : T_y\MM \to T_x\MM$ is the parallel transport operator from $y$ to $x$ along the geodesic connecting them.

Then, for any $t>0$, the heat kernel $e^{t\nabla^2}$ can be approximated on $L^2(T\MM)$ (the space of squared-integrable vector fields) by $T^{\frac{t}{\epsilon}}_{\epsilon,1}$, that is,
\[
\lim_{\epsilon\rightarrow 0}T^{\frac{t}{\epsilon}}_{\epsilon,1}=e^{t\nabla^2},
\]
in the $L^2$ sense.
\end{thm}

%% file: numerical.tex
In all numerical experiments reported in this Section, we use the normalized vector diffusion mapping $V'_{1,t}$ corresponding to $\alpha=1$ in (\ref{Salpha}) and (\ref{Dalpha}), that is, we use the eigenvectors of $D_1^{-1}S_1$ to define the VDM. In all experiments we used the kernel function $K(u)=e^{-5u^2}\chi_{[0,1]}$ for the local PCA step as well as for the definition of the weights $w_{ij}$. The specific choices for $\epsilon$ and $\epsilon_{\text{PCA}}$ are detailed below. We remark that the results are not very sensitive to these choices, that is, similar results are obtained for a wide regime of parameters.

The purpose of the first experiment is to numerically verify Theorem \ref{summary} using spheres of different dimensions. Specifically, we sampled $n=8000$ points uniformly from $S^d$ embedded in $\mathbb{R}^{d+1}$ for $d=2,3,4,5$. Figure \ref{fig1} shows bar plots of the largest 30 eigenvalues of the matrix $D_1^{-1}S_1$ for $\epsilon_{\text{PCA}}=0.1$ when $d=2,3,4$ and $\epsilon_{\text{PCA}}=0.2$ when $d=5$, and $\epsilon=\epsilon_{\text{PCA}}^{\frac{d+1}{d+4}}$. It is noticeable that the eigenvalues have numerical multiplicities greater than 1. Since the connection-Laplacian commutes with rotations, the dimensions of its eigenspaces can be calculated using representation theory (see Appendix \ref{app-A}). In particular, our calculation predicted the following dimensions for the eigenspaces of the largest eigenvalues:
\begin{eqnarray*}
S^2: 6, 10, 14, \ldots. \quad S^3: 4, 6, 9, 16, 16, \ldots.\quad S^4: 5, 10, 14, \ldots. \quad S^5: 6, 15, 20, \ldots.
\end{eqnarray*}
These dimensions are in full agreement with the bar plots shown in Figure \ref{fig1}.

\begin{figure}[h]
\begin{center}
\subfigure[$S^2$]{
\includegraphics[width=0.22\textwidth]{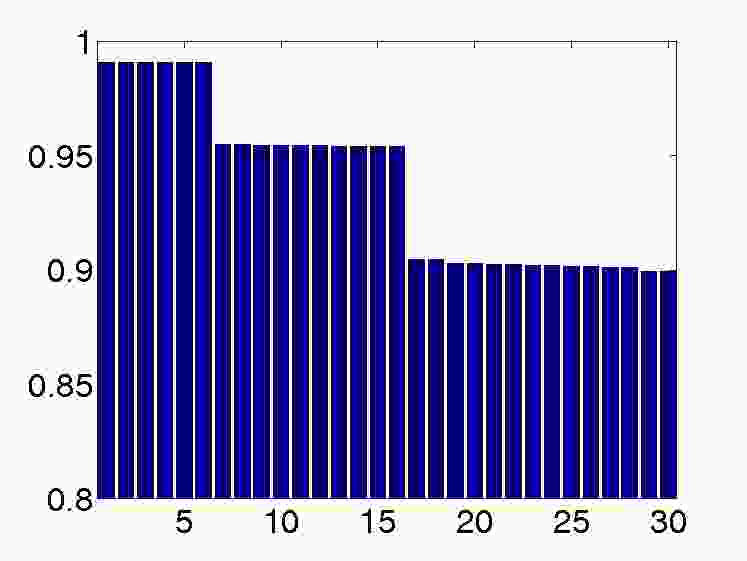}
}\label{fig1:spec}
\subfigure[$S^3$]{
\includegraphics[width=0.22\textwidth]{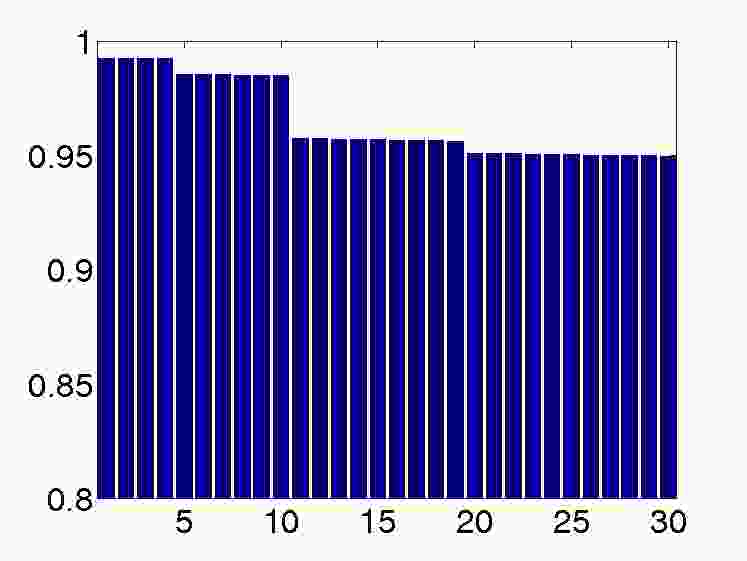}
}\label{fig2:spec}
\subfigure[$S^4$]{
\includegraphics[width=0.22\textwidth]{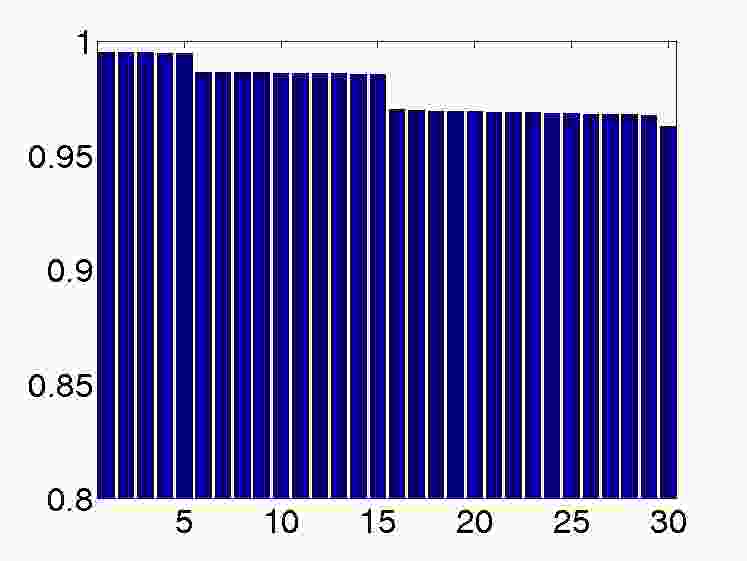}
}
\subfigure[$S^5$]{
\includegraphics[width=0.22\textwidth]{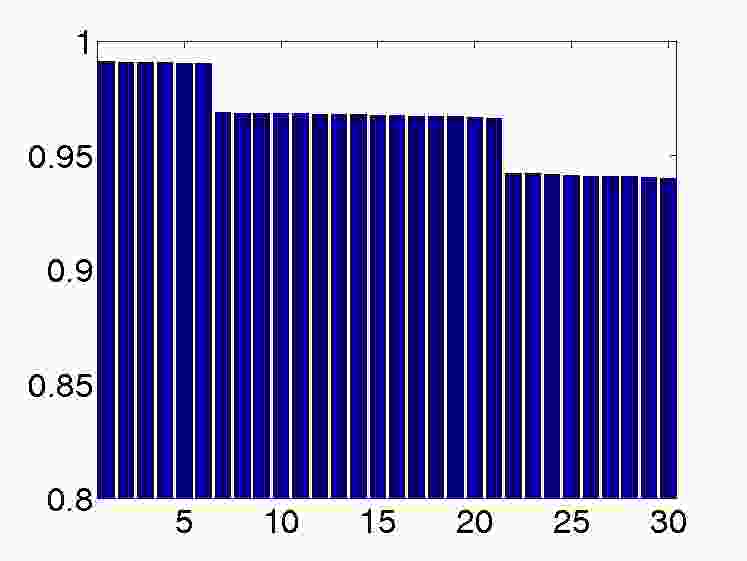}
}
\end{center}
\caption{Bar plots of the largest 30 eigenvalues of $D_1^{-1}S_1$ for $n=8000$ points uniformly distributed over spheres of different dimensions.}
\label{fig1}
\end{figure}

In the second set of experiments, we numerically compare the vector diffusion distance, the diffusion distance, and the geodesic distance for different compact manifolds with and without boundaries. The comparison is performed for the following four manifolds: 1) the sphere $S^2$ embedded in $\mathbb{R}^3$; 2) the torus $T^2$ embedded in $\mathbb{R}^3$; 3) the interval $[-\pi,\pi]$ in $\mathbb{R}$; and 4) the square $[0,2\pi]\times [0,2\pi]$ in $\mathbb{R}^2$. For both VDM and DM we truncate the mappings using $\delta=0.2$, see (\ref{Vtdelta}). The geodesic distance is computed by the algorithm of Dijkstra on a weighted graph, whose vertices correspond to the data points, the edges link data points whose Euclidean distance is less than $\sqrt{\epsilon}$,
and the weights $w_G(i,j)$ are the Euclidean distances, that is,
\begin{equation*}
w_G(i,j) = \left\{\begin{array}{ccc}
                  \|x_i-x_j\|_{\mathbb{R}^p} &  &  \|x_i-x_j\| < \sqrt{\epsilon},\\
                  +\infty &  & \mbox{otherwise}.
                \end{array}
 \right.
\end{equation*}

$S^2$ case: we sampled $n=5000$ points uniformly from $S^2=\{x\in\RR^3:\|x\|=1\}\subset\mathbb{R}^{3}$ and set $\epsilon_{\text{PCA}}=0.1$ and $\epsilon=\sqrt{\epsilon_{\text{PCA}}}\approx0.316$. For the truncated vector diffusion distance, when $t=10$, we find that the number of eigenvectors whose eigenvalue is larger (in magnitude) than $\lambda_1^t\delta$ is $m_{\text{VDM}}=m_{\text{VDM}}(t=10,\delta=0.2)=16$ (recall the definition of $m(t,\delta)$ that appears after (\ref{Vtdelta})). The corresponding embedded dimension is $m_{\text{VDM}}(m_{\text{VDM}}+1)/2$, which in this case is $16\cdot 17/2=136$. Similarly, for $t=100$, $m_{\text{VDM}}=6$ (embedded dimension is $6\cdot 7/2 = 21$), and when $t=1000$, $m_{\text{VDM}}=6$ (embedded dimension is again $21$). Although the first eigenspace (corresponding the largest eigenvalue) of the connection-Laplacian over $S^2$ is of dimension 6, there are small discrepancies between the top 6 numerically computed eigenvalues, due to the finite sampling.  This numerical discrepancy is amplified upon raising the eigenvalues to the $t$'th power, when $t$ is large, e.g., $t=1000$. For demonstration purposes, we remedy this numerical effect by artificially setting $\lambda_l=\lambda_1$ for $l=2,...,6$. For the truncated diffusion distance, when $t=10$, $m_{\text{DM}}=36$ (embedded dimension is $36-1=35$), when $t=100$, $m_{\text{DM}}=4$ (embedded dimension is $3$), and when $t=1000$, $m_{\text{DM}}=4$ (embedded dimension is $3$). Similarly, we have the same numerical effect when $t=1000$, that is, $\mu_2$, $\mu_3$ and $\mu_4$ are close but not exactly the same, so we again set $\mu_l=\mu_2$ for  $l=3,4$. The results are shown in Figure \ref{fig2}.

\begin{figure}[h]
\begin{center}
\subfigure[$d_{\text{VDM}',t=10}$]{
\includegraphics[width=0.22\textwidth]{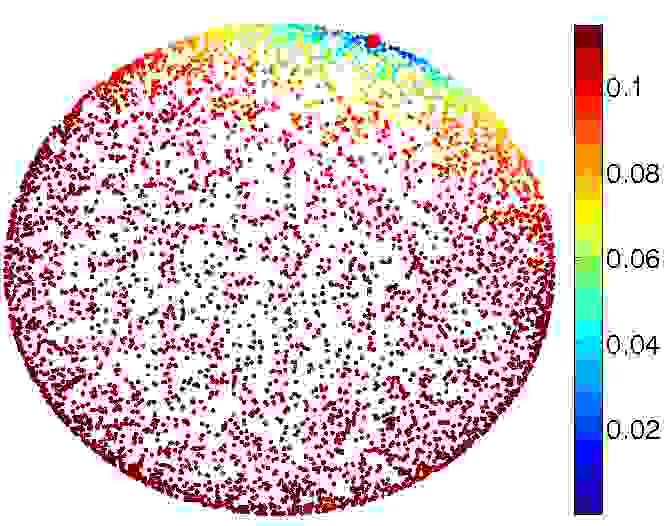}
}
\subfigure[$d_{\text{VDM}',t=100}$]{
\includegraphics[width=0.22\textwidth]{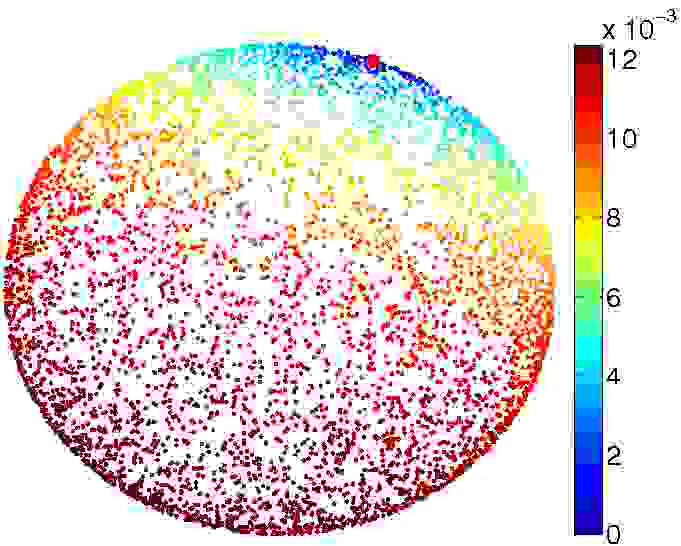}
}
\subfigure[$d_{\text{VDM}',t=1000}$]{
\includegraphics[width=0.22\textwidth]{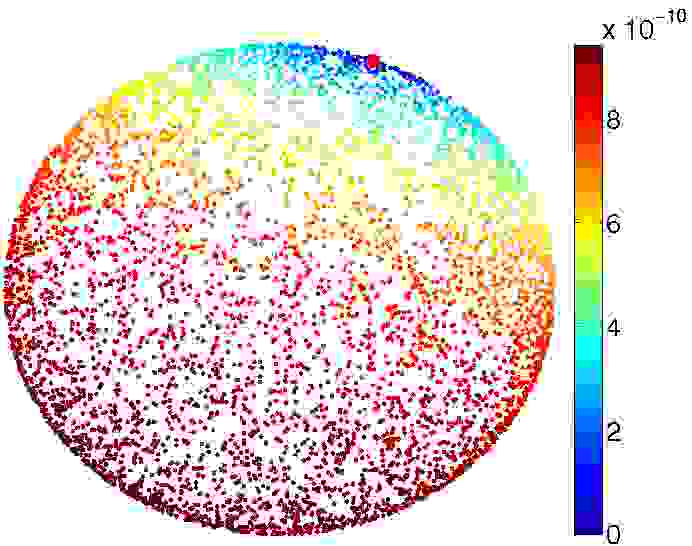}
}
\\
\subfigure[$d_{\text{DM},t=10}$]{
\includegraphics[width=0.22\textwidth]{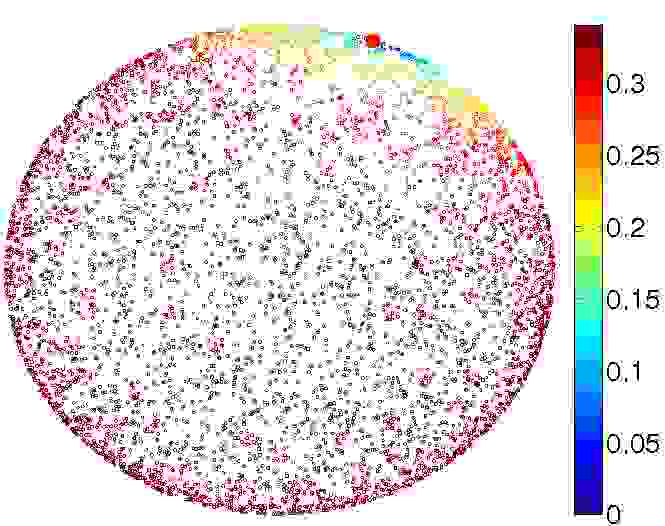}
}
\subfigure[$d_{\text{DM},t=100}$]{
\includegraphics[width=0.22\textwidth]{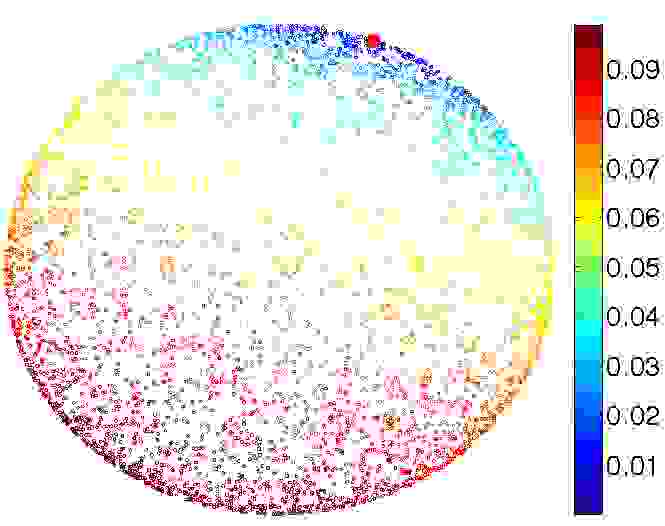}
}
\subfigure[$d_{\text{DM},t=1000}$]{
\includegraphics[width=0.22\textwidth]{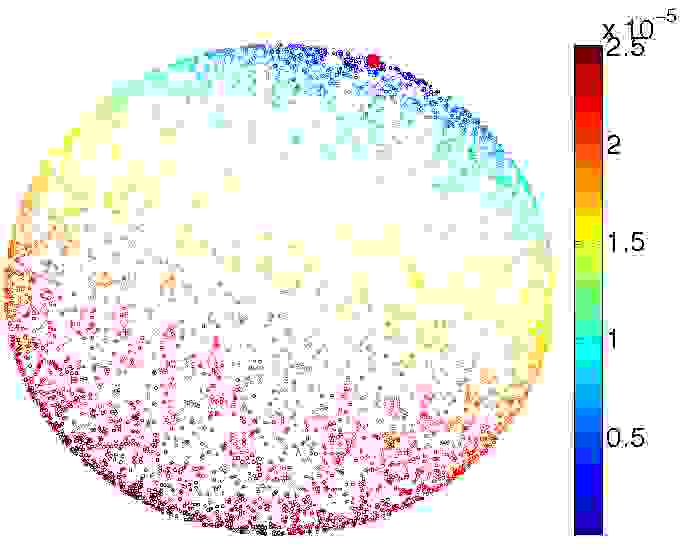}
}
\subfigure[Geodesic distance]{
\includegraphics[width=0.22\textwidth]{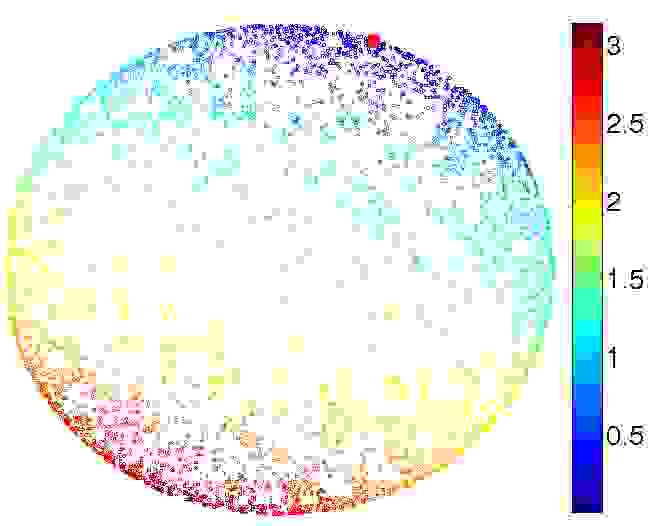}
}
\end{center}
\caption{$S^2$ case. Top: truncated vector diffusion distances for $t=10$, $t=100$ and $t=1000$; Bottom: truncated diffusion distances for $t=10$, $t=100$ and $t=1000$, and the geodesic distance. The reference point from which distances are computed is marked in red.}
\label{fig2}
\end{figure}

$T^2$ case: we sampled $n=5000$ points $(u,v)$ uniformly over the square $[0,2\pi)\times [0,2\pi)$ and then mapped them to $\mathbb{R}^3$ using the following transformation that defines the surface $T^2$ as
\[
T^2=\{((2+\cos(v))\cos(u),(2+\cos(v))\sin(u),\sin(v)):(u,v)\in[0,2\pi)\times[0,2\pi)\}\subset\mathbb{R}^{3}.\]
Notice that the resulting sample points are non-uniformly distributed over $T^2$. Therefore, the usage of $S_1$ and $D_1$ instead of $S$ and $D$ is important if we want the eigenvectors to approximate the eigen-vector-fields of the connection-Laplacian over $T^2$. We used $\epsilon_{\text{PCA}}=0.2$ and $\epsilon=\sqrt{\epsilon_{\text{PCA}}}\approx0.447$, and find that for the truncated vector diffusion distance, when $t=10$, the embedded dimension is 2628, when $t=100$, the embedded dimension is 36, and when $t=1000$, the embedded dimension is 3. For the truncated diffusion distance, when $t=10$, the embedded dimension is 130, when $t=100$, the embedded dimension is 14, and when $t=1000$, the embedded dimension is 2. The results are shown in Figure \ref{fig3}.

\begin{figure}[h]
\begin{center}
\subfigure[$d_{\text{VDM}',t=10}$]{
\includegraphics[width=0.22\textwidth]{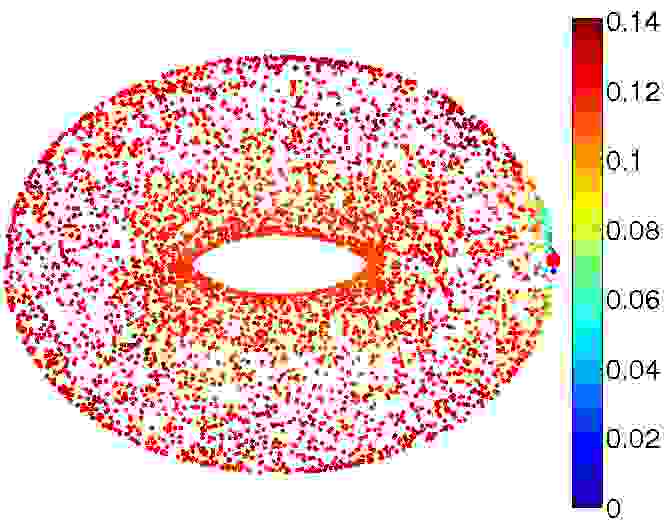}
}
\subfigure[$d_{\text{VDM}',t=100}$]{
\includegraphics[width=0.22\textwidth]{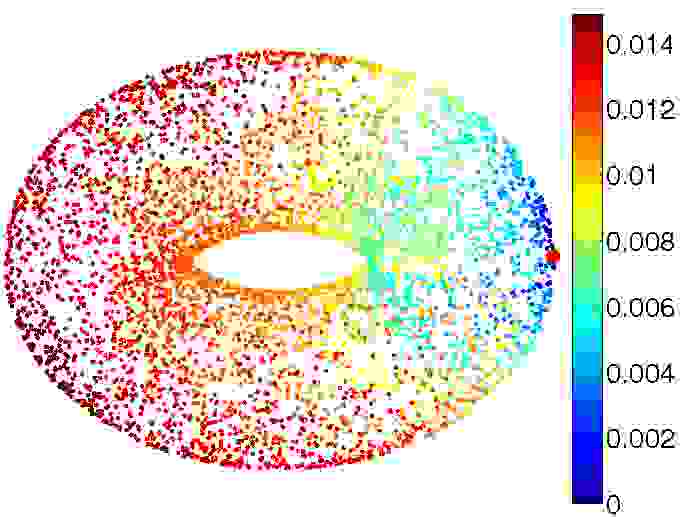}
}
\subfigure[$d_{\text{VDM}',t=1000}$]{
\includegraphics[width=0.22\textwidth]{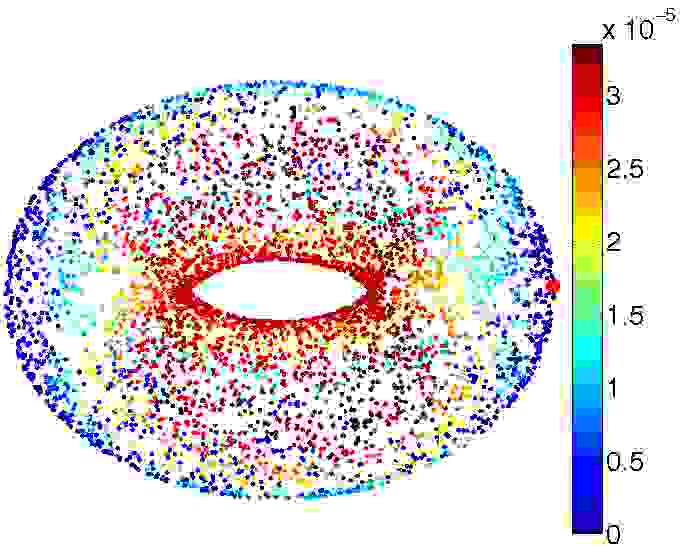}
}
\\
\subfigure[$d_{\text{DM},t=10}$]{
\includegraphics[width=0.22\textwidth]{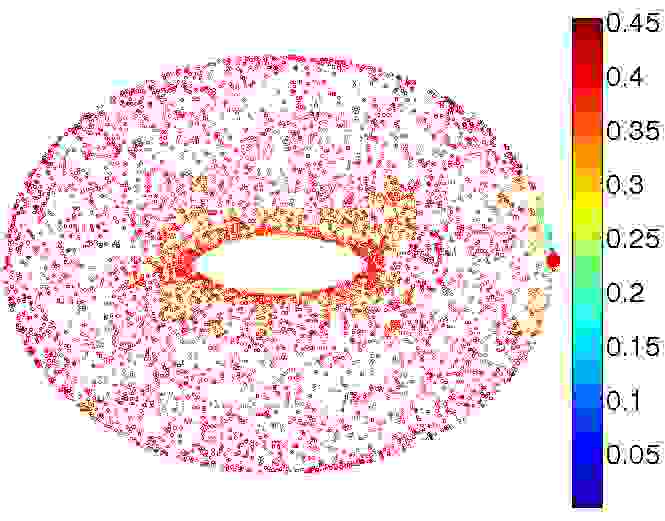}
}
\subfigure[$d_{\text{DM},t=100}$]{
\includegraphics[width=0.22\textwidth]{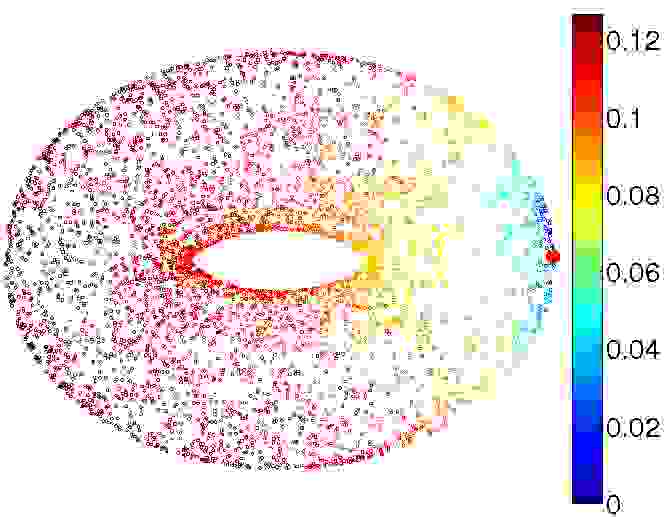}
}
\subfigure[$d_{\text{DM},t=1000}$]{
\includegraphics[width=0.22\textwidth]{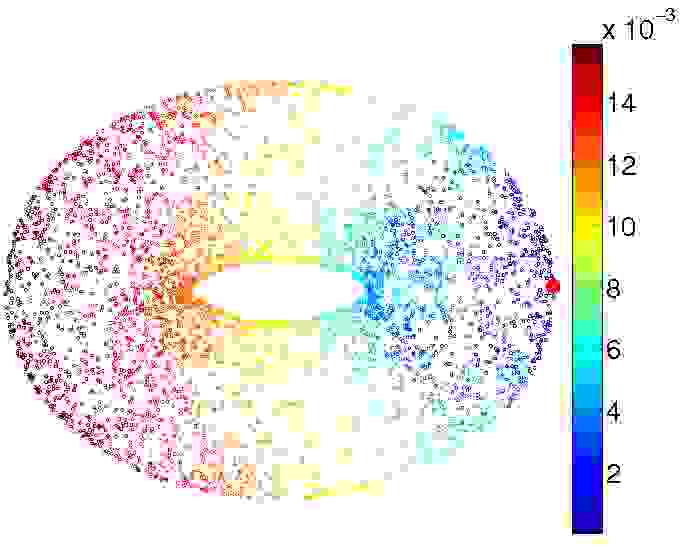}
}
\subfigure[Geodesic distance]{
\includegraphics[width=0.22\textwidth]{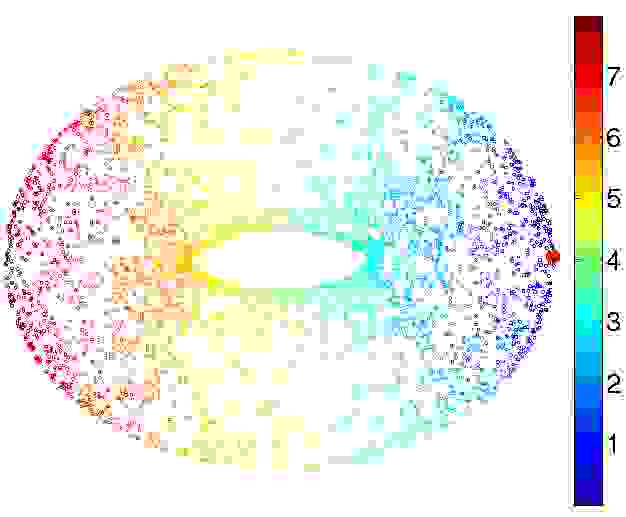}
}
\end{center}
\caption{$T^2$ case. Top: truncated vector diffusion distances for $t=10$, $t=100$ and $t=1000$; Bottom: truncated diffusion distances for $t=10$, $t=100$ and $t=1000$, and the geodesic distance. The reference point from which distances are computed is marked in red.}
\label{fig3}
\end{figure}

1-dim interval case: we sampled $n=5000$ equally spaced grid points from the interval $[-\pi,\pi]\subset\mathbb{R}^{1}$ and set $\epsilon_{\text{PCA}}=0.01$ and $\epsilon=\epsilon_{\text{PCA}}^{2/5}\approx 0.158$. For the truncated vector diffusion distance, when $t=10$, the embedded dimension is $120$, when $t=100$, the embedded dimension is $15$, and when $t=1000$, the embedded dimension is $3$. For the truncated diffusion distance, when $t=10$, the embedded dimension is $36$, when $t=100$, the embedded dimension is $11$, and when $t=1000$, the embedded dimension is $3$. The results are shown in Figure \ref{figint}.

\begin{figure}[h]
\begin{center}
\subfigure[$d_{\text{VDM}',t=10}$]{
\includegraphics[width=0.22\textwidth]{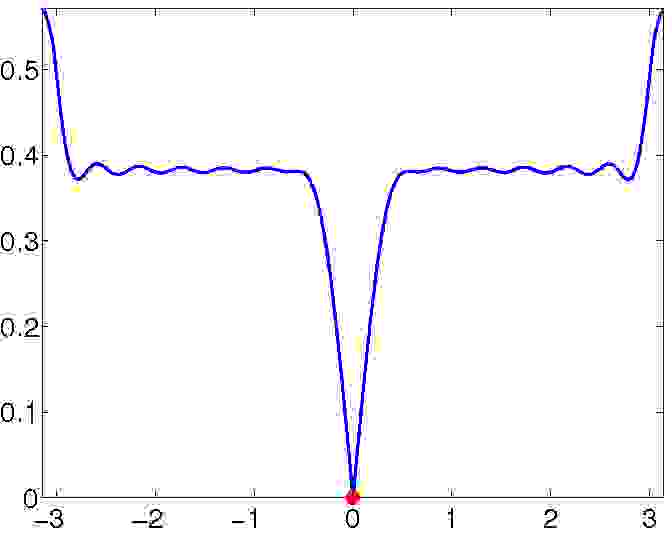}
}
\subfigure[$d_{\text{VDM}',t=100}$]{
\includegraphics[width=0.22\textwidth]{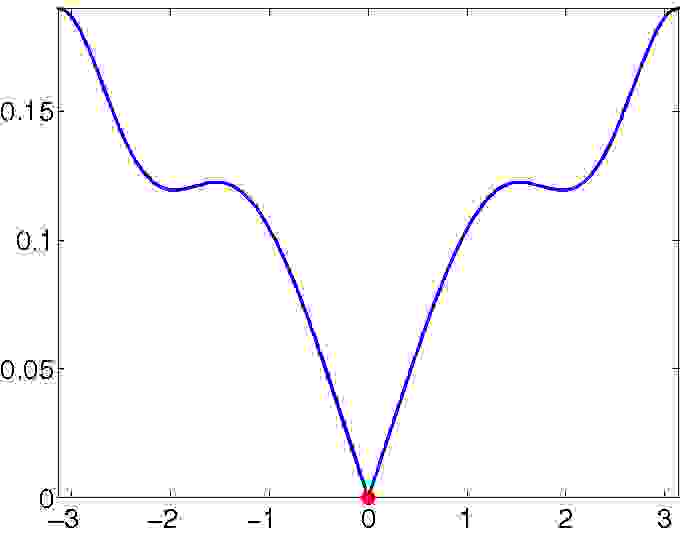}
}
\subfigure[$d_{\text{VDM}',t=1000}$]{
\includegraphics[width=0.22\textwidth]{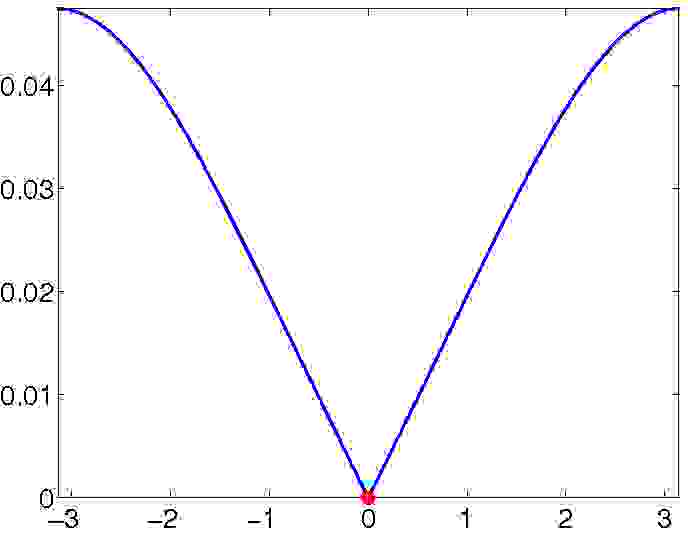}
}
\\
\subfigure[$d_{\text{DM},t=10}$]{
\includegraphics[width=0.22\textwidth]{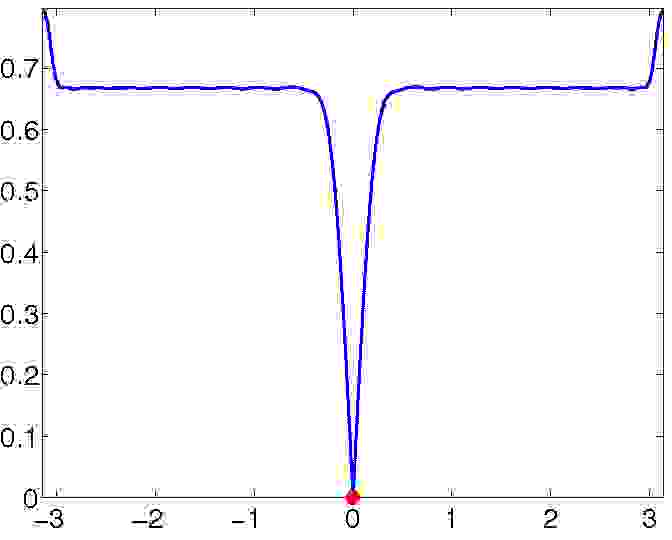}
}
\subfigure[$d_{\text{DM},t=100}$]{
\includegraphics[width=0.22\textwidth]{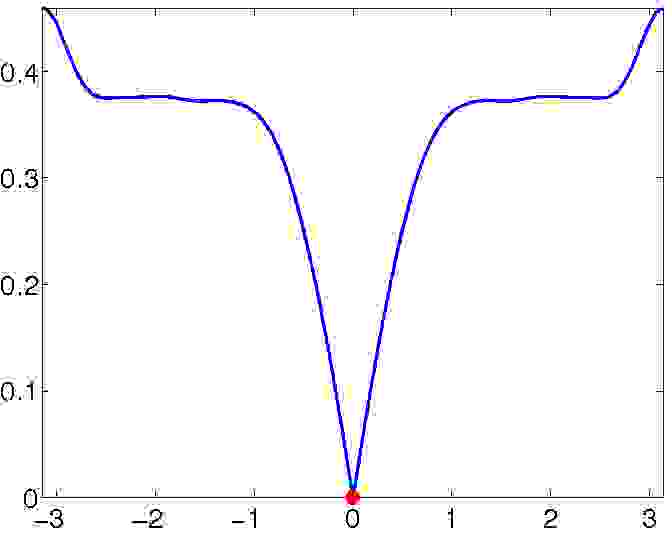}
}
\subfigure[$d_{\text{DM},t=1000}$]{
\includegraphics[width=0.22\textwidth]{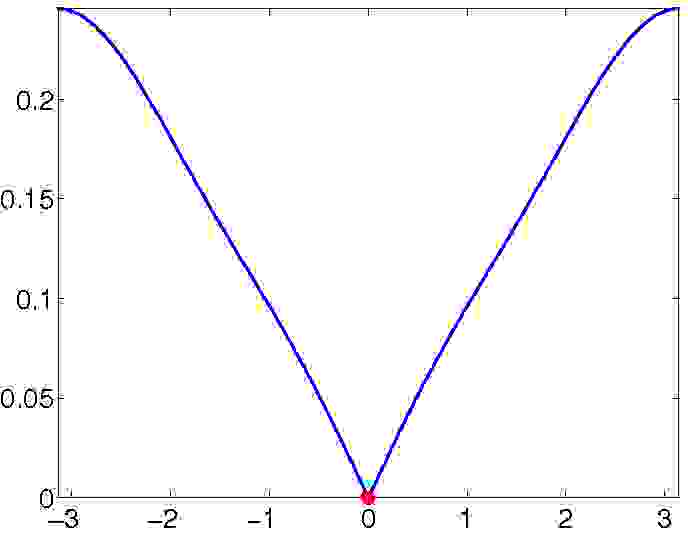}
}
\subfigure[Geodesic distance]{
\includegraphics[width=0.22\textwidth]{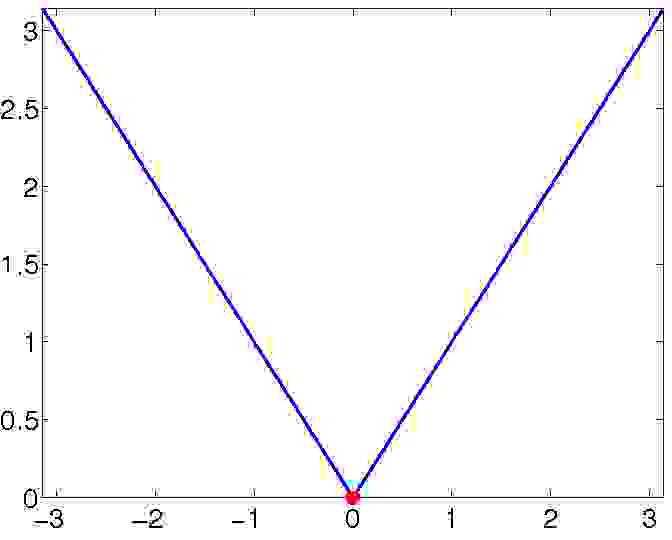}
}
\end{center}
\caption{1-dim interval case. Top: truncated vector diffusion distances for $t=10$, $t=100$ and $t=1000$; Bottom: truncated diffusion distances for $t=10$, $t=100$ and $t=1000$, and the geodesic distance. The reference point from which distances are computed is marked in red.}
\label{figint}
\end{figure}

Square case: we sampled $n=6561=81^2$ equally spaced grid points from the square $[0,2\pi]\times [0,2\pi]$ and fix $\epsilon_{\text{PCA}}=0.01$ and $\epsilon=\sqrt{\epsilon_{\text{PCA}}}=0.1$. For the truncated vector diffusion distance, when $t=10$, the embedded dimension is $20100$ (we only calculate the first 200 eigenvalues), when $t=100$, the embedded dimension is $1596$, and when $t=1000$, the embedded dimension is $36$. For the truncated diffusion distance, when $t=10$, the embedded dimension is $200$ (we only calculate the first 200 eigenvalues), when $t=100$, the embedded dimension is $200$, and when $t=1000$, the embedded dimension is $28$. The results are shown in Figure \ref{fig4}.

\begin{figure}[h]
\begin{center}
\subfigure[$d_{\text{VDM}',t=10}$]{
\includegraphics[width=0.22\textwidth]{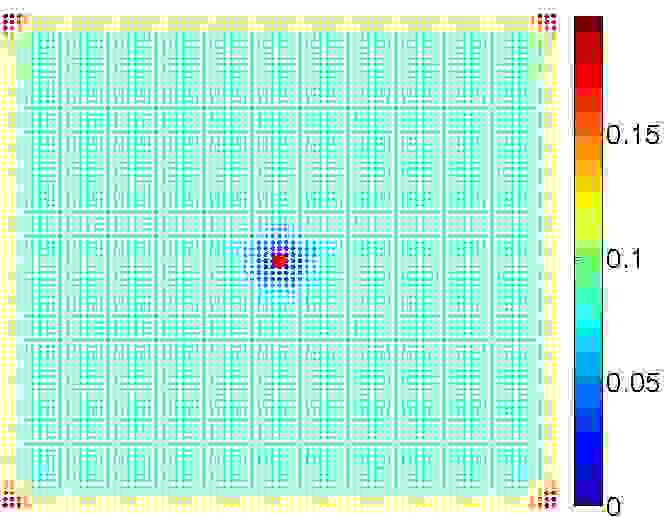}
}
\subfigure[$d_{\text{VDM}',t=100}$]{
\includegraphics[width=0.22\textwidth]{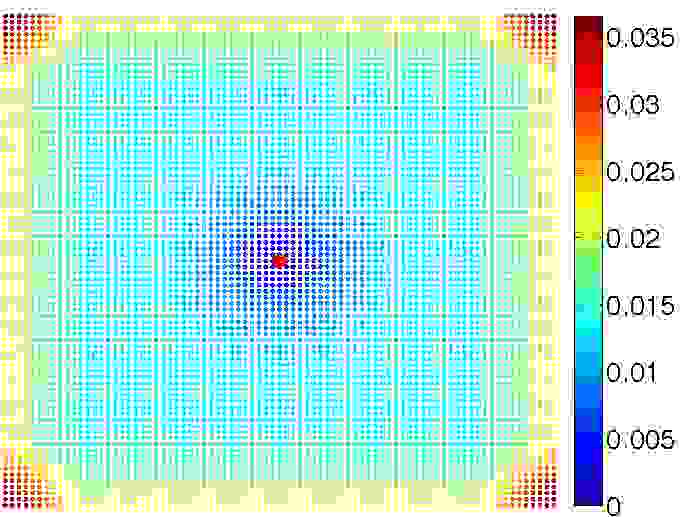}
}
\subfigure[$d_{\text{VDM}',t=1000}$]{
\includegraphics[width=0.22\textwidth]{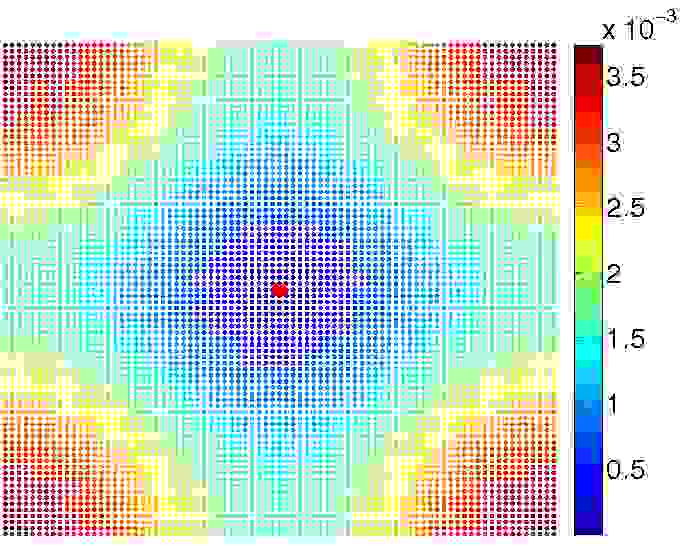}
}
\\
\subfigure[$d_{\text{DM},t=10}$]{
\includegraphics[width=0.22\textwidth]{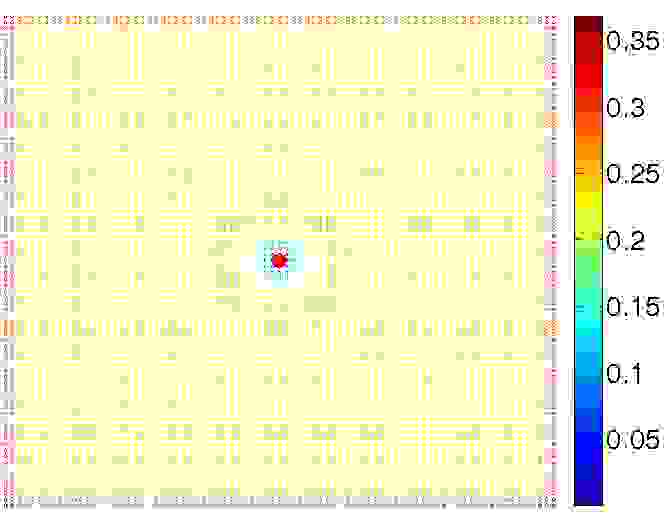}
}
\subfigure[$d_{\text{DM},t=100}$]{
\includegraphics[width=0.22\textwidth]{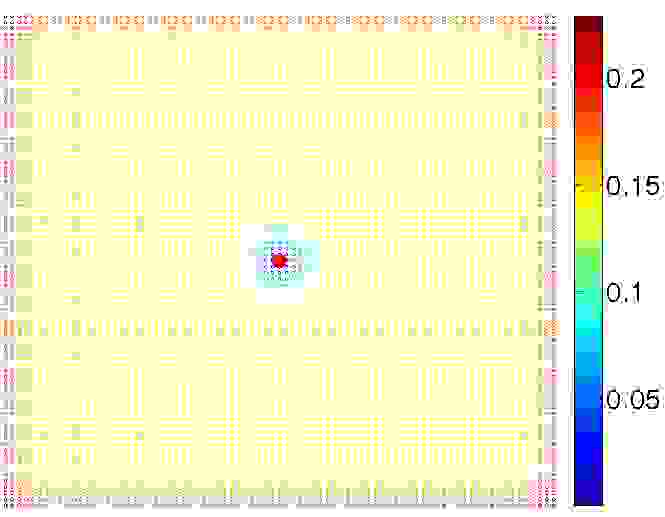}
}
\subfigure[$d_{\text{DM},t=1000}$]{
\includegraphics[width=0.22\textwidth]{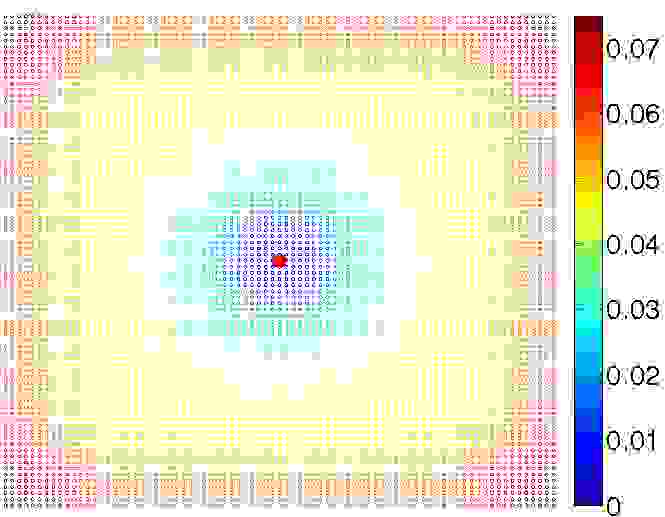}
}
\subfigure[Geodesic distance]{
\includegraphics[width=0.22\textwidth]{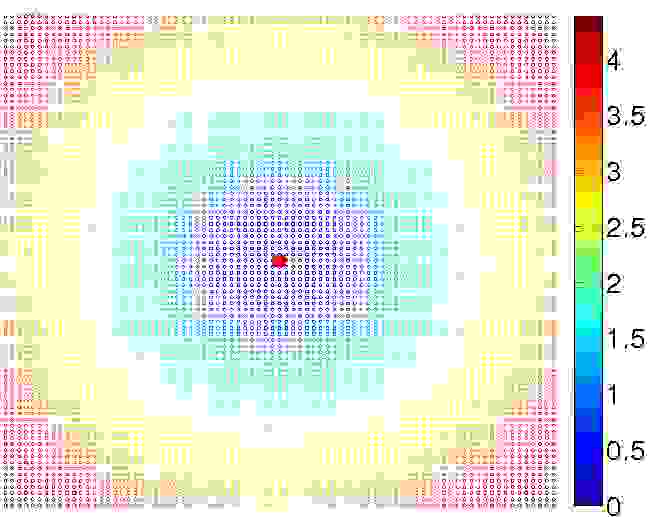}
}
\end{center}
\caption{Square case. Top: truncated vector diffusion distances for $t=10$, $t=100$ and $t=1000$; Bottom: truncated diffusion distances for $t=10$, $t=100$ and $t=1000$, and the geodesic distance. The reference point from which distances are computed is marked in red.}
\label{fig4}
\end{figure}

%% file: extrapolation.tex
Let $\mathcal{X}=\{x_i\}_{i=1}^n$ and $\mathcal{Y}=\{y_i\}_{i=1}^m$ so that $\mathcal{X},\mathcal{Y}\subset \MM^d$, where $\MM$ is embedded in $\RR^p$ by $\iota$. Suppose $X$ is a smooth vector field that we observe only on $\mathcal{X}$ and want to extend to $\mathcal{Y}$. That is, we observe the vectors $\iota_*X(x_1),\ldots, \iota_*X(x_n)\in \mathbb{R}^p$ and want to estimate $\iota_*X(y_1),\ldots, \iota_*X(y_m)$. The set $\mathcal{X}$ is assumed to be fixed, while the points in $\mathcal{Y}$ may arrive on-the-fly and need to be processed in real time. We propose the following Nystr\"om scheme for extending $X$ from $\mathcal{X}$ to $\mathcal{Y}$.

In the preprocessing step we use the points $x_1,\ldots,x_n$ for local PCA, alignment and vector diffusion mapping as described in Sections \ref{sec:manifold} and \ref{sec:VDM}. That is, using local PCA, we find the $p\times d$ matrices $O_i$ $(i=1,\ldots,n)$, such that the columns of $O_i$ are an orthonormal basis for a subspace that approximates the embedded tangent plane $\iota_*T_{x_i}\MM$; using alignment we find the orthonormal $d\times d$ matrices $O_{ij}$ that approximate the parallel transport operator from $T_{x_j}\MM$ to $T_{x_i}\MM$; and using $w_{ij}$ and $O_{ij}$ we construct the matrices $S$ and $D$ and compute (a subset of) the eigenvectors $v_1,v_2,\ldots,v_{nd}$ and eigenvalues $\lambda_1,\ldots,\lambda_{nd}$ of $D^{-1}S$.

We project the embedded vector field $\iota_*X(x_i) \in \mathbb{R}^p$ into the $d$-dimensional subspace spanned by the columns of $O_i$, and define $X_i \in \mathbb{R}^d$ as
\begin{equation}
X_i = O_i^T \iota_*X(x_i).
\end{equation}
We represent the vector field $X$ on $\mathcal{X}$ by the vector $\vx$ of length $nd$, organized as $n$ vectors of length $d$, with $$\vx(i)=X_i,\quad \mbox{for } i=1,\ldots,n.$$
We use the orthonormal basis of eigen-vector-fields $v_1,\ldots,v_{nd}$ to decompose $\vx$ as
\begin{equation}
\vx=\sum_{l=1}^{nd} a_l v_l,
\end{equation}
where $a_l=\vx^T v_l$. This concludes the preprocessing computations.

Suppose $y\in \mathcal{Y}$ is a ``new" out-of-sample point. First, we perform the local PCA step to find a $p\times d$ matrix, denoted $O_y$, whose columns form an orthonormal basis to a $d$-dimensional subspace of $\mathbb{R}^p$ that approximates the embedded tangent plane $\iota_{*}T_{y}\MM$. The local PCA step uses only the neighbors of $y$ among the points in $\mathcal{X}$ (but not in $\mathcal{Y}$) inside a ball of radius $\sqrt{\epsilon_{\text{PCA}}}$ centered at $y$.

Next, we use the alignment process to compute the $d\times d$ orthonormal matrix $O_{y,i}$ between $x_i$ and $y$ by setting
\[
O_{y,i}=\argmin_{O\in O(d)}\|O_y^TO_i-O\|_{HS}.
\]
Notice that the eigen-vector-fields satisfy
$$v_l(i) = \frac{1}{\lambda_l} \frac{\sum_{j=1}^n K_\epsilon(\|x_i-x_j\|)O_{ij}v_l(j)}{\sum_{j=1}^n K_\epsilon(\|x_i-x_j\|)}.$$
We denote the extension of $v_l$ to the point $y$ by $\tilde{v}_l(y)$ and define it as
\begin{equation}
\label{ex-v}
\tilde{v}_l(y) = \frac{1}{\lambda_l} \frac{\sum_{j=1}^n K_\epsilon(\|y-x_j\|)O_{y,j}v_l(j)}{\sum_{j=1}^n K_\epsilon(\|y-x_j\|)}.
\end{equation}
To finish the extrapolation problem, we denote the extension of $\vx$ to $y$ by $\tilde{\vx}(y)$ and define it as
\begin{equation}
\tilde{\vx}(y)=\sum_{l=1}^{m(\delta)} a_l\tilde{v_l}(y),
\end{equation}
where $m(\delta) = \max_l |\lambda_l| > \delta$, and $\delta>0$ is some fixed parameter to ensure the numerical stability of the extension procedure (due to the division by $\lambda_l$ in (\ref{ex-v}), $\frac{1}{\delta}$ can be regarded as the condition number of the extension procedure).
The vector $\iota_* X(y) \in \mathbb{R}^p$ is estimated as
\begin{equation}
\iota_* X(y) = O_y \tilde{\vx}(y).
\end{equation} 

%% file: heatkernel.tex
As discussed earlier, in the limit $n\to \infty$ and $\epsilon\to 0$ considered in (\ref{eq:bias}), the normalized graph Laplacian converges to the Laplace-Beltrami operator, which is the generator of the heat kernel for functions (0-forms). Similarly, in the limit $n\to \infty$ considered in (\ref{summary:eq}), we get the connection Laplacian operator, which is the generator of a heat kernel for vector fields (or 1-forms). The connection Laplacian $\nabla^2$ is a self-adjoint, second order elliptic operator defined over the tangent bundle $T\MM$. It is well-known \cite{gilkey} that the spectrum of $\nabla^2$ is discrete inside $\RR^-$ and the only possible accumulation point is $-\infty$. We will denote the spectrum as $\{-\lambda_k\}_{k=0}^\infty$, where $0\leq\lambda_0\leq\lambda_1...$. From the classical elliptic theory, see for example \cite{gilkey}, we know that
$e^{t\nabla^2}$ has the kernel
\[
k_t(x,y)=\sum_{n=0}^\infty e^{-\lambda_n t}X_n(x)\otimes\overline{X_n(y)}.
\]
where $\nabla^2X_n=-\lambda_nX_n$. Also, the eigenvector-fields $X_n$ of $\nabla^2$ form an orthonormal basis of $L^2(T\MM)$. In the continuous setup, we define the vector diffusion distance between $x,y\in \mathcal{M}$ using $\|k_t(x,y)\|^2_{HS}$. An explicit calculation gives
\begin{eqnarray}
\label{KF}
\|k_t(x,y)\|_{HS}^2 &=& \operatorname{Tr}\left[k_t(x,y)k_t(x,y)^* \right] \nonumber \\
&=& \sum_{n,m=0}^\infty e^{-(\lambda_n+\lambda_m)t} \langle X_n(x), X_m(x) \rangle \overline{\langle X_n(y), X_m(y) \rangle}.
\end{eqnarray}
It is well known that the heat kernel $k_t(x,y)$ is smooth in $x$ and $y$ and analytic in $t$ \cite{gilkey}, so for $t>0$ we can define a family of vector diffusion mappings $V_t$, that map any $x\in\MM$ into the Hilbert space $\ell^2$ by:
\begin{equation}\label{Vt}
V_t:x \mapsto \left( e^{-(\lambda_n + \lambda_m)t/2} \langle X_n(x), X_m(x) \rangle \right)_{n,m=0}^\infty,
\end{equation}
which satisfies
\begin{equation}
\|k_t(x,y)\|_{HS}^2 = \langle V_t(x), V_t(y)\rangle_{\ell^2}.
\end{equation}
The vector diffusion distance $d_{\text{VDM},t}(x,y)$ between $x\in M$ and $y\in M$ is defined as
\begin{equation}
\label{dd}
d_{\text{VDM},t}(x,y) := \|V_t(x)-V_t(y)\|_{\ell^2},
\end{equation}
which is clearly a distance function over $\MM$. In practice, due to the decay of $e^{-(\lambda_n+\lambda_m)t}$, only pairs $(n,m)$ for which $\lambda_n+\lambda_m$ is not too large are needed to get a good approximation of this vector diffusion distance. Like in the discrete case, the dot products $\langle X_n(x), X_m(x) \rangle$ are invariant to the choice of basis for the tangent space at $x$.

We now study some properties of the vector diffusion map $V_t$ (\ref{Vt}).
First, we claim for all $t>0$, the vector diffusion mapping $V_t$ is an embedding of the compact Riemannian manifold $\MM$ into $\ell^2$.
\begin{thm}
Given a $d$-dim closed Riemannian manifold $(\MM,g)$ and an orthonormal basis $\{X_n\}_{n=0}^\infty$ of $L^2(T\MM)$ composed of the eigen-vector-fields of the connection-Laplacian $\nabla^2$, then for any $t>0$, the vector diffusion map $V_t$ is a diffeomorphic embedding of $\MM$ into $\ell^2$.
\end{thm}
\begin{proof}
We show that $V_t:\MM \rightarrow \ell^2$ is continuous in $x$ by noting that
\begin{align}\label{vddexpansion}
\begin{split}
\|V_t(x)-V_t(y)\|^2_{\ell^2} &= \sum_{n,m=0}^\infty e^{-(\lambda_n + \lambda_m)t} (\langle X_n(x), X_m(x) \rangle-\langle X_n(y), X_m(y) \rangle)^2\\
&=\operatorname{Tr}(k_t(x,x)k_t(x,x)^*)+\operatorname{Tr}(k_t(y,y)k_t(y,y)^*)-2\operatorname{Tr}(k_t(x,y)k_t(x,y)^*)\\
\end{split}
\end{align}
From the continuity of the kernel $k_t(x,y)$, it is clear that $\|V_t(x)-V_t(y)\|^2_{\ell^2}\rightarrow 0$ as $y\rightarrow x$. Since $\MM$ is compact, it follows that $V_t(\MM)$ is compact in $\ell^2$. Then we show that $V_t$ is one-to-one. Fix $x\neq y$ and a smooth vector field $X$ that satisfies $\langle X(x),X(x) \rangle \neq \langle X(y),X(y) \rangle$. Since the eigen-vector fields $\{X_n\}_{n=0}^\infty$ form a basis to $L^2(T\MM)$, we have
\[
X(z)=\sum_{n=0}^\infty c_n X_n(z),\quad \mbox{for all } z\in \MM,
\]
where $c_n = \displaystyle \int_\MM \langle X,X_n \rangle \ud V$. As a result,
$$\langle X(z),X(z) \rangle = \sum_{n,m=0}^\infty c_n c_m  \langle X_n(z),X_m(z) \rangle.$$
Since $\langle X(x),X(x) \rangle \neq \langle X(y),X(y) \rangle$, there exist $n,m\in\NN$ such that $\langle X_n(x),X_m(x) \rangle \neq \langle X_n(y),X_m(y)\rangle$, which shows that $V_t(x)\neq V_t(y)$, i.e., $V_t$ is one-to-one. From the fact that the map $V_t$ is continuous and one-to-one from $\MM$, which is compact, onto $V_t(\MM)$, we conclude that $V_t$ is an embedding.

\end{proof}


Next, we demonstrate the asymptotic behavior of the vector diffusion distance $d_{\text{VDM},t}(x,y)$ and the diffusion distance $d_{\text{DM},t}(x,y)$ when $t$ is small and $x$ is close to $y$. The following theorem shows that in this asymptotic limit both the vector diffusion distance and the diffusion distance behave like the geodesic distance.
\begin{thm}\label{thm:geod}
Let $(\MM,g)$ be a smooth $d$-dim closed Riemannian manifold. Suppose $x,y\in \MM$ so that $x=\exp_yv$, where $v\in T_y\MM$. For any $t>0$, when $\|v\|^2\ll t\ll 1$ we have the following asymptotic expansion of the vector diffusion distance:
\begin{align*}
\begin{split}
d^2_{\text{VDM},t}(x,y)= d(4\pi)^{-d}\frac{\|v\|^2}{t^{d+1}}+O(t^{-d})
\end{split}
\end{align*}
Similarly, when $\|v\|^2\ll t\ll 1$, we have the following asymptotic expansion of the diffusion distance:
\begin{align*}
\begin{split}
d^2_{\text{DM},t}(x,y)=(4\pi)^{-d/2}\frac{\|v\|^2}{2t^{d/2+1}}+O(t^{-d/2}).
\end{split}
\end{align*}
\end{thm}
\begin{proof}
Fix $y$ and a normal coordinate around $y$. Denote $j(x,y)=|\det (d_v\exp_y)|$, where $x=\exp_y(v)$, $v\in T_x\MM$. Suppose $\|v\|$ is small enough so that $x=\exp_y(v)$ is away from the cut locus of $y$. It is well known that the heat kernel $k_t(x,y)$ for the connection Laplacian $\nabla^2$ over the vector bundle $\mathcal{E}$ possesses the following asymptotic expansion when $x$ and $y$ are close: \cite[p. 84]{getzler} or \cite{dewitt}
\begin{equation}\label{kt}
\|\partial^k_t(k_t(x,y)-k^N_t(x,y))\|_l=O(t^{N-d/2-l/2-k}),
\end{equation}
where $\|\cdot\|_l$ is the $C^l$ norm,
\begin{equation}\label{kt1}
k^N_t(x,y):=(4\pi t)^{-d/2}e^{-\|v\|^2/4t}j(x,y)^{-1/2}\sum^N_{i=0}t^i\Phi_i(x,y),
\end{equation}
$N>d/2$, and $\Phi_i$ is a smooth section of the vector bundle $\mathcal{E}\otimes \mathcal{E}^*$ over $\MM\times \MM$. Moreover, $\Phi_0(x,y)=P_{x,y}$ is the parallel transport from $\mathcal{E}_y$ to $\mathcal{E}_x$. In the VDM setup, we take $\mathcal{E}=T\MM$, the tangent bundle of $\MM$. Also, by \cite[Proposition 1.28]{getzler}, we have the following expansion:
\begin{equation}\label{kt2}
j(x,y)=1+\Ric(v,v)/6+O(\|v\|^3).
\end{equation}
Equations (\ref{kt1}) and (\ref{kt2}) lead to the following expansion under the assumption $\|v\|^2 \ll t$:
\begin{align*}
\begin{split}
&\quad\operatorname{Tr}(k_t(x,y)k_t(x,y)^*)\\
&=(4\pi t)^{-d}e^{-\|v\|^2/2t}(1+\Ric(v,v)/6+O(\|v\|^3))^{-1}\operatorname{Tr}((P_{x,y}+O(t))((P_{x,y}+O(t))^*)\\&=(4\pi t)^{-d}e^{-\|v\|^2/2t}(1-\Ric(v,v)/6+O(\|v\|^3))(d+O(t))\\
&=(d+O(t))(4\pi t)^{-d}\left(1-\frac{\|v\|^2}{2t}+O\left(\frac{\|v\|^4}{t^2}\right)\right).
\end{split}
\end{align*}
In particular, for $\|v\|=0$ we have
\begin{align*}
\begin{split}
\operatorname{Tr}(k_t(x,x)k_t(x,x)^*)=(d+O(t))(4\pi t)^{-d}.
\end{split}
\end{align*}
Thus, for $\|v\|^2\ll t\ll 1$, we have
\begin{eqnarray}
d^2_{\text{VDM},t}(x,y)&=& \operatorname{Tr}(k_t(x,x)k_t(x,x)^*) + \operatorname{Tr}(k_t(y,y)k_t(y,y)^*) - 2\operatorname{Tr}(k_t(x,y)k_t(x,y)^*)\nonumber\\
&=& d(4\pi)^{-d}\frac{\|v\|^2}{t^{d+1}}+O(t^{-d}).
\end{eqnarray}

By the same argument we can carry out the asymptotic expansion of the diffusion distance $d_{\text{DM},t}(x,y)$. Denote the eigenfunctions and eigenvalues of the Laplace-Beltrami operator $\Delta$ by $\phi_n$ and $\mu_n$. We can rewrite the diffusion distance as follows:
\begin{align}
\begin{split}
d^2_{\text{DM},t}(x,y) = \sum_{n=1}^\infty e^{-\mu_n t} (\phi_n(x)-\phi_n(y))^2
=\tilde{k}_t(x,x)+\tilde{k}_t(y,y)-2\tilde{k}_t(x,y),
\end{split}
\end{align}
where $\tilde{k}_t$ is the heat kernel of the Laplace-Beltrami operator. Note that the Laplace-Beltrami operator is equal to the connection-Laplacian operator defined over the trivial line bundle over $\MM$. As a result, equation (\ref{kt1}) also describes the asymptotic expansion of the heat kernel for the Laplace-Beltrami operator as
\begin{align*}
\begin{split}
\tilde{k}_t(x,y)=(4\pi t)^{-d/2}e^{-\|v\|^2/4t}(1+\Ric(v,v)/6+O(\|v\|^3))^{-1/2}(1+O(t)).
\end{split}
\end{align*}
Put these facts together, we obtain
\begin{align}
\begin{split}
d^2_{\text{DM},t}(x,y)=(4\pi)^{-d/2}\frac{\|v\|^2}{2t^{d/2+1}}+O(t^{-d/2}),
\end{split}
\end{align}
when $\|v\|^2\ll t\ll 1$.
\end{proof}

%% file: realapplication.tex
Besides being a general framework for data analysis and manifold learning, VDM is useful for performing robust multi-reference rotational alignment of objects, such as one-dimensional periodic signals, two-dimensional images and three-dimensional shapes. In this Section, we briefly describe the application of VDM to a particular multi-reference rotational alignment problem of two-dimensional images that arise in the field of cryo-electron microscopy (EM). A more comprehensive study of this problem can be found in \cite{amit20093} and \cite{hadani20092}. It can be regarded as a prototypical multi-reference alignment problem, and we expect many other multi-reference alignment problems that arise in areas such as computer vision and computer graphics to benefit from the proposed approach.

The goal in cryo-EM \cite{frank2006} is to determine 3D macromolecular structures from noisy projection images taken at unknown random orientations by an electron microscope, i.e., a random Computational Tomography (CT).
Determining 3D macromolecular structures for large biological molecules remains vitally important, as witnessed, for example, by the 2003 Chemistry Nobel Prize, co-awarded to R. MacKinnon for resolving the 3D structure of the Shaker $K^+$ channel protein, and by the 2009 Chemistry Nobel Prize, awarded to V. Ramakrishnan, T. Steitz and A. Yonath for studies of the structure and function of the ribosome. The standard procedure for structure determination of large molecules is X-ray crystallography. The challenge in this method is often more in the crystallization itself than in the interpretation of the X-ray results, since many large proteins have so far withstood all attempts to crystallize them.

In cryo-EM, an alternative to X-ray crystallography, the sample of macromolecules is rapidly frozen in an ice layer so thin that their tomographic projections are typically disjoint; this seems the most promising alternative for large molecules that defy crystallization. The cryo-EM imaging process produces a large collection of tomographic projections of the same molecule, corresponding to different and unknown projection orientations. The goal is to reconstruct the three-dimensional structure of the molecule from such unlabeled projection images, where data sets typically range from $10^4$ to $10^5$ projection images whose size is roughly $100\times 100$ pixels.
The intensity of the pixels in a given projection image is proportional to the line integrals of the electric potential induced by the molecule along the path of the imaging electrons (see Figure \ref{fig:sketch}). The highly intense electron beam destroys the frozen molecule and it is therefore impractical to take projection images of the same molecule at known different directions as in the case of classical CT. In other words, a single molecule can be imaged only once, rendering an extremely low signal-to-noise ratio (SNR) for the images (see Figure \ref{fig:images} for a sample of real microscope images), mostly due to shot noise induced by the maximal allowed electron dose (other sources of noise include the varying width of the ice layer and partial knowledge of the contrast function of the microscope). In the basic homogeneity setting considered hereafter, all imaged molecules are assumed to have the exact same structure; they differ only by their spatial rotation. Every image is a projection of the same molecule but at an unknown random three-dimensional rotation, and the cryo-EM problem is to find the three-dimensional structure of the molecule from a collection of noisy projection images.

\begin{figure}
\begin{center}
\includegraphics[width=0.4\textheight]{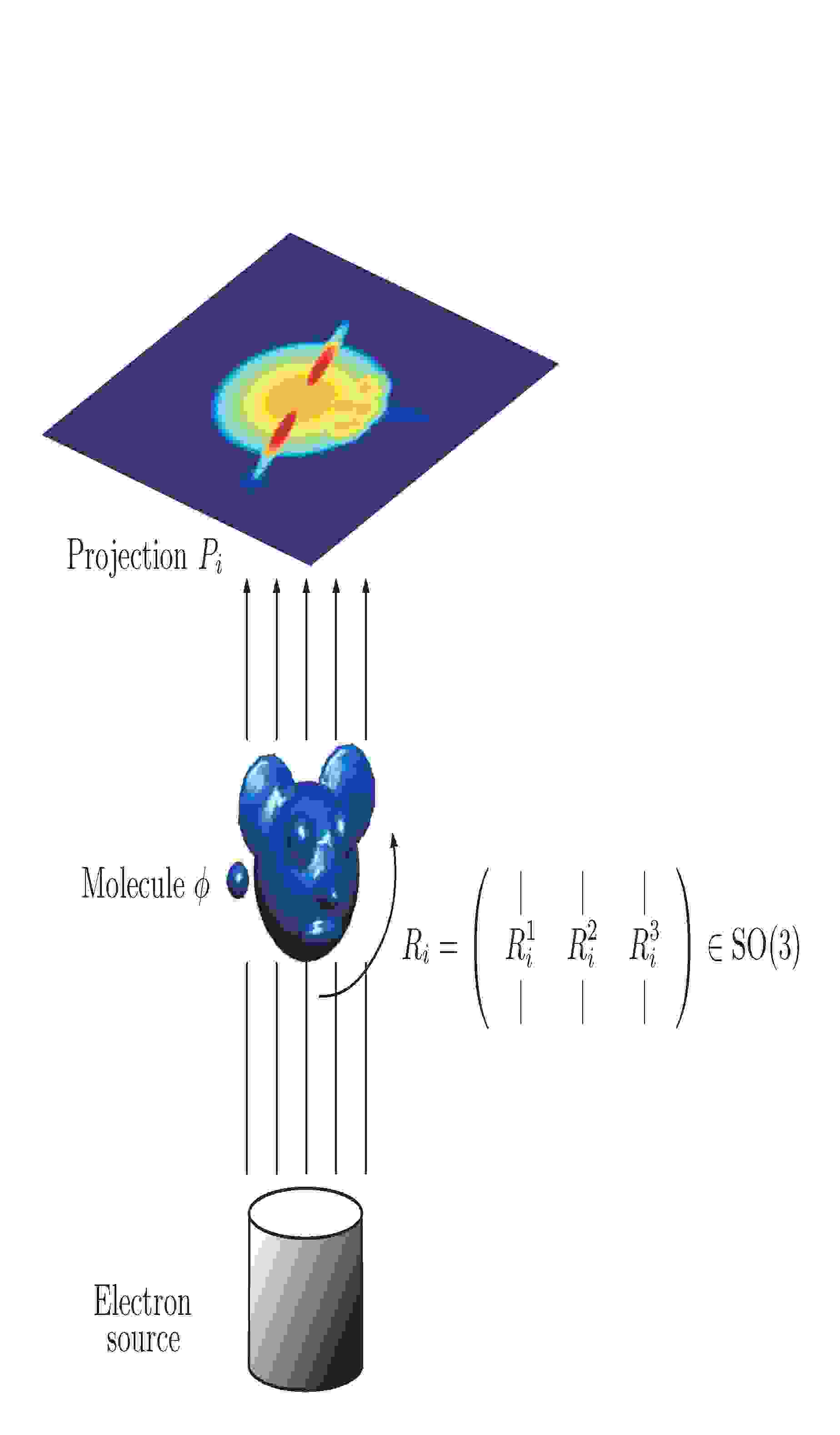}
\end{center}
\caption{Schematic drawing of the imaging process: every projection image corresponds to some unknown 3D rotation of the unknown molecule.}
\label{fig:sketch}
\end{figure}

\begin{figure}[h]
\begin{center}
\includegraphics[width=0.17\textwidth]{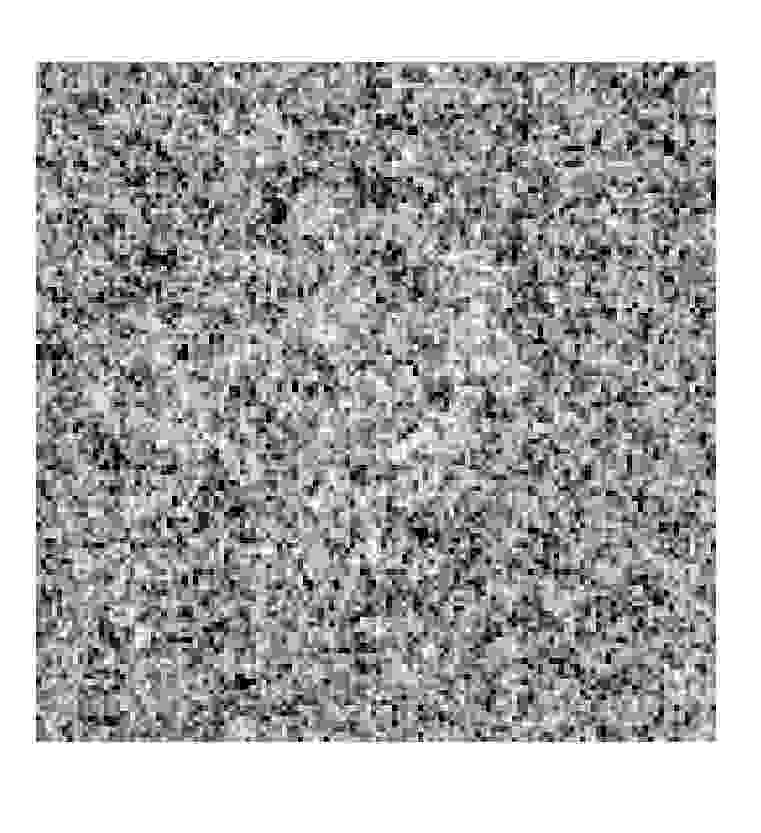}%
\hfill
\includegraphics[width=0.17\textwidth]{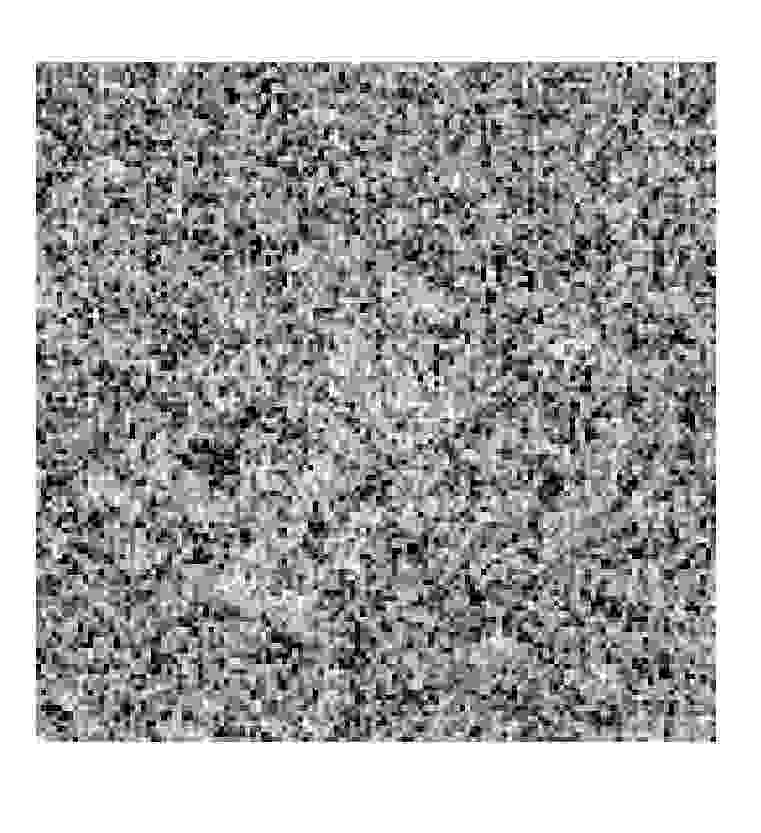}%
\hfill
\includegraphics[width=0.17\textwidth]{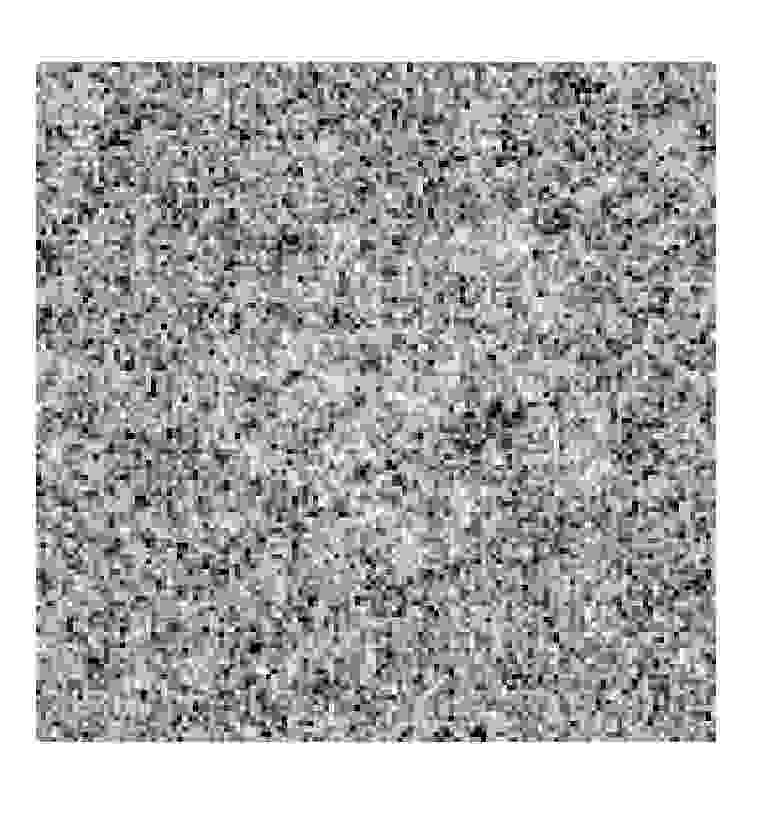}%
\hfill
\includegraphics[width=0.17\textwidth]{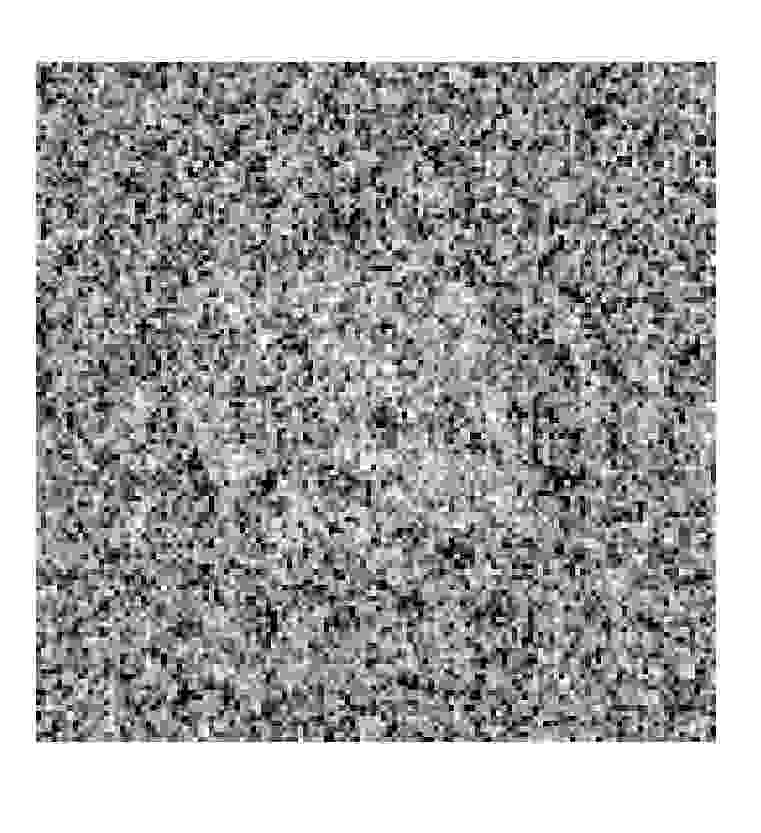}%
\end{center}
\caption{A collection of four real electron microscope images of the E.~coli 50S ribosomal subunit; courtesy of Dr. Fred Sigworth.}\label{fig:images}
\end{figure}

The rotation group SO(3) is the group of all orientation preserving orthogonal transformations about the origin of the three-dimensional Euclidean space $\mathbb{R}^3$ under the operation of composition.
Any 3D rotation can be expressed using a $3\times 3$ orthogonal matrix
$R = \left(\begin{array}{ccc}
                                  | & | & | \\
                                  R^1 & R^2 & R^3 \\
                                  | & | & |
                                \end{array}
 \right)$
satisfying $RR^T=R^TR=I_{3\times 3}$ and  $\det R = 1$.
The column vectors $R^1,R^2,R^3$ of $R$ form an orthonormal basis to $\mathbb{R}^3$.
To each projection image $P$ there corresponds a $3\times 3$ unknown rotation matrix $R$ describing its orientation (see Figure \ref{fig:sketch}). Excluding the contribution of noise, the intensity $P(x,y)$ of the pixel located at $(x,y)$ in the image plane
corresponds to the line integral of the electric potential induced by the molecule along the path of the imaging
electrons, that is,
\begin{equation}\label{eq:projection integral}
P(x,y) = \int_{-\infty}^\infty \phi(xR^1 + yR^2 + zR^3)\,dz
\end{equation}
where $\phi : \mathbb{R}^3\mapsto \mathbb{R}$ is the electric potential of the molecule in some fixed `laboratory'
coordinate system.
The projection operator~\eqref{eq:projection integral} is also known
as the X-ray transform \cite{Natr2001a}.

We therefore identify the third column $R^3$ of $R$ as the imaging direction, also known as the {\em viewing angle} of the molecule.
The first two columns $R^1$ and $R^2$ form an orthonormal basis for the plane in $\mathbb{R}^3$ perpendicular to the viewing angle $R^3$.
All clean projection images of the molecule that share the same viewing angle look the same up to some in-plane rotation. That is, if $R_i$ and $R_j$ are two rotations with the same viewing angle $R_i^3=R_j^3$ then $R_i^1, R_i^2$ and $R_j^1, R_j^2$ are two orthonormal bases for the same plane.
On the other hand, two rotations with opposite viewing angles $R_i^3=-R_j^3$ give rise to two projection images that are the same after reflection (mirroring) and some in-plane rotation.

As projection images in cryo-EM have extremely low SNR, a crucial initial step in all reconstruction methods is ``class averaging" \cite{frank2006,vanheel}. Class averaging is the grouping of a large data set of $n$ noisy raw projection images $P_1,\ldots,P_n$ into clusters, such that images within a single cluster have similar viewing angles (it is possible to artificially double the number of projection images by including all mirrored images). Averaging rotationally-aligned noisy images within each cluster results in ``class averages"; these are images that enjoy a higher SNR and are used in later cryo-EM procedures such as the angular reconstitution procedure \cite{Fred5} that requires better quality images. Finding consistent class averages is challenging due to the high level of noise in the raw images as well as the large size of the image data set. A sketch of the class averaging procedure is shown in Figure \ref{fig:class_averaging}.

\begin{figure}[h]
\begin{center}
\subfigure[Clean image]{%
\label{fig:clean_image}
\includegraphics[width=0.17\textwidth]{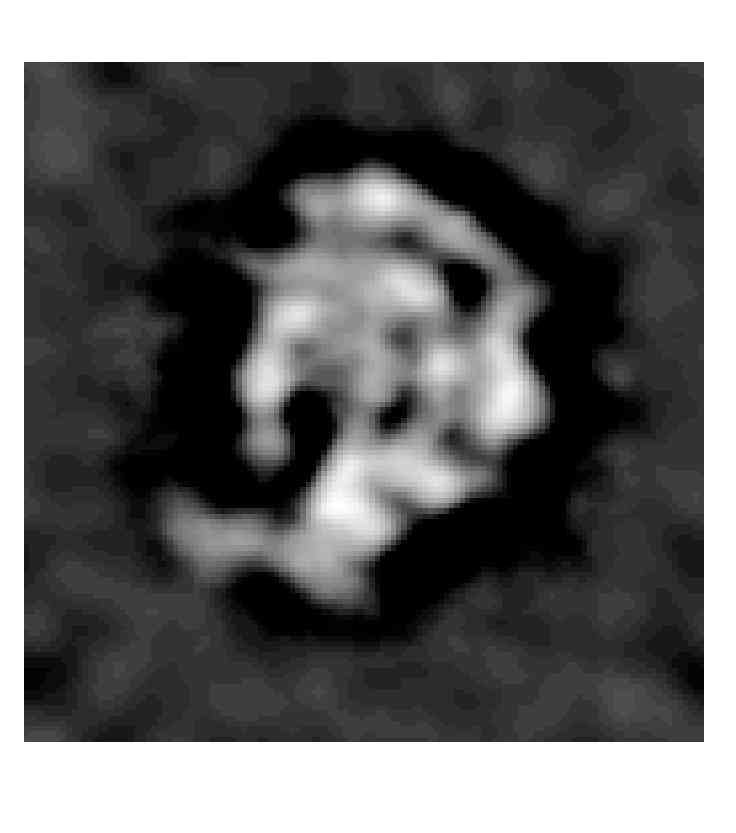}}
\hfill
\subfigure[$P_i$]{%
\label{fig:noisy_image}
\includegraphics[width=0.17\textwidth]{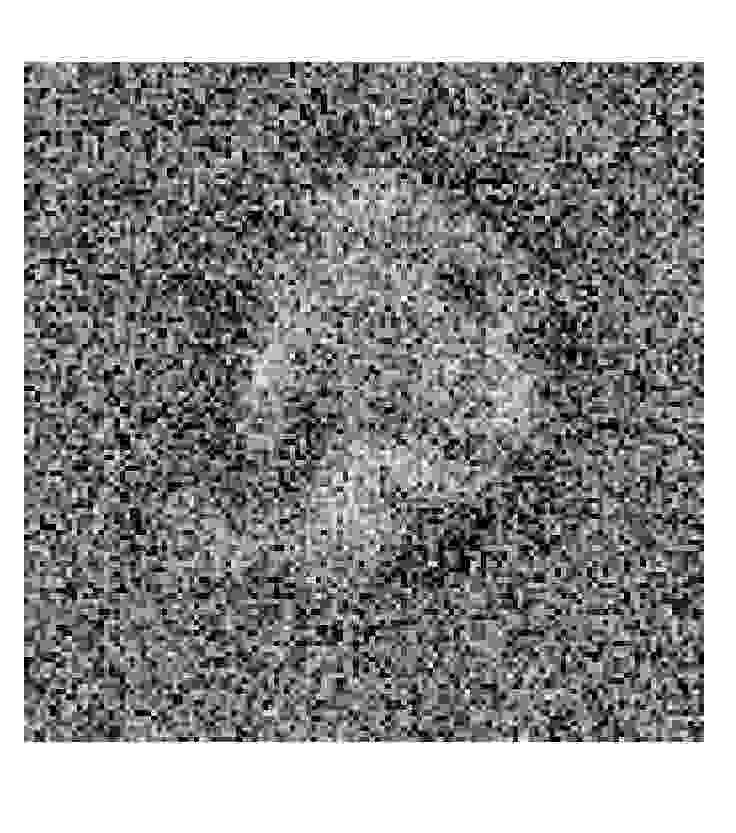}}
\hfill
\subfigure[$P_j$]{%
\label{fig:rot_image}
\includegraphics[width=0.17\textwidth]{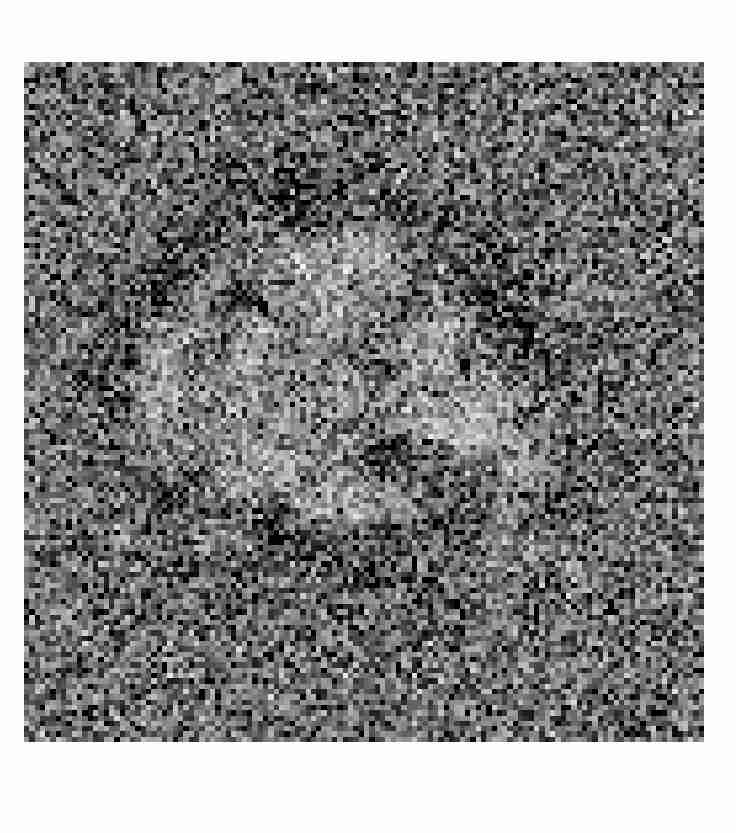}}
\hfill
\subfigure[Average]{%
\label{fig:class_average}
\includegraphics[width=0.17\textwidth]{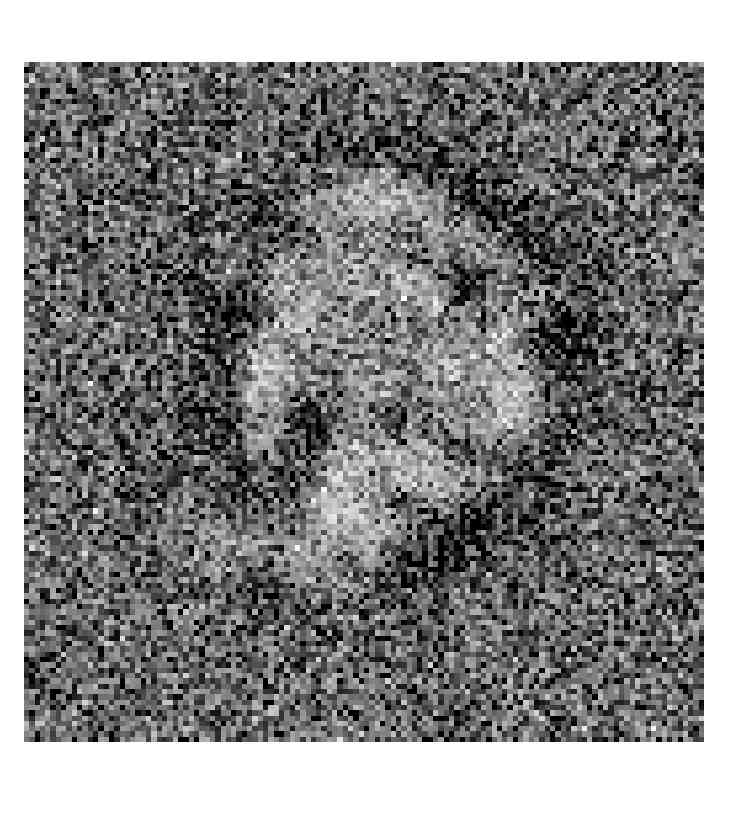}}
\end{center}
\caption{\subref{fig:clean_image} A clean simulated projection image of the ribosomal subunit generated from its known volume; \subref{fig:noisy_image} Noisy instance of \subref{fig:clean_image}, denoted $P_i$, obtained by the addition of white Gaussian noise. For the simulated images we chose the SNR to be higher than that of experimental images in order for image features to be clearly visible; \subref{fig:rot_image} Noisy projection, denoted $P_j$, taken at the same viewing angle but with a different in-plane rotation; \subref{fig:class_average} Averaging the noisy images \subref{fig:noisy_image} and \subref{fig:rot_image} after in-plane rotational alignment. The class average of the two images has a higher SNR than that of the noisy images \subref{fig:noisy_image} and \subref{fig:rot_image}, and it has better similarity with the clean image \subref{fig:clean_image}. }\label{fig:class_averaging}
\end{figure}

Penczek, Zhu and Frank \cite{Penczek1996} introduced the rotationally invariant K-means clustering procedure to identify images that have similar viewing angles. Their Rotationally Invariant Distance $d_{\text{RID}}(i,j)$ between image $P_i$ and image $P_j$ is defined as the Euclidean distance between the images when they are optimally aligned with respect to in-plane rotations (assuming the images are centered)
\begin{equation}
\label{dist}
d_{\text{RID}}(i,j) = \min_{\theta \in [0,2\pi)} \|P_i - R(\theta) P_j \|,
\end{equation}
where $R(\theta)$ is the rotation operator of an image by an angle $\theta$ in the counterclockwise direction. Prior to computing the invariant distances of (\ref{dist}), a common practice is to center all images by correlating them with their total average $\frac{1}{n}\sum_{i=1}^n P_i$, which is approximately radial (i.e., has little angular variation) due to the randomness in the rotations. The resulting centers usually miss the true centers by only a few pixels (as can be validated in simulations during the refinement procedure). Therefore, like \cite{Penczek1996}, we also choose to focus on the more challenging problem of rotational alignment by assuming that the images are properly centered, while the problem of translational alignment can be solved later by solving an overdetermined linear system.

It is worth noting that the specific choice of metric to measure proximity between images can make a big difference in class averaging. The cross-correlation or Euclidean distance (\ref{dist}) are by no means optimal measures of proximity. In practice, it is common to denoise the images prior to computing their pairwise distances. 
Although the discussion which follows is independent of the particular choice of filter or distance metric, we emphasize that filtering can have a dramatic effect on finding meaningful class averages.

The invariant distance between noisy images that share the same viewing angle (with perhaps a different in-plane rotation) is expected to be small. Ideally, all neighboring images of some reference image $P_i$ in a small invariant distance ball centered at $P_i$ should have similar viewing angles, and averaging such neighboring images (after proper rotational alignment) would amplify the signal and diminish the noise.

Unfortunately, due to the low SNR, it often happens that two images of completely different viewing angles have a small invariant distance. This can happen when the realizations of the noise in the two images match well for some random in-plane rotational angle, leading to spurious neighbor identification. Therefore, averaging the nearest neighbor images can sometimes yield a poor estimate of the true signal in the reference image.

\begin{figure}[h]
\begin{center}
\subfigure[Clean]{
    \includegraphics[width=0.16\textwidth]{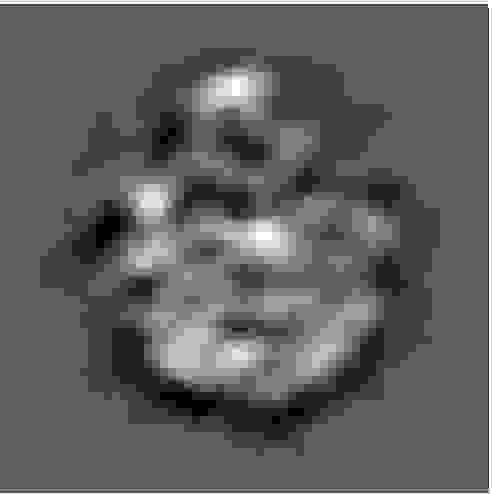}%
    }
\subfigure[SNR=1]{
    \includegraphics[width=0.16\textwidth]{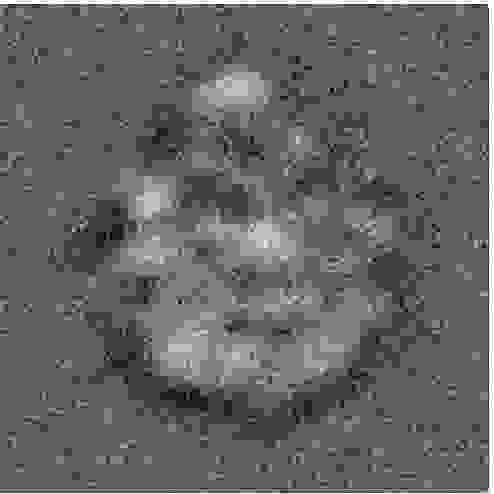}%
    }
\subfigure[SNR=1/2]{
    \includegraphics[width=0.16\textwidth]{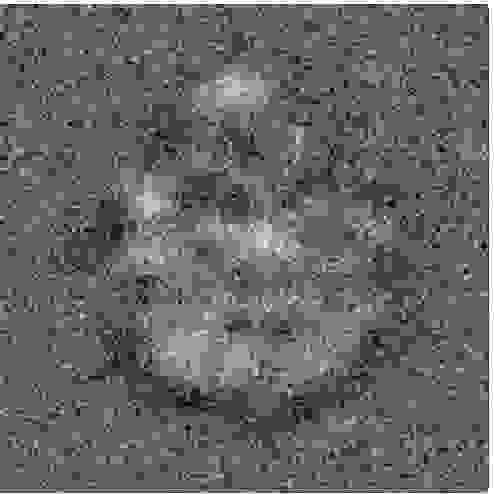}%
    }
\subfigure[SNR=1/4]{
    \includegraphics[width=0.16\textwidth]{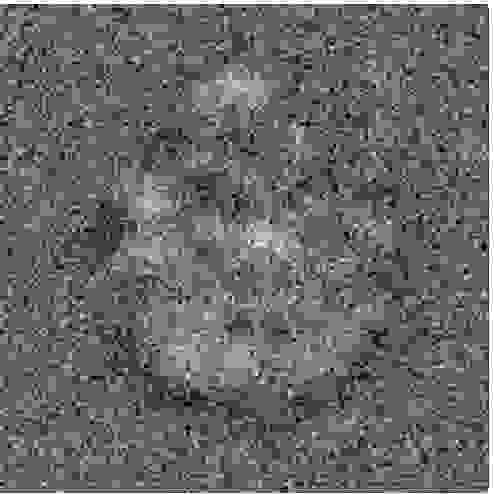}%
    }
\subfigure[SNR=1/8]{
    \includegraphics[width=0.16\textwidth]{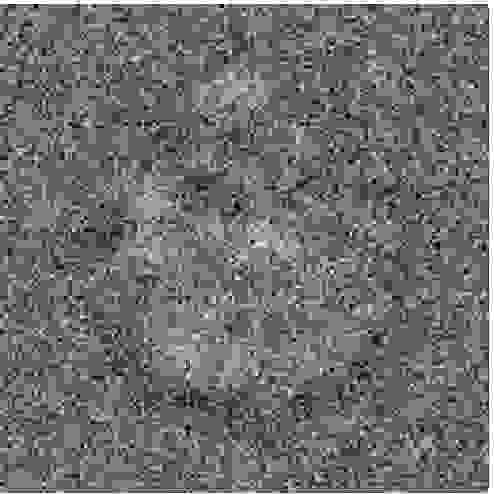}%
    }\\
\subfigure[SNR=1/16]{
    \includegraphics[width=0.16\textwidth]{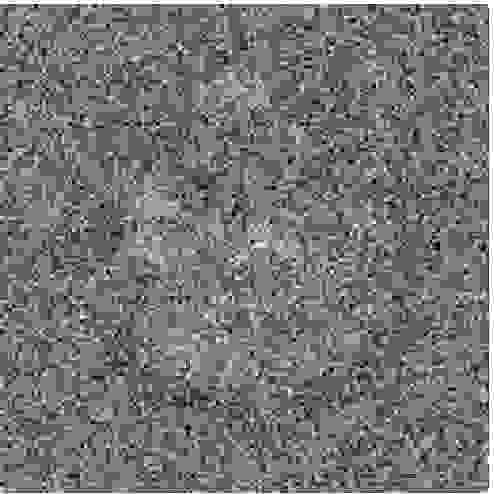}%
    }
\subfigure[SNR=1/32]{
    \includegraphics[width=0.16\textwidth]{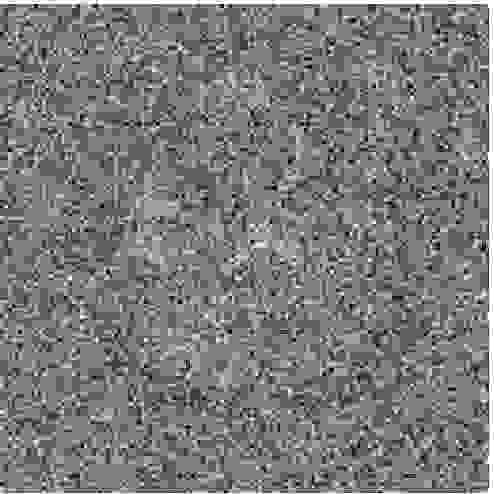}%
    }
\subfigure[SNR=1/64]{
    \includegraphics[width=0.16\textwidth]{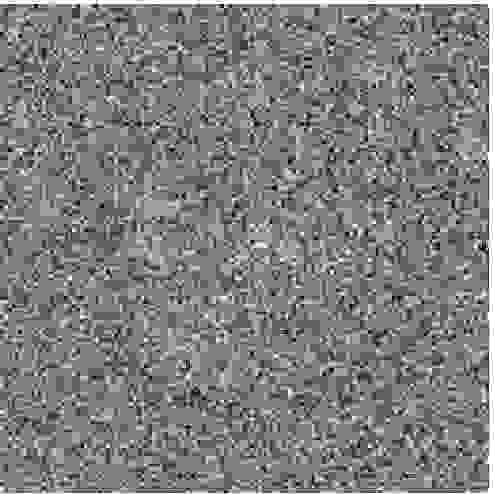}%
    }
\subfigure[SNR=1/128]{
    \includegraphics[width=0.16\textwidth]{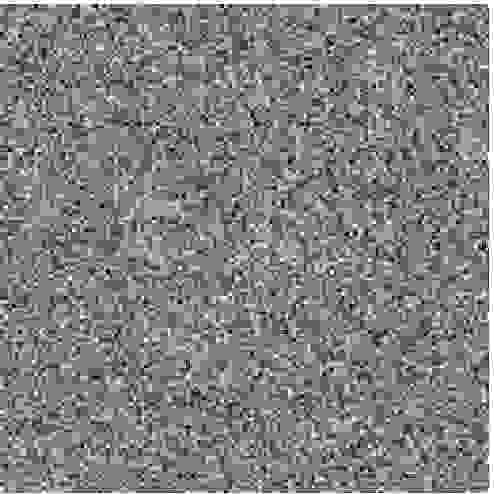}%
    }
\subfigure[SNR=1/256]{
    \includegraphics[width=0.16\textwidth]{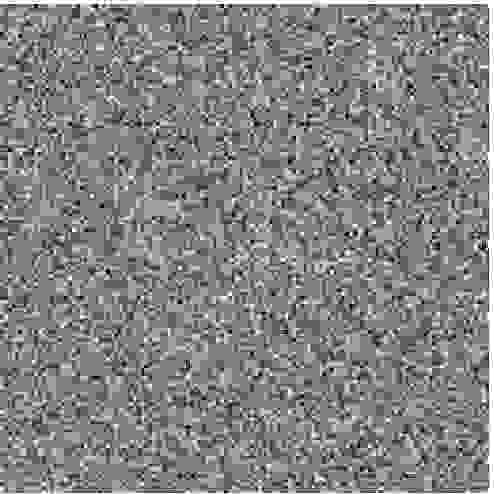}%
    }
\end{center}
\caption{Simulated projection with various levels of additive
Gaussian white noise.} \label{fig:ribosome projections}
\end{figure}

The histograms of Figure \ref{fig:histograms} demonstrate the ability of small rotationally invariant distances to identify images with similar viewing directions. For each image we use the rotationally invariant distances to find its 40 nearest neighbors among the entire set of $n=40,000$ images. In our simulation we know the original viewing directions, so for each image we compute the angles (in degrees) between the viewing direction of the image and the viewing directions of its 40 neighbors. Small angles indicate successful identification of ``true" neighbors that belong to a small spherical cap, while large angles correspond to outliers. We see that for SNR=$1/2$ there are no outliers, and all the viewing directions of the neighbors belong to a spherical cap whose opening angle is about $8^\circ$. However, for lower values of the SNR, there are outliers, indicated by arbitrarily large angles (all the way to $180^\circ$).

\begin{figure}[h]
\begin{center}
\subfigure[SNR=1/2]{
    \includegraphics[width=0.23\textwidth]{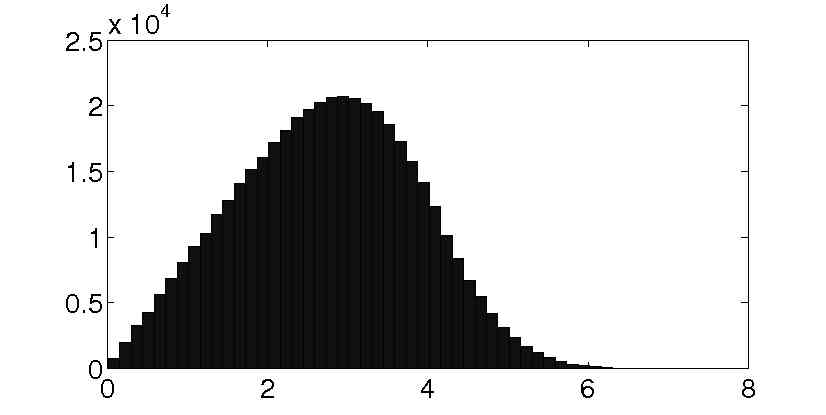}%
    }
\subfigure[SNR=1/16]{
    \includegraphics[width=0.23\textwidth]{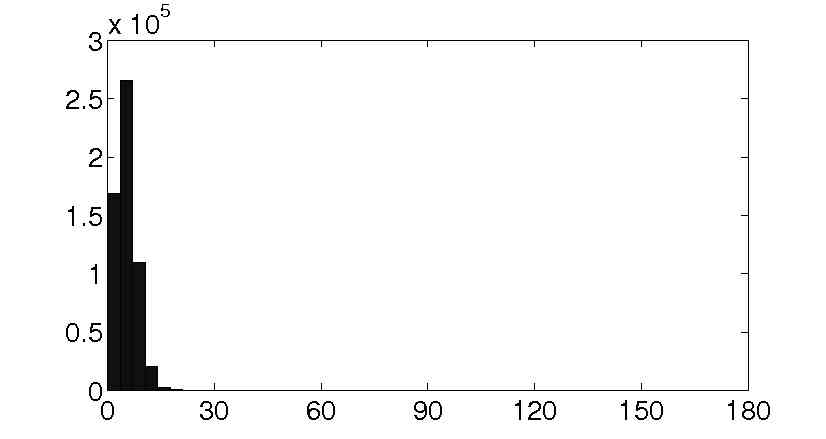}%
    }
\subfigure[SNR=1/32]{
    \includegraphics[width=0.23\textwidth]{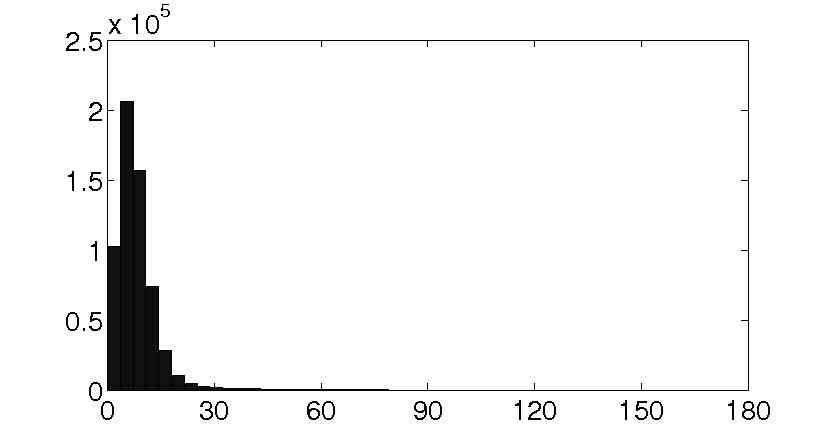}%
    }
\subfigure[SNR=1/64]{
    \includegraphics[width=0.23\textwidth]{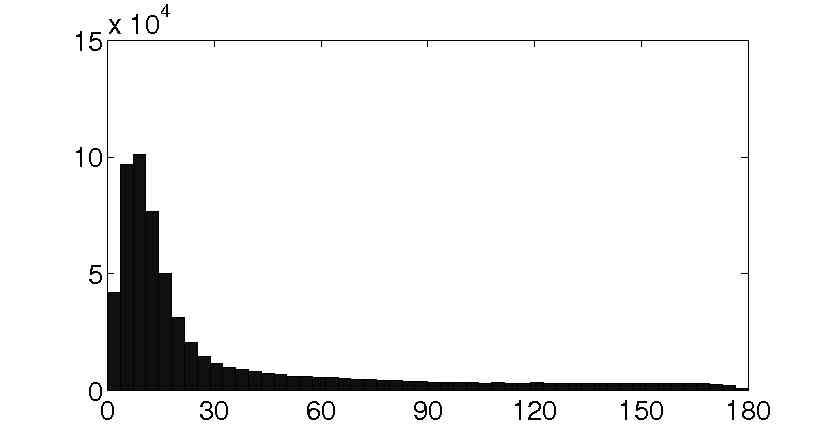}%
    }
\end{center}
\caption{Histograms of the angle (in degrees, $x$-axis) between the viewing directions of 40,000 images and the viewing directions of their 40 nearest neighboring images as found by computing the rotationally invariant distances.} \label{fig:histograms}
\end{figure}

Clustering algorithms, such as the K-means algorithm, perform much better than na\"ive nearest neighbors averaging, because they take into account all pairwise distances, not just distances to the reference image. Such clustering procedures are based on the philosophy that images that share a similar viewing angle with the reference image are expected to have a small invariant distance not only to the reference image but also to all other images with similar viewing angles. This observation was utilized in the rotationally invariant K-means clustering algorithm \cite{Penczek1996}. 
Still, due to noise, the rotationally invariant K-means clustering algorithm may suffer from misidentifications at the low SNR values present in experimental data.

VDM is a natural algorithmic framework for the class averaging problem, as it can further improve the detection of neighboring images even at lower SNR values. The rotationally invariant distance neglects an important piece of information, namely, the optimal angle that realizes the best rotational alignment in (\ref{dist}):
\begin{equation}
\label{delta}
\theta_{ij} = \argmin_{\theta \in [0,2\pi)} \|P_i - R(\theta) P_j \|, \quad i,j=1,\ldots,n.
\end{equation}
In VDM, we use the optimal in-plane rotation angles $\theta_{ij}$ to define the orthogonal transformations $O_{ij}$ and to construct the matrix $S$ in (\ref{S}). The eigenvectors and eigenvalues of $D^{-1}S$ (other normalizations of $S$ are also possible) are then used to define the vector diffusion distances between images.

This VDM based classification method is proven to be quite powerful in practice. We applied it to a set of $n=40,000$ noisy images with SNR=$1/64$. For every image we find the 40 nearest neighbors using the vector diffusion metric. In the simulation we know the viewing directions of the images, and we compute for each pair of neighbors the angle (in degrees) between their viewing directions. The histogram of these angles is shown in Figure \ref{fig:1/64} (Left panel). About $92\%$ of the identified images belong to a small spherical cap of opening angle $20^\circ$, whereas this percentage is only about $65\%$ when neighbors are identified by the rotationally invariant distances (Right panel). We remark that for SNR=$1/50$, the percentage of correctly identified images by the VDM method goes up to about $98\%$.
\begin{figure}[h]
\begin{center}
\subfigure[Neighbors are identified using $d_{\text{VDM}',t=2}$]{
    \includegraphics[width=0.48\textwidth]{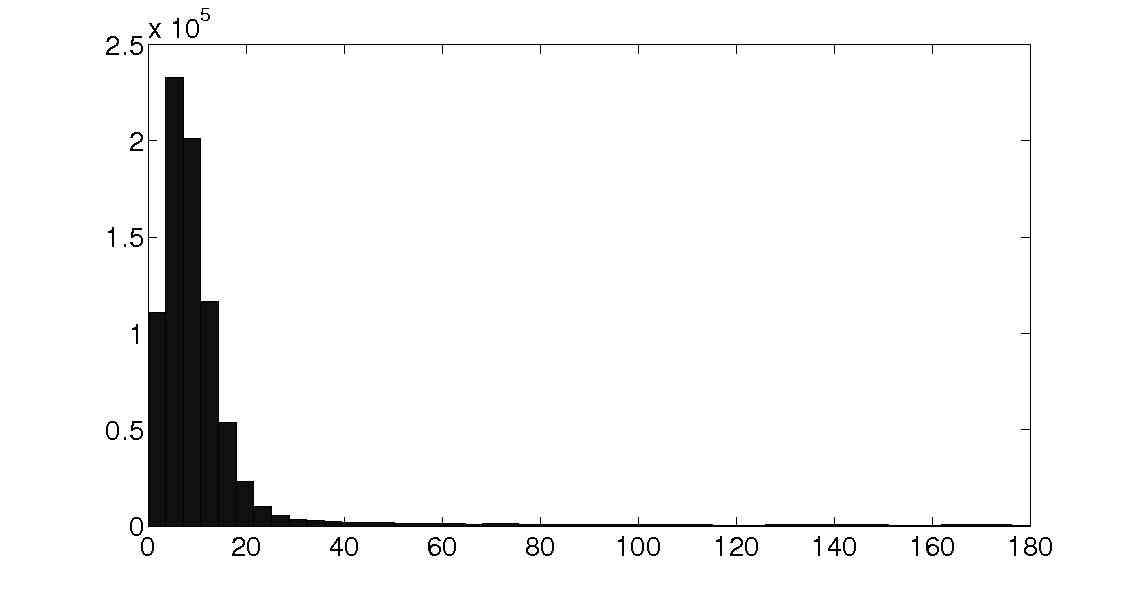}%
    }
\subfigure[Neighbors are identified using $d_{\text{RID}}$]{
    \includegraphics[width=0.48\textwidth]{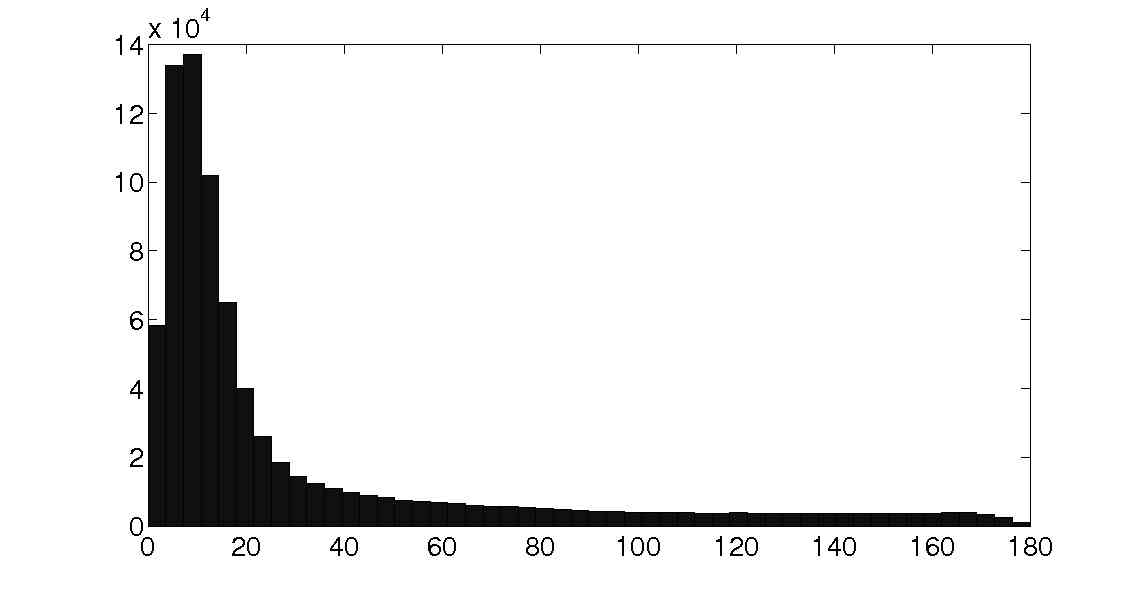}%
    }
\end{center}
\caption{SNR=$1/64$: Histogram of the angles ($x$-axis, in degrees) between the viewing directions of each image (out of $40000$) and it 40 neighboring images. Left: neighbors are post identified using vector diffusion distances. Right: neighbors are identified using the original rotationally invariant distances $d_{\text{RID}}$.} \label{fig:1/64}
\end{figure}

The main advantage of the algorithm presented here is that it successfully identifies images with similar viewing angles even in the presence of a large number of spurious neighbors, that is, even when many pairs of images with viewing angles that are far apart have relatively small rotationally invariant distances. In other words, the VDM-based algorithm is shown to be robust to outliers.

%% file: summary.tex
This paper introduced vector diffusion maps, an algorithmic and mathematical framework for analyzing data sets where scalar affinities between data points are accompanied with orthogonal transformations. The consistency among the orthogonal transformations along different paths that connect any fixed pair of data points is used to define an affinity between them. We showed that this affinity is equivalent to an inner product, giving rise to the embedding of the data points in a Hilbert space and to the definition of distances between data points, to which we referred as vector diffusion distances.

For data sets of images, the orthogonal transformations and the scalar affinities are naturally obtained via the procedure of optimal registration. The registration process seeks to find the optimal alignment of two images over some class of transformations (also known as deformations), such as rotations, reflections, translations and dilations. For the purpose of vector diffusion mapping, we extract from the optimal deformation only the corresponding orthogonal transformation (rotation and reflection). We demonstrated the usefulness of the vector diffusion map framework in the organization of noisy cryo-electron microscopy images, an important step towards resolving three-dimensional structures of macromolecules. Optimal registration is often used in various mainstream problems in computer vision and computer graphics, for example, in optimal matching of three-dimensional shapes. We therefore expect the vector diffusion map framework to become a useful tool in such applications.

In the case of manifold learning, where the data set is a collection of points in a high dimensional Euclidean space, but with a low dimensional Riemannian manifold structure, we detailed the construction of the orthogonal transformations via the optimal alignment of the orthonormal bases of the tangent spaces. These bases are found using the classical procedure of PCA. Under certain mild conditions about the sampling process of the manifold, we proved that the orthogonal transformation obtained by the alignment procedure approximates the parallel transport operator between the tangent spaces. The proof required careful analysis of the local PCA step which we believe is interesting of it own. Furthermore, we proved that if the manifold is sampled uniformly, then the matrix that lies at the heart of the vector diffusion map framework approximates the connection-Laplacian operator. Following spectral graph theory terminology, we call that matrix the connection-Laplacian of the graph. Using different normalizations of the matrix we proved convergence to the connection-Laplacian operator also for the case of non-uniform sampling. We showed that the vector diffusion mapping is an embedding and proved its relation with the geodesic distance using the asymptotic expansion of the heat kernel for vector fields. These results provide the mathematical foundation for the algorithmic framework that underlies the vector diffusion mapping.

We expect many possible extensions and generalizations of the vector diffusion mapping framework. We conclude by mentioning a few of them.
\begin{itemize}
\item {\em The topology of the data.} In \cite{odm} we showed how the vector diffusion mapping can determine if a manifold is orientable or non-orientable, and in the latter case to embed its double covering in a Euclidean space. To that end we used the information in the determinant of the optimal orthogonal transformation between bases of nearby tangent spaces. In other words, we used just the optimal reflection between two orthonormal bases. This simple example shows that vector diffusion mapping can be used to extract topological information from the point cloud. We expect more topological information can be extracted using appropriate modifications of the vector diffusion mapping.

\item {\em Hodge and higher order Laplacians.} Using tensor products of the optimal orthogonal transformations it is possible to construct higher order connection-Laplacians that act on $p$-forms ($p\geq 1$). The index theorem \cite{gilkey} relates topological structure with geometrical structure. For example, the so-called Betti numbers are related to the multiplicities of the harmonic $p$-forms of the Hodge Laplacian. For the extraction of topological information it would therefore be useful to modify our construction in order to approximate the Hodge Laplacian instead of the connection-Laplacian.

\item {\em Multiscale, sparse and robust PCA.} In the manifold learning case, an important step of our algorithm is local PCA for estimating the bases for tangent spaces at different data points. In the description of the algorithm, a single scale parameter $\epsilon_{\text{PCA}}$ is used for all data points. It is conceivable that a better estimation can be obtained by choosing a different, location-dependent scale parameter. A better estimation of the tangent space $T_{x_i}\mathcal{M}$ may be obtained by using a location-dependent scale parameter $\epsilon_{\text{PCA},i}$ due to several reasons: non-uniform sampling of the manifold, varying curvature of the manifold, and global effects such as different pieces of the manifold that are almost touching at some points (i.e., varying ``condition number" of the manifold). Choosing the correct scale $\epsilon_{\text{PCA},i}$ is a problem of its own interest that was recently considered in \cite{Little2010}, where a multiscale approach was taken to resolve the optimal scale. We recommend the incorporation of such multiscale PCA approaches into the vector diffusion mapping framework. Another difficulty that we may face when dealing with real-life data sets is that the underlying assumption about the data points being located exactly on a low-dimensional manifold does not necessarily hold. In practice, the data points are expected to reside off the manifold, either due to measurement noise or due to the imperfection of the low-dimensional manifold model assumption. It is therefore necessary to estimate the tangent spaces in the presence of noise. Noise is a limiting factor for successful estimation of the tangent space, especially when the data set is embedded in a high dimensional space and noise effects all coordinates \cite{johnstone-2006}. We expect recent methods for robust PCA \cite{RobustPCA} and sparse PCA \cite{BickelLevina,JohnstoneLu} to improve the estimation of the tangent spaces and as a result to become useful in the vector diffusion map framework.

\item {\em Random matrix theory and noise sensitivity.} The matrix $S$ that lies at the heart of the vector diffusion map is a block matrix whose blocks are either $d\times d$ orthogonal matrices $O_{ij}$ or the zero blocks. We anticipate that for some applications the measurement of $O_{ij}$ would be imprecise and noisy. In such cases, the matrix $S$ can be viewed as a random matrix and we expect tools from random matrix theory to be useful in analyzing the noise sensitivity of its eigenvectors and eigenvalues. The noise model may also allow for outliers, for example, orthogonal matrices that are uniformly distributed over the orthogonal group $O(d)$ (according to the Haar measure). Notice that the expected value of such random orthogonal matrices is zero, which leads to robustness of the eigenvectors and eigenvalues even in the presence of large number of outliers (see, for example, the random matrix theory analysis in \cite{amit2009}).

\item {\em Compact and non-compact groups and their matrix representation.} As mentioned earlier, the vector diffusion mapping is a natural framework to organize data sets for which the affinities and transformations are obtained from an optimal alignment process over some class of transformations (deformations). In this paper we focused on utilizing orthogonal transformations. At this point the reader have probably asked herself the following question: Is the method limited to orthogonal transformations, or is it possible to utilize other groups of transformations such as translations, dilations, and more? We note that the orthogonal group $O(d)$ is a compact group that has a matrix representation and remark that the vector diffusion mapping framework can be extended to such groups of transformations without much difficulty. However, the extension to non-compact groups, such as the Euclidean group of rigid transformation, the general linear group of invertible matrices and the special linear group is less obvious. Such groups arise naturally in various applications, rendering the importance of extending the vector diffusion mapping to the case of non-compact groups.
\end{itemize}

%% file: background.tex
The purpose of this appendix is to provide the required mathematical background for readers who are not familiar with concepts such as the parallel transport operator, connection, and the connection Laplacian. We illustrate these concepts by considering a surface $\mathcal{M}$ embedded in $\mathbb{R}^3$.

Given a function $f(x):\RR^3\rightarrow \RR$, its gradient vector field is given by
\[
\nabla f:=\left(\frac{\partial f}{\partial x},\frac{\partial f}{\partial y},\frac{\partial f}{\partial z}\right).
\]
Through the gradient, we can find the rate of change of $f$ at $x\in\RR^3$ in a given direction $v\in\RR^3$, using the directional derivative:
\[
vf(x):=\lim_{t\rightarrow 0}\frac{f(x+tv)-f(v)}{t}.
\]
By chain rule we have $vf(x)=\nabla f(x)(v)$. Define $\nabla_vf(x):=\nabla f(x)(v)$.

Let $X$ be a vector field on $\RR^3$,
\[
X(x,y,z)=(f_1(x,y,z),f_2(x,y,z),f_3(x,y,z)).
 \]
It is natural to extend the derivative notion to a given vector field $X$ at $x\in\RR^3$ by mimicking the derivative definition for functions in the following way:
\begin{equation}\label{derivative1}
\lim_{t\rightarrow 0}\frac{X(x+tv)-X(x)}{t}
\end{equation}
where $v\in\RR^3$. Following the same notation for the directional derivative of a function, we denote this limit by $\nabla_v X(x)$. This quantity tells us that at $x$, following the direction $v$, we compare the vector field at two points $x$ and $x+tv$, and see how the vector field changes. While this definition looks good at first sight, we now explain that it has certain shortcomings that need to be fixed in order to generalize it to the case of a surface embedded in $\RR^3$.

Consider a two dimensional smooth surface $\MM$ embedded in $\RR^3$ by $\iota$. Fix a point $x\in\MM$ and a smooth curve $\gamma(t):(-\epsilon,\epsilon)\rightarrow \MM\subset\RR^3$, where $\epsilon\ll 1$ and $\gamma(0)=x$. $\gamma'(0)\in\RR^3$ is called a tangent vector to $\MM$ at $x$. The 2 dimensional affine space spanned by the collection of all tangent vectors to $\MM$ at $x$ is defined to be the tangent plane at $x$ and denoted by\footnote{Here we abuse notation slightly. Usually $T_x\MM$ defined here is understood as the embedded tangent plane by the embedding $\iota$ of the tangent plane at $x$. Please see \cite{petersen} for a rigorous definition of the tangent plane.} $T_x\MM$, which is a two dimensional affine space inside $\RR^3$, as illustrated in Figure \ref{txm} (left panel). Having defined the tangent plane at each point $x\in\MM$, we define a vector field $X$ over $\MM$ to be a differentiable map that maps $x$ to a tangent vector in $T_x\MM$.\footnote{See \cite{petersen} for the exact notion of differentiability. Here, again, we abuse notation slightly. Usually $X$ defined here is understood as the embedded vector field by the embedding $\iota$ of the vector field $X$. For the rigorous definition of a vector field, please see \cite{petersen}.}

\begin{figure}[h]
\begin{centering}
\subfigure{
\includegraphics[width=0.3\textwidth]{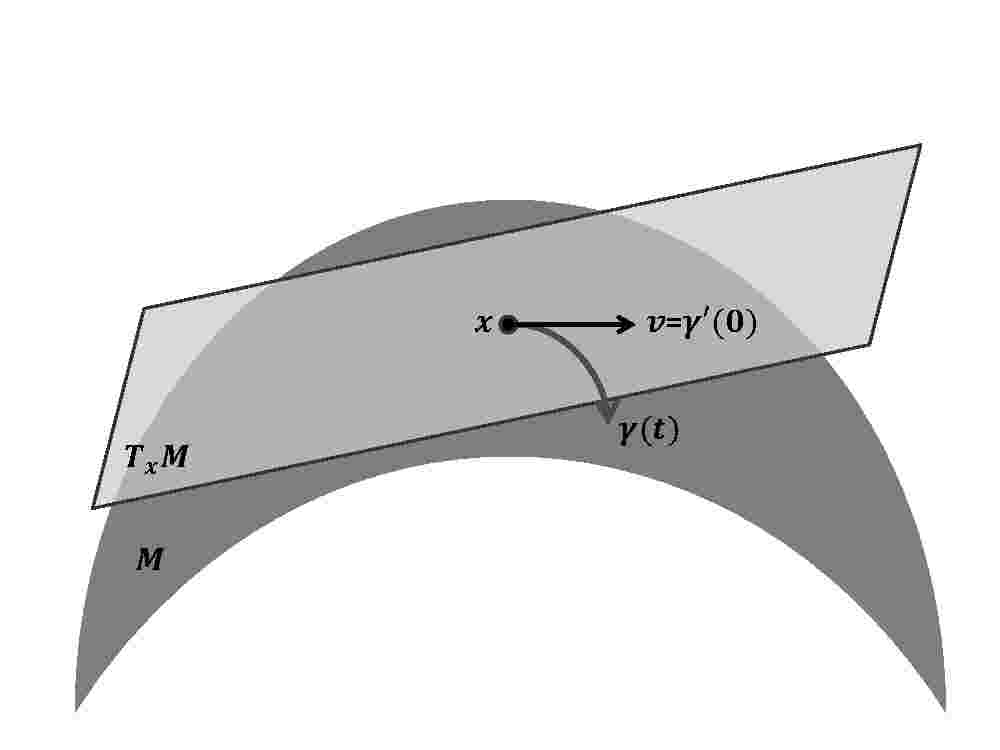}
}
\subfigure{
\includegraphics[width=0.3\textwidth]{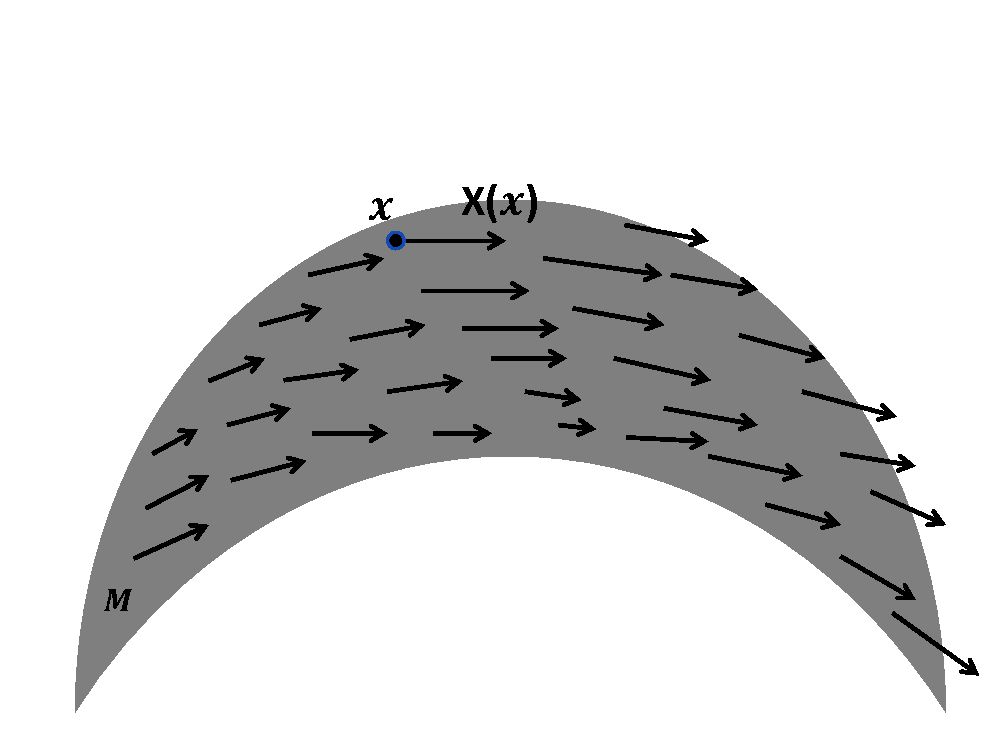}
}
\subfigure{
\includegraphics[width=0.3\textwidth]{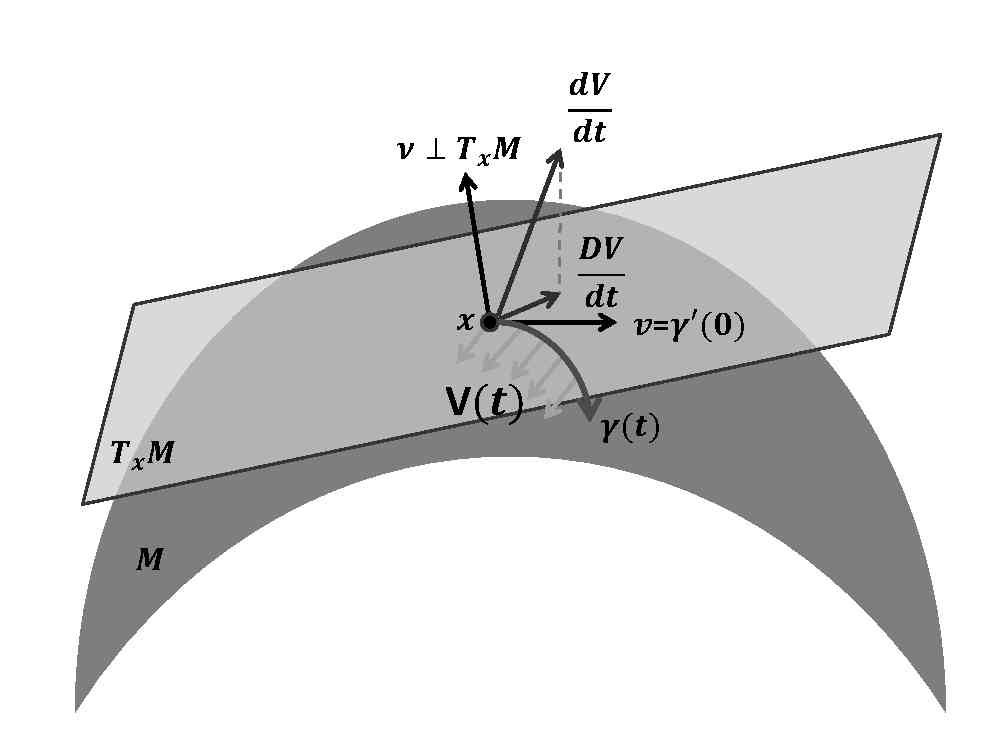}
}
\end{centering}
\caption{Left: a tangent plane and a curve $\gamma$; Middle: a vector field; Right: the covariant derivative}
\label{txm}
\end{figure}

We now generalize the definition of the derivative of a vector field over $\RR^3$ (\ref{derivative1}) to define the derivative of a vector field over $\MM$. The first difficulty we face is how to make sense of ``$X(x+tv)$'', since $x+tv$ does not belong to $\MM$. This difficulty can be tackled easily by changing the definition (\ref{derivative1}) a bit by considering the curve $\gamma:(-\epsilon,\epsilon)\rightarrow \RR^3$ so that $\gamma(0)=x$ and $\gamma'(0)=v$. Thus, (\ref{derivative1}) becomes
\begin{equation}\label{derivative2}
\lim_{t\rightarrow 0}\frac{X(\gamma(t))-X(\gamma(0))}{t}
\end{equation}
where $v\in\RR^3$. In $\MM$, the existence of the curve $\gamma:(-\epsilon,\epsilon)\rightarrow \RR^3$ so that $\gamma(0)=x$ and $\gamma'(0)=v$ is guaranteed by the classical ordinary differential equation theory. However, (\ref{derivative2}) still cannot be generalized to $\MM$ directly even though $X(\gamma(t))$ is well defined. The difficulty we face here is how to compare $X(\gamma(t))$ and $X(x)$, that is, how to make sense of the subtraction $X(\gamma(t))-X(\gamma(0))$. It is not obvious since a priori we do not know how $T_{\gamma(t)}\MM$ and $T_{\gamma(0)}\MM$ are related. The way we proceed is by defining an important notion in differential geometry called ``parallel transport'', which plays an essential role in our VDM framework.

Fix a point $x\in\MM$ and a vector field $X$ on $\MM$, and consider a parametrized curve $\gamma:(-\epsilon,\epsilon)\rightarrow \MM$ so that $\gamma(0)=x$. Define a vector valued function $V:(-\epsilon,\epsilon)\rightarrow \RR^3$ by restricting $X$ to $\gamma$, that is, $V(t)=X(\gamma(t))$. The derivative of $V$ is well defined as usual:
\[
\frac{\ud V}{\ud t}(h):=\lim_{t\rightarrow 0}\frac{V(h+t)-V(h)}{t},
\]
where $h\in(-\epsilon,\epsilon)$. The \textit{covariant derivative} $\frac{D V}{\ud t}(h)$ is defined as the projection of $\frac{\ud V}{\ud t}(h)$ onto $T_{\gamma(h)}\MM$. Then, using the definition of $\frac{D V}{\ud t}(h)$,
we consider the following equation:
\[
\left\{
\begin{array}{l}
\frac{DW}{\ud t}(t)=0\\
W(0)=w
\end{array}
\right.
\]
where $w\in T_{\gamma(0)}\MM$. The solution $W(t)$ exists by the classical ordinary differential equation theory. The solution $W(t)$ along $\gamma(t)$ is called the \textit{parallel vector field} along the curve $\gamma(t)$, and we also call $W(t)$ the \textit{parallel transport of $w$ along the curve $\gamma(t)$} and denote $W(t)=P_{\gamma(t),\gamma(0)}w$.

We come back to address the initial problem: how to define the ``derivative'' of a given vector field over a surface $\MM$. We define the \textit{covariant derivative} of a given vector field $X$ over $\MM$ as follows:
\begin{equation}\label{covderi}
\nabla_vX(x)=\lim_{t\rightarrow 0}\frac{P_{\gamma(0),\gamma(t)}X(\gamma(t))-X(\gamma(0))}{t},
\end{equation}
where $\gamma:(-\epsilon,\epsilon)\rightarrow\MM$ with $\gamma(0)=x\in\MM$, $\gamma'(0)=v\in T_{\gamma(0)}\MM$. This definition says that if we want to analyze how a given vector field at $x\in\MM$ changes along the direction $v$, we choose a curve $\gamma$ so that $\gamma(0)=x$ and $\gamma'(0)=v$, and then ``transport'' the vector field value at point $\gamma(t)$ to $\gamma(0)=x$ so that the comparison of the two tangent planes makes sense. The key fact of the whole story is that without applying parallel transport to transport the vector at point $\gamma(t)$ to $T_{\gamma(0)}\MM$, then the subtraction $X(\gamma(t))-X(\gamma(0))\in\RR^3$ in general does not live on $T_x\MM$, which distorts the notion of derivative. For comparison, let us reconsider the definition (\ref{derivative1}). Since at each point $x\in\RR^3$, the tangent plane at $x$ is $T_x\RR^3=\RR^3$, the substraction $X(x+tv)-X(x)$ always makes sense. To be more precise, the true meaning of $X(x+tv)$ is $P_{\gamma(0),\gamma(t)}X(\gamma(t))$, where $P_{\gamma(0),\gamma(t)}=id$, and $\gamma(t)=x+tv$.

With the above definition, when $X$ and $Y$ are two vector fields on $\MM$, we define $\nabla_XY$ to be a new vector field on $\MM$ so that $$\nabla_XY(x):=\nabla_{X(x)}Y.$$
Note that $X(x)\in T_x\MM$. We call $\nabla$ a connection on $\MM$.\footnote{The notion of connection can be quite general. For our purposes, this definition is sufficient.}

Once we know how to differentiate a vector field over $\MM$, it is natural to consider the second order differentiation of a vector field. The second order differentiation of a vector field is a natural notion in $\RR^3$.
For example, we can define a second order differentiation of a vector field $X$ over $\RR^3$ as follows:
\begin{equation}\label{connlapr3}
\nabla^2 X:=\nabla_x \nabla_xX+\nabla_y \nabla_yX+\nabla_z \nabla_zX,
\end{equation}
where $x,y,z$ are standard unit vectors corresponding to the three axes. This definition can be generalized to a vector field over $\MM$ as follows:
\begin{equation}
\nabla^2 X(x):=\nabla_{E_1} \nabla_{E_1}X(x)+\nabla_{E_2} \nabla_{E_2}X(x),
\end{equation}
where $X$ is a vector field over $\MM$, $x\in\MM$, and $E_1,E_2$ are two vector fields on $\MM$ that satisfy $\nabla_{E_i}E_j=0$ for $i,j=1,2$. The condition  $\nabla_{E_i}E_j=0$ (for $i,j=1,2$) is needed for technical reasons. Note that in the $\RR^3$ case (\ref{connlapr3}), if we set $E_1=x$, $E_2=y$ and $E_3=z$, then $\nabla_{E_i}E_j=0$ for $i,j=1,2,3$.\footnote{Please see \cite{petersen} for details.} The operator $\nabla^2$ is called the \textit{connection Laplacian operator}, which lies in the heart of the VDM framework. The notion of \textit{eigen-vector-field} over $\MM$ is defined to be the solution of the following equation:
\[
\nabla^2 X(x)=\lambda X(x)
\]
for some $\lambda\in\RR$. The existence and other properties of the eigen-vector-fields can be found in \cite{gilkey}.
Finally, we comment that all the above definitions can be extended to the general manifold setup without much difficulty, where, roughly speaking, a ``manifold'' is the higher dimensional generalization of a surface\footnote{We will not provide details in the manifold setting, and refer readers to standard differential geometry textbooks, such as \cite{petersen}.}.

%% file: proof6.tex
Before stating and proving the theorems, we set up the notation that is used throughout this Appendix. Let  $\iota:\mathcal{M}\hookrightarrow\RR^p$ be a smooth $d$-dim compact Riemannian manifold embedded in $\RR^p$, with metric $g$ induced from the canonical metric on $\RR^p$. Denote $\MM_{t}=\{x\in\MM:~\min_{y\in\partial\MM}d(x,y)\leq t\}$, where $d(x,y)$ is the geodesic distance between $x$ and $y$. The data points $x_1,x_2,\ldots,x_n$ are independent samples from $\MM$ according to the probability density function $p\in C^3(\MM)$ supported on $\MM\subset\RR^p$ and satisfies $0< p(x)<\infty$. We assume that the kernels used in the local PCA step and for the construction of the matrix $S$ are in $C^2([0,1])$. Although these kernels can be different we denote both of them by $K$ and expect their meaning to be clear from the context. Denote $\tau$ to be the largest number having the property: the open normal bundle about $\MM$ of radius $r$ is embedded in $\RR^p$ for every $r<\tau$ \cite{NSW}. This condition holds automatically since $\MM$ is compact. In all theorems, we assume that $\sqrt{\epsilon}< \tau$. In \cite{NSW}, $1/\tau$ is referred to as the ``condition number'' of $\MM$. We denote $P_{y,x} : T_{x}\MM \to T_{y}\MM$ to be the parallel transport from $x$ to $y$ along the geodesic linking them. Denote by $\nabla$ the connection over $T\MM$ and $\nabla^2$ the connection Laplacian over $\MM$. Denote by $\mathcal{R}$, $\Ric$, and $s$ the curvature tensor, the Ricci curvature, and the scalar curvature of $\MM$, respectively. The second fundamental form of the embedding $\iota$ is denoted by $\Pi$. To ease notation, in the sequel we use the same notation $\nabla$ to denote different connections on different bundles whenever there is no confusion and the meaning is clear from the context.

We divide the proof of Theorem \ref{summary} into four theorems, each of which has its own interest. The first theorem, Theorem \ref{localpcatheorem}, states that the columns of the matrix $O_i$ that are found by local PCA (see (\ref{Oi})) form an orthonormal basis to a $d$-dimensional subspace of $\mathbb{R}^p$ that approximates the embedded tangent plane $\iota_*T_{x_i}\MM$. The proven order of approximation is crucial for proving Theorem \ref{summary}. The proof of Theorem \ref{localpcatheorem} involves geometry and probability theory.
\begin{thm}\label{localpcatheorem}
If $\epsilon_{\text{PCA}}=O(n^{-\frac{2}{d+2}})$ and $x_i\notin\MM_{\sqrt{\epsilon_{\text{PCA}}}}$, then, with high probability (w.h.p.), the columns $\{u_l(x_i)\}_{l=1}^d$ of the $p\times d$ matrix $O_i$ which is determined by local PCA, form an orthonormal basis to a $d$-dim subspace of $\RR^p$ that deviates from $\iota_*T_{x_i}\MM$ by $O(\epsilon_{\text{PCA}}^{3/2})$, in the following sense:
\begin{equation}\label{minOiQi}
\min_{O\in O(d)} \|O_i^T \Theta_i - O\|_{HS} = O(\epsilon_{\text{PCA}}^{3/2}) = O(n^{-\frac{3}{d+2}}),
\end{equation}
where $\Theta_i$ is a $p\times d$ matrix whose columns form an orthonormal basis to  $\iota_*T_{x_i}\MM$.
Let the minimizer in (\ref{minOiQi}) be
\begin{equation}
\hat{O}_i=\argmin_{O\in O(d)} \|O_i^T \Theta_i - O\|_{HS},
\end{equation}
and denote by $Q_i$ the $p\times d$ matrix
\begin{equation}\label{hatQ}
Q_i:=\Theta_i\hat{O}_i^T,
\end{equation}
and $e_l(x_i)$ the $l$-th column of $Q_i$. The columns of $Q_i$ form an orthonormal basis to $\iota_*T_{x_i}\MM$, and
\begin{equation}
\label{OiQiPCA}
\|O_i - Q_i \|_{HS} = O(\epsilon_{\text{PCA}}).
\end{equation}

If $x_i\in\MM_{\sqrt{\epsilon_{\text{PCA}}}}$, then, w.h.p.
\begin{equation*}
\min_{O\in O(d)} \|O_i^T \Theta_i - O\|_{HS} = O(\epsilon_{\text{PCA}}^{1/2}) = O(n^{-\frac{1}{d+2}}).
\end{equation*}

Better convergence near the boundary is obtained for $\epsilon_{\text{PCA}} = O(n^{-\frac{2}{d+1}})$, which gives
\begin{equation*}
\min_{O\in O(d)} \|O_i^T \Theta_i - O\|_{HS} = O(\epsilon_{\text{PCA}}^{3/4}) = O(n^{-\frac{3}{2(d+1)}}),
\end{equation*}
for $x_i\in\MM_{\sqrt{\epsilon_{\text{PCA}}}}$, and 
\begin{equation}
\min_{O\in O(d)} \|O_i^T \Theta_i - O\|_{HS} = O(\epsilon_{\text{PCA}}^{5/4}) = O(n^{-\frac{5}{2(d+1)}}),
\end{equation}
for $x_i\notin\MM_{\sqrt{\epsilon_{\text{PCA}}}}$.
\end{thm}

Theorem \ref{localpcatheorem} may seem a bit counterintuitive at first glance. When considering data points in a ball of radius $\sqrt{\epsilon_{\text{PCA}}}$, it is expected that the order of approximation would be $O(\epsilon_{\text{PCA}})$, while equation (\ref{minOiQi}) indicates that the order of approximation is higher ($3/2$ instead of 1). The true order of approximation for the tangent space, as observed in (\ref{OiQiPCA}) is still $O(\epsilon)$. The improvement observed in (\ref{minOiQi}) is of relevance to Theorem \ref{relateOP} and we relate it to the probabilistic nature of the PCA procedure, more specifically, to a large deviation result for the error in the law of large numbers for the covariance matrix that underlies PCA. Since the convergence of PCA is slower near the boundary, then for manifolds with boundary we need a smaller $\epsilon_{\text{PCA}}$. Specifically, for manifolds without boundary we choose $\epsilon_{\text{PCA}}=O(n^{-\frac{2}{d+2}})$ and for manifolds with boundary we choose $\epsilon_{\text{PCA}} = O(n^{-\frac{2}{d+1}})$. We remark that the first choice works also for manifolds with boundary at the expense of a slower convergence rate.

The second theorem, Theorem \ref{relateOP}, states that the $d\times d$ orthonormal matrix $O_{ij}$, which is the output of the alignment procedure (\ref{HS}), approximates the parallel transport operator $P_{x_i,x_j}$ from $x_j$ to $x_i$ along the geodesic connecting them. Assuming that $\|x_i-x_j \| = O(\sqrt{\epsilon})$ (here, $\epsilon$ is different than $\epsilon_{\text{PCA}}$), the order of this approximation is $O(\epsilon_{\text{PCA}}^{3/2} + \epsilon^{3/2})$ whenever $x_i,x_j$ are away from the boundary. This result is crucial for proving Theorem \ref{summary}. The proof of Theorem \ref{relateOP} uses Theorem \ref{localpcatheorem} and is purely geometric.

\begin{thm}\label{relateOP}
Consider $x_i,x_j\notin\MM_{\sqrt{\epsilon_\text{PCA}}}$ satisfying that the geodesic distance between $x_i$ and $x_j$ is $O(\sqrt{\epsilon})$. For $\epsilon_{\text{PCA}}=O(n^{-\frac{2}{d+2}})$, w.h.p., $O_{ij}$ approximates $P_{x_i,x_j}$ in the following sense:
\begin{align}
\begin{split}
O_{ij}\bar{X}_j=\left(\langle \iota_*P_{x_i,x_j}X(x_j),u_l(x_i)\rangle\right)^d_{l=1}+O(\epsilon_{\text{PCA}}^{3/2} + \epsilon^{3/2}), \quad \mbox{for all } X\in C^3(T\MM),
\end{split}
\end{align}
where $\bar{X}_i \equiv (\langle \iota_*X(x_i),u_l(x_i)\rangle)^d_{l=1}\in \mathbb{R}^d$, and $\{u_l(x_i)\}_{l=1}^d$ is an orthonormal set determined by local PCA. For $x_i,x_j\in\MM_{\sqrt{\epsilon_\text{PCA}}}$
\begin{align}
\begin{split}
O_{ij}\bar{X}_j=\left(\langle \iota_*P_{x_i,x_j}X(x_j),u_l(x_i)\rangle\right)^d_{l=1}+O(\epsilon_{\text{PCA}}^{1/2} + \epsilon^{3/2}), \quad \mbox{for all } X\in C^3(T\MM),
\end{split}
\end{align}
For $\epsilon_{\text{PCA}}=O(n^{-\frac{2}{d+1}})$, the orders of $\epsilon_{\text{PCA}}$ in the error terms change according to Theorem \ref{localpcatheorem}.  
\end{thm}

The third theorem, Theorem \ref{convergetokernel}, states that the $n\times n$ block matrix $D_\alpha^{-1}S_\alpha$ is a discrete approximation of an integral operator over smooth sections of the tangent bundle. The integral operator involves the parallel transport operator. The proof of Theorem \ref{convergetokernel} mainly uses probability theory.

\begin{thm}\label{convergetokernel}
Suppose $\epsilon_{\text{PCA}}=O(n^{-\frac{2}{d+2}})$, and for $0\leq \alpha\leq 1$, define \textit{the estimated probability density distribution} by
\[
p_\epsilon(x_i)=\sum_{j=1}^n K_\epsilon\left(x_i,x_j\right)
\]
and \textit{the normalized kernel $K_{\epsilon,\alpha}$} by
\[
K_{\epsilon,\alpha}(x_i,x_j)=\frac{K_{\epsilon}(x_i,x_j)}{p^\alpha_\epsilon(x_i)p^\alpha_\epsilon(x_j)},
\]
where $K_\epsilon\left(x_i,x_j\right)=K\left(\frac{\|\iota(x_i)-\iota(x_j)\|_{\mathbb{R}^p}}{\sqrt{\epsilon}}\right)$.

For $x_i\notin\MM_{\sqrt{\epsilon_{\text{PCA}}}}$ we have w.h.p.
\begin{eqnarray}\label{conv}
\frac{\sum_{j=1,j\neq i}^n K_{\epsilon,\alpha}\left(x_i,x_j\right)O_{ij}\bar{X}_j}{\sum_{j=1,j\neq i}^n K_{\epsilon,\alpha}\left(x_i,x_j\right)} &=& (\langle \iota_*T_{\epsilon,\alpha} X(x_i),u_l(x_i)\rangle )^d_{l=1}\\
&& + O\left(\frac{1}{n^{1/2}\epsilon^{d/4-1/2}} + \epsilon_{\text{PCA}}^{3/2}+\epsilon^{3/2}\right),\nonumber
\end{eqnarray}
where
\begin{equation}\label{TepsalphaDEF}
T_{\epsilon,\alpha} X(x_i)=\frac{\int_\MM K_{\epsilon,\alpha}(x_i,y) P_{x_i,y}X(y)\ud V(y)}{\int_\MM K_{\epsilon,\alpha}(x_i,y)\ud V(y)},
\end{equation}
$\bar{X}_i \equiv (\langle \iota_*X(x_i),u_l(x_i)\rangle)^d_{l=1}\in \mathbb{R}^d$, $\{u_l(x_i)\}_{l=1}^d$ is the orthonormal set determined by local PCA, $X\in C^3(T\MM)$, and $O_{ij}$ is the optimal orthogonal transformation determined by the alignment procedure.

For $x_i\in\MM_{\sqrt{\epsilon_{\text{PCA}}}}$ we have w.h.p.
\begin{eqnarray}\label{convbdry}
\frac{\sum_{j=1,j\neq i}^n K_{\epsilon,\alpha}\left(x_i,x_j\right)O_{ij}\bar{X}_j}{\sum_{j=1,j\neq i}^n K_{\epsilon,\alpha}\left(x_i,x_j\right)} &=& (\langle \iota_*T_{\epsilon,\alpha} X(x_i),u_l(x_i)\rangle )^d_{l=1} \\
&& + O\left(\frac{1}{n^{1/2}\epsilon^{d/4-1/2}} + \epsilon_{\text{PCA}}^{1/2}+\epsilon^{3/2}\right).\nonumber
\end{eqnarray}

For $\epsilon_{\text{PCA}}=O(n^{-\frac{2}{d+1}})$ the orders of $\epsilon_{\text{PCA}}$ in the error terms change according to Theorem \ref{localpcatheorem}. 
\end{thm}


The fourth theorem, Theorem \ref{Tepsexpansion}, states that the operator $T_{\epsilon,\alpha}$ can be expanded in powers of $\sqrt{\epsilon}$, where the leading order term is the identity operator, the second order term is the connection-Laplacian operator plus some possible potential terms, and the first and third order terms vanish for vector fields that are sufficiently smooth. For $\alpha=1$, the potential terms vanish, and as a result, the second order term is the connection-Laplacian. The proof is based on geometry.

\begin{thm}\label{Tepsexpansion}
For $X \in C^3(T\mathcal{M})$ and $x\notin\MM_{\sqrt{\epsilon}}$ we have:
\begin{equation}\label{FPEexpansion}
T_{\epsilon,\alpha} X(x)=X(x)+\epsilon\frac{m_2}{2dm_0}\left\{\nabla^2X(x)+d\frac{\int_{S^{d-1}} \nabla_{\theta}X(x)\nabla_\theta(p^{1-\alpha})(x)\ud\theta}{p^{1-\alpha}(x)}\right\}+O(\epsilon^{2}),
\end{equation}
where $T_{\epsilon,\alpha}$ is defined in (\ref{TepsalphaDEF}), $\nabla^2$ is the connection-Laplacian over vector fields, and $m_l=\int_{\mathbb{R}^d}\|x\|^l K(\|x\|)\ud x$. 

\end{thm}

\begin{cor}\label{Tepsexpansioncor}
Under the same conditions and notations as in Theorem \ref{Tepsexpansion}, if $X\in C^3(T\MM)$, then for all $x\notin\MM_{\sqrt{\epsilon}}$ we have:
\begin{equation}\label{unif}
T_{\epsilon,1} X(x)=X(x)+\epsilon\frac{m_2}{2dm_0}\nabla^2 X(x)+O(\epsilon^{2}).
\end{equation}
\end{cor}
Putting Theorems \ref{localpcatheorem} \ref{convergetokernel} and \ref{Tepsexpansion} together, we now prove Theorem \ref{summary}:
\begin{proof}[Proof of Theorem \ref{summary}]
Suppose $x_i \notin \MM_{\sqrt{\epsilon}}$. By Theorem \ref{convergetokernel}, w.h.p.
\begin{eqnarray}\label{eq:proof1}
\frac{\sum_{j=1,j\neq i}^n K_{\epsilon,\alpha}\left(x_i,x_j\right)O_{ij}\bar{X}_j}{\sum_{j=1,j\neq i}^n K_{\epsilon,\alpha}\left(x_i,x_j\right)} &=& (\langle \iota_*T_{\epsilon,\alpha} X(x_i),u_l(x_i)\rangle )^d_{l=1}\\
&& + O\left(\frac{1}{n^{1/2}\epsilon^{d/4-1/2}} + \epsilon_{\text{PCA}}^{3/2}+\epsilon^{3/2}\right),\nonumber \\
&=& (\langle \iota_*T_{\epsilon,\alpha} X(x_i),e_l(x_i)\rangle )^d_{l=1}\nonumber \\
&& + O\left(\frac{1}{n^{1/2}\epsilon^{d/4-1/2}} + \epsilon_{\text{PCA}}^{3/2}+\epsilon^{3/2}\right),\nonumber
\end{eqnarray}
where $\epsilon_{\text{PCA}} = O(n^{-\frac{2}{d+2}})$, and we used Theorem \ref{localpcatheorem} to replace $u_l(x_i)$ by $e_l(x_i)$. Using Theorem \ref{Tepsexpansion} for the right hand side of (\ref{eq:proof1}), we get
\begin{eqnarray*}
&&\frac{\sum_{j=1,j\neq i}^n K_{\epsilon,\alpha}\left(x_i,x_j\right)O_{ij}\bar{X}_j}{\sum_{j=1,j\neq i}^n K_{\epsilon,\alpha}\left(x_i,x_j\right)}\\
&=&\left(\left\langle \iota_*X(x_i)+\epsilon\frac{m_2}{2dm_0}\iota_*\left\{\nabla^2X(x_i)+d\frac{\int_{S^{d-1}} \nabla_{\theta}X(x_i)\nabla_\theta(p^{1-\alpha})(x_i)\ud\theta}{p^{1-\alpha}(x_i)}\right\},e_l(x_i)\right\rangle\right)_{l=1}^d\\
&&+O\left(\frac{1}{n^{1/2}\epsilon^{d/4-1/2}} + \epsilon_{\text{PCA}}^{3/2}+\epsilon^{3/2}\right).
\end{eqnarray*}
For $\epsilon = O(n^{-\frac{2}{d+4}})$, upon dividing by $\epsilon$, the three error terms are
\begin{eqnarray*}
\frac{1}{n^{1/2}\epsilon^{d/4+1/2}} &=& O(n^{-\frac{1}{d+4}}),\\
\frac{1}{\epsilon}\epsilon_{\text{PCA}}^{3/2} &=& O(n^{-\frac{d+8}{(d+1)(d+2)}}), \\
\epsilon^{1/2} &=& O(n^{-\frac{1}{d+4}}).
\end{eqnarray*}
Clearly the three error terms vanish as $n\to \infty$. Specifically, the dominant error is $O(n^{-\frac{1}{d+4}})$ which is the same as $O(\sqrt{\epsilon})$. 
As a result, in the limit $n\to \infty$, almost surely,
\begin{align*}
\begin{split}
&\quad\lim_{n\rightarrow\infty}\frac{1}{\epsilon}\left[\frac{\sum_{j=1,j\neq i}^n K_{\epsilon,\alpha}\left(x_i,x_j\right)O_{ij}\bar{X}_j}{\sum_{j=1,j\neq i}^n K_{\epsilon,\alpha}\left(x_i,x_j\right)}-\bar{X}_i\right]\\
&=\frac{m_2}{2dm_0}\left(\left\langle \iota_*\left\{\nabla^2X(x_i)+d\frac{\int_{S^{d-1}} \nabla_{\theta}X(x_i)\nabla_\theta(p^{1-\alpha})(x_i)\ud\theta}{p^{1-\alpha}(x_i)}\right\},e_l(x_i)\right\rangle\right)_{l=1}^d,
\end{split}
\end{align*}
as required.
\end{proof}

\subsection{Preliminary Lemmas}
For the proofs of Theorem \ref{localpcatheorem}-\ref{Tepsexpansion}, we need the following Lemmas.

\begin{lem}\label{volexpansion}
In polar coordinates around $x\in\MM$, the Riemannaian measure is given by
\[
\ud V(\exp_xt\theta)=J(t,\theta)\ud t\ud\theta,
\]
where $\theta\in T_x\MM$, $\|\theta\|=1$, $t>0$, and
\[
J(t,\theta)=t^{d-1}+t^{d+1}\Ric(\theta,\theta)+O(t^{d+2}).
\]
\end{lem}
\begin{proof}
Please see \cite{petersen}.
\end{proof}

The following Lemma is needed in Theorem \ref{localpcatheorem} and \ref{relateOP}.
\begin{lem}\label{relateexp}
Fix $x\in\MM$ and denote $\exp_x$ the exponential map at $x$ and $\exp^{\RR^p}_{\iota(x)}$ the exponential map at $\iota(x)$. With the identification of $T_{\iota(x)}\RR^p$ with $\RR^p$, for $v\in T_x\MM$ with $\|v\|\ll 1$ we have
\begin{equation}\label{relateexp1}
\iota\circ\exp_x(v)=\iota(x)+d\iota(v)+\frac{1}{2}\Pi(v,v)+\frac{1}{6}\nabla_v\Pi(v,v)+O(\|v\|^4).
\end{equation}
Furthermore, for $w\in T_x\MM\cong \RR^d$, we have
\begin{equation}\label{relateexp2}
d\left[\iota\circ\exp_x\right]_{v}(w)=d\left[\iota\circ\exp_x\right]_{v=0}(w)+\Pi(v,w)+\frac{1}{6}\nabla_v\Pi(v,w)+\frac{1}{3}\nabla_w\Pi(v,v)+O(\|v\|^3)
\end{equation}

\end{lem}
\begin{proof}
Denote $\phi=(\exp^{\RR^p}_{\iota(x)})^{-1}\circ\iota\circ\exp_x$, that is,
\begin{equation}\label{iexp}
\iota\circ\exp_x=\exp_{\iota(x)}^{\RR^p}\circ \phi.
\end{equation}
Note that $\phi(0)=0$. Since $\phi$ can be viewed as a function from $T_x\MM\cong \RR^d$ to $T_{\iota(x)}\RR^p\cong \RR^p$, we can Taylor expand it to get
\[
\iota\circ\exp_x(v)=\exp_{\iota(x)}^{\RR^p}\left(d\phi|_0(v)+\frac{1}{2}\nabla d\phi|_0(v,v)+\frac{1}{6}\nabla^2d\phi|_0(v,v,v)+O(\|v\|^4)\right).
\]
We claim that
\begin{equation}\label{dexpx}
\nabla^kd\exp_x|_0(v,\ldots,v)=0 \mbox{ for all }k\geq 1.
\end{equation}
Indeed, from the definition of the exponential map we have that $d\exp_x\in \Gamma(T^*T_x\MM\otimes T\MM)$ and $$\nabla d\exp_x(v'(t),v'(t))=\nabla_{d\exp_x(v'(t))}d\exp_x(v'(t))-d\exp_x(\nabla_{v'(t)}v'(t)),$$ where $v\in T_x\MM$, $v(t)=tv\in T_x\MM$, and $v'(t)=v\in T_{tv}T_x\MM$. When evaluated at $t=0$, we get the claim (\ref{dexpx}) for $k=1$. The result for $k\geq 2$ follows from a similar argument.

We view $d\iota$ as a smooth section of $Hom(T\MM,T\RR^p)$. Thus, by combining (\ref{dexpx}) with the chain rule, from (\ref{iexp}) we have
\[
\nabla^kd\phi|_0(v,\ldots,v)=\nabla^kd\iota|_x(v,\ldots,v)\mbox{ for all }k\geq 0,
\]
and hence we obtain
\[
\iota\circ\exp_x(v)=\exp_{\iota(x)}^{\RR^p}\left(d\iota(v)+\frac{1}{2}\nabla d\iota(v,v)+\frac{1}{6}\nabla^2d\iota(v,v,v)+O(\|v\|^4)\right).
\]
To conclude (\ref{relateexp1}), note that for all $v\in T_x\MM$, we have $\exp^{\RR^p}_{\iota(x)}(v)=\iota(x)+v$ for all $v\in T_{\iota(x)}\RR^p$ if we identify $T_{\iota(x)}\RR^p$ with $\RR^p$. Next consider vector fields $U$, $V$ and $W$ around $x$ so that $U(x)=u$, $V(x)=v$ and $W(x)=w$, where $u,v,w\in T_x\MM$. A direct calculation gives
\begin{align*}
\begin{split}
\nabla d\iota(V,W)=(\nabla_Vd\iota)W=\nabla_{V}(d\iota (W))-d\iota(\nabla_V W)
\end{split}
\end{align*}
which is by definition the second fundamental form $\Pi(V,W)$ of the embedding $\iota$. Similarly, we have
\begin{align*}
\begin{split}
\nabla^2 d\iota(U,V,W)&=(\nabla^2_{U,V}d\iota)W
=(\nabla_U(\nabla_Vd\iota))W-(\nabla_{\nabla_UV}d\iota)W\\
&=\nabla_{U}((\nabla_Vd\iota )W)-\nabla_Vd\iota(\nabla_U W)-(\nabla_{\nabla_UV}d\iota)W\\
&=\nabla_{U}(\Pi(V,W))-\Pi(V,\nabla_UW)-\Pi(\nabla_UV,W)=:(\nabla_U\Pi)(V,W),
\end{split}
\end{align*}
Evaluating $\nabla d\iota(V,V)$ and $\nabla^2 d\iota(V,V,V)$ at $x$ gives us (\ref{relateexp1}).

Next, when $w\in T_vT_x\MM$ and $v\in T_x\MM$, since $d\left[\iota\circ\exp_x\right]_{v}(w)\in T_{\iota\circ\exp_x v}\RR^p\cong \RR^p$, we can view $d\left[\iota\circ\exp_x\right]_{\cdot}(w)$ as a function from $T_x\MM\cong \RR^d$ to $\RR^p$. Thus, when $\|v\|$ is small enough, Taylor expansion gives us
\[
d\left[\iota\circ\exp_x\right]_{v}(w)=d\left[\iota\circ\exp_x\right]_{0}(w)+\nabla (d\left[\iota\circ\exp_x\right]_{\cdot}(w))|_0(v)+\frac{1}{2}\nabla^2(d\left[\iota\circ\exp_x\right]_{\cdot}(w))|_0(v,v)+O(\|v\|^3),
\]
here $d$ and $\nabla$ are understood as the ordinary differentiation over $\RR^d$. To simplify the following calculation, for $u,v,w\in \RR^d$, we denote
\[
H_w(v)=\frac{1}{6}\nabla_{w}\Pi(v,v)+\frac{1}{6}\nabla_v\Pi(w,v)+\frac{1}{6}\nabla_v\Pi(v,w),
\]
and
\[
G_w(u,v)=\frac{1}{3}(\nabla_{w}\Pi(u,v)+\nabla_u\Pi(w,v)+\nabla_v\Pi(u,w)).
\]
Note that we again identify $\RR^d$ with $T_x\MM$ in the following calculation. By (\ref{relateexp1}), when $\|v\|$ is small enough, we have
\begin{align*}
\begin{split}
d\left[\iota\circ\exp_x\right]_{v}(w)&=\lim_{\delta\rightarrow 0}\frac{\iota\circ\exp_x(v+\delta w)-\iota\circ\exp_x(v)}{\delta}\\
&=\lim_{\delta\rightarrow 0}\frac{d\iota(\delta w)+\Pi(v,\delta w)+\delta H(v)+R(v+\delta w)-R(v)}{\delta},\\
\end{split}
\end{align*}
where $R(v)$ is the remainder term in the Taylor expansion:
\[
R(v)=\sum_{|\alpha|=4}\frac{1}{\alpha!}\left(\int_0^1(1-t)^3\nabla^3(\iota\circ\exp_x)(tv)\ud t\right)v^\alpha,
\]
from which it follows that
\[
\frac{R(v+\delta w)-R(v)}{\delta}=O(\|v\|\|w\|),
\]
and as a result
\begin{align}\label{relateexp3}
\begin{split}
d\left[\iota\circ\exp_x\right]_{v}(w)=d\iota(w)+\Pi(v,w)+H(v)+O(\|v\|\|w\|).
\end{split}
\end{align}
Similarly, from (\ref{relateexp3}), when $\|u\|$ is small enough we have
\begin{align}\label{relateexp4}
\begin{split}
\nabla (d\left[\iota\circ\exp_x\right]_{\cdot}(w))|_{u}(v)&=\lim_{\delta\rightarrow 0}\frac{d\left[\iota\circ\exp_x\right]_{u+\delta v}(w)-d\left[\iota\circ\exp_x\right]_{u}(w)}{\delta}\\
&=\Pi(v,w)+G(u,v)+O(\|u\|\|v\|\|w\|).
\end{split}
\end{align}
Finally, from (\ref{relateexp4}) we have
\begin{align}\label{relateexp5}
\begin{split}
\nabla^2 d(\left[\iota\circ\exp_x\right]_{\cdot}(w))|_{0}(v,v)&=\lim_{\delta\rightarrow 0}\frac{\nabla (d\left[\iota\circ\exp_x\right]_{\cdot}(w))|_{\delta v}(v)-\nabla (d\left[\iota\circ\exp_x\right]_{\cdot}(w))|_0(v)}{\delta}\\
&=G(v,v).
\end{split}
\end{align}
Thus, from (\ref{relateexp3}) we have that
$$d\left[\iota\circ\exp_x\right]_{0}(w)=d\iota(w),$$ from (\ref{relateexp4}) we have that
$$\nabla (d\left[\iota\circ\exp_x\right]_{\cdot}(w))|_{0}(v)=\Pi(v,w),$$ and from (\ref{relateexp5}) we have that $$G(v,v)=\frac{1}{3}\nabla_v\Pi(v,w)+\frac{2}{3}\nabla_w\Pi(v,v).$$ Putting it all together we get (\ref{relateexp2}) as required.
\end{proof}

\begin{lem}\label{rdist}
Suppose $x,y\in \MM$ such that $y=\exp_x(t\theta)$, where $\theta\in T_x\MM$ and $\|\theta\|=1$. If $t\ll 1$, then $h=\|\iota(x)-\iota(y)\|\ll 1$ satisfies
\begin{equation}
t=h+\frac{1}{24}\|\Pi(\theta,\theta)\|h^3+O(h^4).
\end{equation}
\end{lem}
\begin{proof}
Please see \cite{Coifman20065} or apply (\ref{relateexp1}) directly.
\end{proof}

\begin{lem}\label{relatebasis}
Fix $x\in\MM$ and $y=\exp_x(t\theta)$, where $\theta\in T_x\MM$ and $\|\theta\|=1$. Let $\{\partial_l(x)\}_{l=1}^d$ be the normal coordinate on a neighborhood $U$ of $x$, then for a sufficiently small $t$,  we have:
\begin{align}
\begin{split}
\iota_*P_{y,x}\partial_l(x)&=\iota_*\partial_l(x)+t\Pi(\theta,\partial_l(x))+\frac{t^2}{6}\nabla_\theta\Pi(\theta,\partial_l(x))\\
&\quad\quad+\frac{t^2}{3}\nabla_{\partial_l(x)}\Pi(\theta,\theta)
-\frac{t^2}{6}\iota_*P_{y,x}(\mathcal{R}(\theta,\partial_l(x))\theta)+O(t^3).
\end{split}
\end{align}
for all $l=1,\ldots,d$.
\end{lem}

\begin{proof}
Choose an open subset $U\subset \MM$ small enough and find an open neighborhood $B$ of $0\in T_x\MM$ so that $\exp_x:B\rightarrow U$ is diffeomorphic. It is well known that
\[
\partial_l(\exp_x(t\theta))=\frac{J_l(t)}{t},
\]
where $J_l(t)$ is the Jacobi field with $J_l(0)=0$ and $\nabla_tJ_l(0)=\partial_l(x)$. By applying Taylor's expansion in a neighborhood of $t=0$, we have
\begin{align*}
\begin{split}
J_l(t)=P_{y,x}\left(J_l(0)+t\nabla_t J_l(0)+\frac{t^2}{2}\nabla_t^2 J_l(0)+\frac{t^3}{6}\nabla_t^3J_l(0)\right)+O(t^4),\\
\end{split}
\end{align*}
Since $J_l(0)=\nabla^2_t J_l(0)=0$, the following relationship holds:
\begin{align}\label{jacobirelation}
\begin{split}
\partial_l(\exp_x(t\theta))&=P_{y,x}\left(\nabla_t J_l(0)+\frac{t^2}{6}\nabla_t^3J_l(0)\right)+O(t^3)\\
&=P_{y,x}\partial_l(x)+\frac{t^2}{6}P_{y,x}(\mathcal{R}(\theta,\partial_l(x))\theta)+O(t^3).
\end{split}
\end{align}
Thus we obtain
\begin{align}\label{partialE}
\begin{split}
P_{y,x}\partial_l(x)=\partial_l(\exp_x(t\theta))-\frac{t^2}{6}P_{y,x}(\mathcal{R}(\theta,\partial_l(x))\theta)+O(t^3),
\end{split}
\end{align}
On the other hand, from (\ref{relateexp2}) in Lemma \ref{relateexp} we have
\begin{align}\label{partiall}
\begin{split}
\iota_*\partial_l(\exp_x(t\theta))=\iota_*\partial_l(x)+t\Pi(\theta,\partial_l(x))+\frac{t^2}{6}\nabla_\theta\Pi(\theta,\partial_l(x))+\frac{t^2}{3}\nabla_{\partial_l(x)}\Pi(\theta,\theta)
+O(t^3).
\end{split}
\end{align}
Putting (\ref{partialE}) and (\ref{partiall}) together, it follows that for $l=1,\ldots,d$:
\begin{align}
\begin{split}
\iota_*P_{y,x}\partial_l(x)&=\iota_*\partial_l(\exp_x(t\theta))-\frac{t^2}{6}\iota_*P_{y,x}(\mathcal{R}(\theta,\partial_l(x))\theta)+O(t^3)\\
&=\iota_*\partial_l(x)+t\Pi(\theta,\partial_l(x))+\frac{t^2}{6}\nabla_\theta\Pi(\theta,\partial_l(x))\\
&\quad\quad+\frac{t^2}{3}\nabla_{\partial_l(x)}\Pi(\theta,\theta)
-\frac{t^2}{6}\iota_*P_{y,x}(\mathcal{R}(\theta,\partial_l(x))\theta)+O(t^3).
\end{split}
\end{align}

\end{proof}

\subsection{[Proof of Theorem \ref{localpcatheorem}]}

\begin{proof}
Fix $x_i\notin\MM_{\sqrt{\epsilon_{\text{PCA}}}}$. Denote $\{v_k\}_{k=1}^p$ the standard orthonormal basis of $\RR^p$, that is, $v_k$ has $1$ in the $k$-th entry and $0$ elsewhere. We can properly translate and rotate the embedding $\iota$ so that $\iota(x_i)=0$, the first $d$ components $\{v_1,\ldots,v_d\}\subset\RR^p$ form the orthonormal basis of $\iota_*T_{x_i}\MM$, and find a normal coordinate $\{\partial_k\}_{k=1}^d$ around $x_i$ so that $\iota_*\partial_k(x_i)=v_k$. Instead of directly analyzing the matrix $B_i$ that appears in the local PCA procedure given in (\ref{Xi2}), we analyze the covariance matrix $\Xi_i:=B_iB_i^T$, whose eigenvectors coincide with the left singular vectors of $B_i$. We rewrite $\Xi_i$ as
\begin{equation}
\Xi_i = \sum_{j\neq i}^{n} F_j,
\end{equation}
where
\begin{equation}
\label{Fj}
F_j=K\left(\frac{\|\iota(x_i)-\iota(x_j)\|_{\mathbb{R}^p}}{\sqrt{\epsilon_{\text{PCA}}}}\right)(\iota(x_j)-\iota(x_i))(\iota(x_j)-\iota(x_i))^T,
\end{equation}
and
\begin{equation}
\label{Fjkl}
F_j(k,l)=K\left(\frac{\|\iota(x_i)-\iota(x_j)\|_{\mathbb{R}^p}}{\sqrt{\epsilon_{\text{PCA}}}}\right)\langle \iota(x_j)-\iota(x_i),v_k\rangle\langle \iota(x_j)-\iota(x_i),v_l\rangle.
\end{equation}

Denote $B_{\sqrt{\epsilon_{\text{PCA}}}}(x_i)$ to be the geodesic ball of radius $\sqrt{\epsilon_{\text{PCA}}}$ around $x_i$. We apply the same variance error analysis as in \cite[Section 3]{Singer2006a} to approximate $\Xi_i$. Since the points $x_i$ are independent identically distributed (i.i.d.), $F_j$, $j\neq i$, are also i.i.d., by the law of large numbers one expects
\begin{align}\label{approximationlocalpca}
\begin{split}
\frac{1}{n-1}\sum_{j\neq i}^{n} F_j \approx \EE F,
\end{split}
\end{align}
where $F=F_1$,
\begin{align}
\begin{split}
\EE F = \int_{B_{\sqrt{\epsilon_{\text{PCA}}}}(x_i)} K_{\epsilon_{\text{PCA}}}(x_i,y)(\iota(y)-\iota(x_i))(\iota(y)-\iota(x_i))^T
p(y)\ud V(y),
\end{split}
\end{align}
and
\begin{align}\label{Fkl}
\begin{split}
\EE F(k,l)=\int_{B_{\sqrt{\epsilon_{\text{PCA}}}}(x_i)} K_{\epsilon_{\text{PCA}}}(x_i,y)\langle\iota(y)-\iota(x_i),v_k\rangle\langle \iota(y)-\iota(x_i),v_l\rangle p(y)
\ud V(y).
\end{split}
\end{align}
In order to evaluate the first moment $\EE F(k,l)$ of (\ref{Fkl}), we note that for $y=\exp_{x_i}v$, where $v\in T_{x_i}\MM$, by (\ref{relateexp1}) in Lemma \ref{relateexp} we have
\begin{equation}\label{fkl}
\langle\iota(\exp_{x_i}v)-\iota(x_i),v_k\rangle = \langle \iota_*v, v_k\rangle+\frac{1}{2}\langle \Pi(v,v), v_k\rangle+\frac{1}{6}\langle \nabla_v\Pi(v,v), v_k\rangle+O(\|v\|^4).
\end{equation}
Substituting (\ref{fkl}) into (\ref{Fkl}), applying Taylor's expansion, and combining Lemma \ref{rdist} and Lemma \ref{volexpansion}, we have
\begin{eqnarray}
&& \int_{B_{\sqrt{\epsilon_{\text{PCA}}}}(x_i)} K_{\epsilon_{\text{PCA}}}(x_i,y)\langle\iota(y)-\iota(x_i),v_k\rangle\langle \iota(y)-\iota(x_i),v_l\rangle p(y)
\ud V(y) \nonumber \\
&=&\int_{S^{d-1}}\int_0^{\sqrt{\epsilon_{\text{PCA}}}} \left[K\left(\frac{t}{\sqrt{\epsilon_{\text{PCA}}}}\right)+O\left(\frac{t^3}{\sqrt{\epsilon_{\text{PCA}}}}\right)\right] \times \label{expansionfkl} \\
&& \Big\{t^2\langle \iota_*\theta, v_k\rangle\langle \iota_*\theta, v_l\rangle + \frac{t^3}{2}\Big(\langle \Pi(\theta,\theta), v_k\rangle\langle \iota_*\theta, v_l\rangle+\langle \Pi(\theta,\theta), v_l\rangle\langle \iota_*\theta, v_k\rangle \Big) + O(t^4)\Big\} \times \nonumber \\
&& \left(p(x_i)+t\nabla_\theta p(x_i)+O(t^2)\right)\left(t^{d-1}+O(t^{d+1})\right)\ud t\ud \theta \nonumber \\
&=& \int_{S^{d-1}}\int_0^{\sqrt{\epsilon_{\text{PCA}}}} \left[ K\left(\frac{t}{\sqrt{\epsilon_{\text{PCA}}}}\right)\langle \iota_*\theta, v_k\rangle\langle \iota_*\theta, v_l\rangle p(x_i)t^{d+1} + O(t^{d+3})\right]\ud t\ud \theta, \label{expansionfkl2}
\end{eqnarray}
where (\ref{expansionfkl2}) holds since integrals involving odd powers of $\theta$ must vanish due to the symmetry of the sphere $S^{d-1}$.
Note that $\langle \iota_*\theta, v_k\rangle=0$ when $k=d+1,\ldots,p$. Therefore,
\begin{equation}
\label{firstmomentFjkl}
\EE F(k,l)= \left\{\begin{array}{ll}
                           D{\epsilon_{\text{PCA}}}^{d/2+1}+O({\epsilon_{\text{PCA}}}^{d/2+2}) & \mbox{for } 1\leq k=l \leq d, \\
                           O({\epsilon_{\text{PCA}}}^{d/2+2}) & \mbox{otherwise} .
                         \end{array}
 \right.
\end{equation}
where $D = \int_{S^{d-1}} |\langle \iota_*\theta, v_1 \rangle|^2 \ud \theta \int_0^{1} K(u) u^{d+1} \ud u$ is a positive constant.

Similar considerations give the second moment of $F(k,l)$ as
\begin{equation}\label{secondmoment}
\EE [F(k,l)^2]= \left\{
\begin{array}{ll}
O({\epsilon_{\text{PCA}}}^{d/2+2})&\mbox{for } k,l=1,\ldots,d,\\
O({\epsilon_{\text{PCA}}}^{d/2+4})&\mbox{for } k,l=d+1,\ldots,p,\\
O({\epsilon_{\text{PCA}}}^{d/2+3})&\mbox{otherwise.}
\end{array}
\right.
\end{equation}
Hence, the variance of $F(k,l)$ becomes
\begin{equation}
\operatorname{Var} F(k,l) = \left\{
\begin{array}{ll}
O({\epsilon_{\text{PCA}}}^{d/2+2})&\mbox{for } k,l=1,\ldots,d,\\
O({\epsilon_{\text{PCA}}}^{d/2+4})&\mbox{for } k,l=d+1,\ldots,p,\\
O({\epsilon_{\text{PCA}}}^{d/2+3})&\mbox{otherwise.}
\end{array}
\right.
\end{equation}

We now move on to establish a large deviation bound on the estimation of $\frac{1}{n-1}\sum_{j\neq i} F_j(k,l)$ by its mean $\EE F_j(k,l)$.
For that purpose, we measure the deviation from the mean value by $\alpha$ and define its probability by
\begin{equation}
p_{k,l}(n,\alpha):=\mbox{Pr}\left\{\left|\frac{1}{n-1}\sum_{j\neq i}^{n} F_j(k,l)-\EE F(k,l)\right|>\alpha\right\}.
\end{equation}
To establish an upper bound for the probability $p_{k,l}(n,\alpha)$, we use Bernstein's inequality, see, e.g., \cite{hoeffding}.
Define $$Y_j(k,l):=F_j(k,l)-\EE F(k,l).$$
Clearly $Y_j(k,l)$ are zero mean i.i.d. random variables.
From the definition of $F_j(k,l)$ (see \ref{Fj} and \ref{Fjkl}) and from the calculation of its first moment (\ref{firstmomentFjkl}), it follows that $Y_j(k,l)$ are bounded random variables. More specifically,
\begin{equation}
Y_j(k,l) = \left\{\begin{array}{ll}
                    O({\epsilon_{\text{PCA}}}) & \mbox{for } k,l=1,\ldots,d,\\
                    O({\epsilon_{\text{PCA}}}^2) & \mbox{for } k,l=d+1,\ldots,p, \\
                    O({\epsilon_{\text{PCA}}}^{3/2}) & \mbox{otherwise}.
                  \end{array}
 \right.
\end{equation}
%
%
Consider first the case $k,l=1,\ldots,d$, for which Bernstein's inequality gives
\begin{align}\begin{split}\label{pklbdd}
p_{k,l}(n,\alpha)\leq \exp\left\{-\frac{(n-1)\alpha^2}{2\EE (Y_1(k,l)^2)+O({\epsilon_{\text{PCA}}})\alpha}\right\}\leq\exp\left\{-\frac{(n-1)\alpha^2}{O({\epsilon_{\text{PCA}}}^{d/2+2})+O({\epsilon_{\text{PCA}}})\alpha}\right\}.
\end{split}\end{align}
From (\ref{pklbdd}) it follows that w.h.p.
\[
\alpha =  O\left(\frac{{\epsilon_{\text{PCA}}}^{d/4+1}}{n^{1/2}}\right),
\]
provided that
\begin{equation}
\label{condition-n-epsilon}
\frac{1}{n^{1/2}{\epsilon_{\text{PCA}}}^{d/4}}\ll 1.
\end{equation}
Similarly, for $k,l=d+1,\ldots,p$, we have
\begin{align*}\begin{split}
p_{k,l}(n,\alpha)\leq \exp\left\{-\frac{(n-1)\alpha^2}{O({\epsilon_{\text{PCA}}}^{d/2+4})+O({\epsilon_{\text{PCA}}}^2)\alpha}\right\},
\end{split}\end{align*}
which means that w.h.p.
\[
\alpha = O\left(\frac{{\epsilon_{\text{PCA}}}^{d/4+2}}{n^{1/2}}\right)
\]
provided (\ref{condition-n-epsilon}).
Finally, for $k=d+1,\ldots,p$, $l=1,\ldots,d$ or $l=d+1,\ldots,,p$, $k=1,\ldots,d$, we have
\begin{align*}\begin{split}
p_{k,l}(n,\alpha)\leq \exp\left\{-\frac{(n-1)\alpha^2}{O({\epsilon_{\text{PCA}}}^{d/2+3})+O({\epsilon_{\text{PCA}}}^{3/2})\alpha}\right\},
\end{split}\end{align*}
which means that w.h.p.
\[
\alpha = O\left(\frac{{\epsilon_{\text{PCA}}}^{d/4+3/2}}{n^{1/2}}\right)
\]
provided (\ref{condition-n-epsilon}). The condition (\ref{condition-n-epsilon}) is quite intuitive as it is equivalent to $n {\epsilon_{\text{PCA}}}^{d/2} \gg 1$, which says that the expected number of points inside $B_{\sqrt{\epsilon_{\text{PCA}}}}(x_i)$ is large.

As a result, when (\ref{condition-n-epsilon}) holds, w.h.p., the covariance matrix $\Xi_i$ is given by
\begin{eqnarray}
\label{covariance-bias-variance}
\Xi_i &=& {\epsilon_{\text{PCA}}}^{d/2+1}D\left[
\begin{array}{ll}
I_{d\times d} & 0_{d\times p-d} \\
0_{p-d\times d} & 0_{p-d\times p-d} \\
\end{array}
\right]\\
&&+ {\epsilon_{\text{PCA}}}^{d/2+2}\left[\begin{array}{ll}
                                                                      O(1) & O(1) \\
                                                                      O(1) & O(1)
                                                                    \end{array}\right]\label{covariance-bias-variance2}\\
&&+ \frac{{\epsilon_{\text{PCA}}}^{d/4+1}}{\sqrt{n}}\left[\begin{array}{ll}
                                                                      O(1) & O({\epsilon_{\text{PCA}}}^{1/2}) \\
                                                                      O({\epsilon_{\text{PCA}}}^{1/2}) & O({\epsilon_{\text{PCA}}})
                                                                    \end{array}
 \right],\label{covariance-bias-variance3}
\end{eqnarray}
where $I_{d\times d}$ is the identity matrix of size $d\times d$, and $0_{m\times m'}$ is the zero matrix of size $m \times m'$. The error term in (\ref{covariance-bias-variance2}) is the bias term due to the curvature of the manifold, while the error term in (\ref{covariance-bias-variance3}) is the variance term due to finite sampling (i.e., finite $n$). In particular, under the condition in the statement of the theorem for the sampling rate, namely, ${\epsilon_{\text{PCA}}}=O(n^{-\frac{2}{d+2}})$, we have w.h.p.
\begin{eqnarray}
\Xi_i &=& {\epsilon_{\text{PCA}}}^{d/2+1}D\left[
\begin{array}{ll}
I_{d\times d} & 0_{d\times p-d}  \\
0_{p-d\times d} & 0_{p-d\times p-d} \\
\end{array}
\right]\label{Xiexpansion}\\
&&+ {\epsilon_{\text{PCA}}}^{d/2+2}\left[\begin{array}{ll}
                                                                      O(1) & O(1) \\
                                                                      O(1) & O(1)
                                                                    \end{array}\right] + {\epsilon_{\text{PCA}}}^{d/2 + 3/2}\left[\begin{array}{ll}
                                                                      O(1) & O({\epsilon_{\text{PCA}}}^{1/2}) \\
                                                                      O({\epsilon_{\text{PCA}}}^{1/2}) & O({\epsilon_{\text{PCA}}})
                                                                    \end{array}
 \right] \nonumber \\
 &=& {\epsilon_{\text{PCA}}}^{d/2+1}\left\{D\left[
\begin{array}{ll}
I_{d\times d} & 0_{d\times p-d} \\
0_{p-d\times d} & 0_{p-d\times p-d} \\
\end{array}
\right] +  \left[\begin{array}{ll}
                                                                      O({\epsilon_{\text{PCA}}}^{1/2}) & O({\epsilon_{\text{PCA}}}) \\
                                                                      O({\epsilon_{\text{PCA}}}) & O({\epsilon_{\text{PCA}}})
                                                                    \end{array}
 \right]\right\}.\nonumber
\end{eqnarray}
Note that by definition $\Xi_i$ is symmetric, so we rewrite (\ref{Xiexpansion}) as
\begin{eqnarray}
\Xi_i &=& {\epsilon_{\text{PCA}}}^{d/2+1}D\left[
\begin{array}{ll}
I+{\epsilon_{\text{PCA}}}^{1/2}A & {\epsilon_{\text{PCA}}} C\\
{\epsilon_{\text{PCA}}}C^T & {\epsilon_{\text{PCA}}} B \\
\end{array}
\right],
\end{eqnarray}
where $I$ is the $d\times d$ identity matrix, $A$ is a $d\times d$ symmetric matrix, $C$ is a $d\times (p-d)$ matrix, and $B$ is a $(p-d)\times (p-d)$ symmetric matrix. All entries of $A$, $B$, and $C$ are $O(1)$.
Denote by $u_k$ and $\lambda_k$, $k=1,\ldots,p$, the eigenvectors and eigenvalues of $\Xi_i$, where the eigenvectors are orthonormal, and the eigenvalues are ordered in a decreasing order. Using regular perturbation theory, we find that $\lambda_k = D{\epsilon_{\text{PCA}}}^{d/2+1}\left(1+O({\epsilon_{\text{PCA}}}^{1/2})\right)$ (for $k=1,\ldots,d$), and that the expansion of the first $d$ eigenvectors $\{u_k\}_{k=1}^d$ is given by
\begin{equation}
\label{uk-wk}
u_k=\left[
\begin{array}{l}
\left[ w_k+O({\epsilon_{\text{PCA}}}^{3/2}) \right]_{d\times 1}\\
\left[ O({\epsilon_{\text{PCA}}}) \right]_{p-d\times 1} \\
\end{array}
\right]\in\RR^p,
\end{equation}
where $\{w_k\}_{k=1}^d$ are orthonormal eigenvectors of $A$ satisfying $Aw_k=\lambda^A_kw_k$. Indeed, a direct calculation gives us
\begin{eqnarray}
&&\left[
\begin{array}{ll}
I+{\epsilon_{\text{PCA}}}^{1/2}A & {\epsilon_{\text{PCA}}} C \\
{\epsilon_{\text{PCA}}} C^T & {\epsilon_{\text{PCA}}} B \\
\end{array}
\right] \left[
\begin{array}{l}
w_k+{\epsilon^{3/2}_{\text{PCA}}}v_{3/2}+{\epsilon^2_{\text{PCA}}}v_{2}+O({\epsilon_{\text{PCA}}}^{5/2})\\
{\epsilon_{\text{PCA}}}z_{1}+{\epsilon^{3/2}_{\text{PCA}}} z_{3/2}+O({\epsilon_{\text{PCA}}}^{2}) \\
\end{array}
\right]\nonumber\\
&& \label{lhs}\\
&=&\left[
\begin{array}{l}
w_k+{\epsilon_{\text{PCA}}}^{1/2}Aw_k+{\epsilon_{\text{PCA}}}^{3/2}v_{3/2}+{\epsilon_{\text{PCA}}}^{2}(Av_{3/2}+v_{2}+Cz_{1})+O({\epsilon_{\text{PCA}}}^{5/2})\\
{\epsilon_{\text{PCA}}}C^Tw_k+{\epsilon_{\text{PCA}}}^{2}Bz_{1}+O({\epsilon_{\text{PCA}}}^{5/2}) \\
\end{array}
\right], \nonumber
\end{eqnarray}
where $v_{3/2},v_{2}\in\RR^d$ and $z_{1},z_{3/2}\in\RR^{p-d}$. On the other hand,
\begin{eqnarray}
&&(1+{\epsilon_{\text{PCA}}}^{1/2}\lambda^A_k+{\epsilon_{\text{PCA}}}^{2}\lambda_{2}+O({\epsilon_{\text{PCA}}}^{5/2}))\left[
\begin{array}{l}
w_k+{\epsilon_{\text{PCA}}}^{3/2}v_{3/2}+{\epsilon_{\text{PCA}}}^{2}v_{2}+O({\epsilon_{\text{PCA}}}^{5/2})\\
{\epsilon_{\text{PCA}}}z_{1}+{\epsilon_{\text{PCA}}}^{3/2} z_{3/2}+O({\epsilon_{\text{PCA}}}^{2}) \\
\end{array}
\right]\nonumber\\
&& \label{rhs} \\
&=&\left[
\begin{array}{l}
w_k+{\epsilon_{\text{PCA}}}^{1/2}\lambda_k^A w_k+{\epsilon_{\text{PCA}}}^{3/2}v_{3/2}+{\epsilon_{\text{PCA}}}^{2}(\lambda_k^A v_{2}+v_{3/2}+\lambda_{2}w_k)+O({\epsilon_{\text{PCA}}}^{5/2})\\
{\epsilon_{\text{PCA}}}z_{1}+{\epsilon_{\text{PCA}}}^{3/2}(\lambda^A_kz_{1}+z_{3/2})+O({\epsilon_{\text{PCA}}}^{2}) \\
\end{array}
\right],\nonumber
\end{eqnarray}
where $\lambda_{2}\in\RR$. Matching orders of ${\epsilon_{\text{PCA}}}$ between (\ref{lhs}) and (\ref{rhs}), we conclude that
\begin{eqnarray}
O({\epsilon_{\text{PCA}}}):\quad & z_{1}&=C^Tw_k, \nonumber \\
O({\epsilon_{\text{PCA}}}^{3/2}):\quad & z_{3/2}&=-\lambda_k^Az_{1}, \nonumber \\
O({\epsilon_{\text{PCA}}}^{2}):\quad & (A-\lambda_k^A I)v_{3/2}&=\lambda_{2}w_k-CC^Tw_k. \label{ACCT}
\end{eqnarray}
Note that the matrix $(A-\lambda_k^A I)$ appearing in (\ref{ACCT}) is singular and its null space is spanned by the vector $w_k$, so the solvability condition is $\lambda_{2}=\|C^Tw_k\|^2/\|w_k\|^2$. We mention that $A$ is a generic symmetric matrix generated due to random finite sampling, so almost surely the eigenvalue $\lambda_k^A$ is simple.

Denote $O_i$ the $p\times d$ matrix whose $k$-th column is the vector $u_k$. We measure the deviation of the $d$-dim subspace of $\RR^p$ spanned by $u_k$, $k=1,\ldots,d$, from $\iota_*T_{x_i}\MM$ by
\begin{equation}
\min_{O\in O(d)} \|O_i^T \Theta_i - O\|_{HS},
\end{equation}
where $\Theta_i$ is a $p\times d$ matrix whose $k$-th column is $v_k$ (recall that $v_k$ is the $k$-th standard unit vector in $\mathbb{R}^p$). Let $\hat{O}$ be the $d\times d$ orthonormal matrix
\[
\hat{O}=\left[
\begin{array}{c}
 w_1^T \\
  \vdots  \\
 w_d^T
\end{array}
\right]_{d\times d}.
\]
Then,
\begin{equation}
\min_{O\in O(d)} \|O_i^T \Theta_i - O\|_{HS} \leq \|O_i^T \Theta_i - \hat{O}\|_{HS} = O({\epsilon_{\text{PCA}}}^{3/2}),
\end{equation}
which completes the proof for points away from the boundary.


Next, we consider $x_i\in \MM_{\sqrt{\epsilon_{\text{PCA}}}}$. The proof is almost the same as the above, so we just point out the main differences without giving the full details. The notations $\Xi_i$, $F_j(k,l)$, $p_{k,l}(n,\alpha)$, $Y_j(k,l)$ refer to the same quantities. Here the expectation of $F_j(k,l)$ is:
\begin{align}\label{bdryEFkl}
\begin{split}
\EE F(k,l)=\int_{B_{\sqrt{\epsilon_{\text{PCA}}}}(x_i)\cap \MM} K_{\epsilon_{\text{PCA}}}(x_i,y)\langle\iota(y)-\iota(x_i),v_k\rangle\langle \iota(y)-\iota(x_i),v_l\rangle p(y)
\ud V(y).
\end{split}
\end{align}
Due to the asymmetry of the integration domain $\exp_{x_i}^{-1}(B_{\sqrt{\epsilon_{\text{PCA}}}}(x_i)\cap \MM)$ when $x_i$ is near the boundary, we do not expect $\EE F_j(k,l)$ to be the same as (\ref{firstmomentFjkl}) and (\ref{secondmoment}), since integrals involving odd powers of $\theta$ do not vanish. In particular, when $l=d+1,\ldots,p$, $k=1,\ldots,d$ or $k=d+1,\ldots,p$, $l=1,\ldots,d$, (\ref{bdryEFkl}) becomes
\begin{eqnarray}
\EE F(k,l)&=&\int_{\exp_{x_i}^{-1}(B_{\sqrt{\epsilon_{\text{PCA}}}}(x_i)\cap \MM)} K\left(\frac{t}{\sqrt{\epsilon_{\text{PCA}}}}\right)\langle \iota_*\theta, v_k\rangle \langle \Pi(\theta,\theta), v_l\rangle p(x_i)t^{d+2}\ud t \ud \theta+O({\epsilon_{\text{PCA}}}^{d/2+2}) \nonumber \\
&=&O({\epsilon_{\text{PCA}}}^{d/2+3/2}).
\end{eqnarray}
Note that for $x_i\in\MM_{\sqrt{\epsilon_{\text{PCA}}}}$ the bias term in the expansion of the covariance matrix differs from (\ref{covariance-bias-variance2}) when $l=d+1,\ldots,p$, $k=1,\ldots,d$ or $k=d+1,\ldots,p$, $l=1,\ldots,d$. Similar calculations show that
\begin{equation}\label{firstmomentbdry}
\EE F(k,l)=\left\{
\begin{array}{ll}
O({\epsilon_{\text{PCA}}}^{d/2+1})&\mbox{ when }k,l=1,\ldots,d,\\
O({\epsilon_{\text{PCA}}}^{d/2+2})&\mbox{ when }k,l=d+1,\ldots,p,\\
O({\epsilon_{\text{PCA}}}^{d/2+3/2})&\mbox{ otherwise, }
\end{array}
\right.
\end{equation}
\begin{equation}\label{secondmomentbdry}
\EE [F(k,l)^2]=\left\{
\begin{array}{ll}
O({\epsilon_{\text{PCA}}}^{d/2+2})&\mbox{ when }k,l=1,\ldots,d,\\
O({\epsilon_{\text{PCA}}}^{d/2+4})&\mbox{ when }k,l=d+1,\ldots,p,\\
O({\epsilon_{\text{PCA}}}^{d/2+3})&\mbox{ otherwise, }
\end{array}
\right.
\end{equation}
and
\begin{equation}
\operatorname{Var} F(k,l) =\left\{
\begin{array}{ll}
O({\epsilon_{\text{PCA}}}^{d/2+2})&\mbox{ when }k,l=1,\ldots,d,\\
O({\epsilon_{\text{PCA}}}^{d/2+4})&\mbox{ when }k,l=d+1,\ldots,p,\\
O({\epsilon_{\text{PCA}}}^{d/2+3})&\mbox{ otherwise. }
\end{array}
\right.
\end{equation}
Similarly, $Y_j(k,l)$ are also bounded random variables satisfying
\begin{equation}
Y_j(k,l) = \left\{\begin{array}{ll}
                    O({\epsilon_{\text{PCA}}}) & \mbox{for } k,l=1,\ldots,d,\\
                    O({\epsilon_{\text{PCA}}}^2) & \mbox{for } k,l=d+1,\ldots,p, \\
                    O({\epsilon_{\text{PCA}}}^{3/2}) & \mbox{otherwise}.
                  \end{array}
 \right.
\end{equation}
Consider first the case $k,l=1,\ldots,d$, for which Bernstein's inequality gives
\begin{align}\begin{split}\label{pklbdry}
p_{k,l}(n,\alpha)\leq\exp\left\{-\frac{(n-1)\alpha^2}{O({\epsilon_{\text{PCA}}}^{d/2+2})+O({\epsilon_{\text{PCA}}})\alpha}\right\},
\end{split}\end{align}
From (\ref{pklbdry}) it follows that w.h.p.
\[
\alpha =  O\left(\frac{{\epsilon_{\text{PCA}}}^{d/4+1}}{n^{1/2}}\right),
\]
provided (\ref{condition-n-epsilon}). Similarly, for $k,l=d+1,\ldots,p$, we have
\begin{align*}\begin{split}
p_{k,l}(n,\alpha)\leq \exp\left\{-\frac{(n-1)\alpha^2}{O({\epsilon_{\text{PCA}}}^{d/2+4})+O({\epsilon_{\text{PCA}}}^2)\alpha}\right\},
\end{split}\end{align*}
which means that w.h.p.
\[
\alpha = O\left(\frac{{\epsilon_{\text{PCA}}}^{d/4+2}}{n^{1/2}}\right)
\]
provided (\ref{condition-n-epsilon}).
Finally, for $k=d+1,\ldots,p$, $l=1,\ldots,d$ or $l=d+1,\ldots,,p$, $k=1,\ldots,d$, we have
\begin{align*}\begin{split}
p_{k,l}(n,\alpha)\leq \exp\left\{-\frac{(n-1)\alpha^2}{O({\epsilon_{\text{PCA}}}^{d/2+3})+O({\epsilon_{\text{PCA}}}^{3/2})\alpha}\right\},
\end{split}\end{align*}
which means that w.h.p.
\[
\alpha = O\left(\frac{{\epsilon_{\text{PCA}}}^{d/4+3/2}}{n^{1/2}}\right)
\]
provided (\ref{condition-n-epsilon}).
As a result, under the condition in the statement of the theorem for the sampling rate, namely, ${\epsilon_{\text{PCA}}}=O(n^{-\frac{2}{d+2}})$, we have w.h.p.
\begin{eqnarray}
\Xi_i &=& {\epsilon_{\text{PCA}}}^{d/2+1}\left[
\begin{array}{ll}
O(1) & 0 \\
0 & 0 \\
\end{array}
\right]\nonumber\\
&&+ {\epsilon_{\text{PCA}}}^{d/2+3/2}\left[\begin{array}{ll}
                                                                      O(1) & O(1) \\
                                                                      O(1) & O({\epsilon_{\text{PCA}}}^{1/2})
                                                                    \end{array}\right] + \frac{{\epsilon_{\text{PCA}}}^{d/4+1}}{\sqrt{n}}\left[\begin{array}{ll}
                                                                      O(1) & O({\epsilon_{\text{PCA}}}^{1/2}) \\
                                                                      O({\epsilon_{\text{PCA}}}^{1/2}) & O({\epsilon_{\text{PCA}}})
                                                                    \end{array}
 \right]\nonumber\\
&=& {\epsilon_{\text{PCA}}}^{d/2+1}\left\{\left[
\begin{array}{ll}
O(1) & 0_{d\times p-d} \\
0_{p-d\times d} & 0_{p-d\times p-d} \\
\end{array}
\right] +  \left[\begin{array}{ll}
                                                                      O({\epsilon_{\text{PCA}}}^{1/2}) & O({\epsilon_{\text{PCA}}}^{1/2}) \\
                                                                      O({\epsilon_{\text{PCA}}}^{1/2}) & O({\epsilon_{\text{PCA}}})
                                                                    \end{array}
 \right]\right\}.\nonumber
\end{eqnarray}
Then, by the same argument as in the case when $x_i\notin\MM_{\sqrt{\epsilon_{\text{PCA}}}}$, we conclude that
\[
\min_{O\in O(d)} \|O_i^T \Theta_i - O\|_{HS}=O({\epsilon_{\text{PCA}}}^{1/2}).
\]

Similar calculations show that for $\epsilon_{\text{PCA}} = O(n^{-\frac{2}{d+1}})$ we get
\[
\min_{O\in O(d)} \|O_i^T \Theta_i - O\|_{HS}=O({\epsilon_{\text{PCA}}}^{5/4})
\]
for $x_i\notin\MM_{\sqrt{\epsilon_{\text{PCA}}}}$, and
\[
\min_{O\in O(d)} \|O_i^T \Theta_i - O\|_{HS}=O({\epsilon_{\text{PCA}}}^{3/4})
\]
for $x_i\in\MM_{\sqrt{\epsilon_{\text{PCA}}}}$.

\end{proof}

\subsection{[Proof of Theorem \ref{relateOP}]}

\begin{proof}
Denote by $O_i$ the $p\times d$ matrix whose columns $u_l(x_i)$, $l=1,\ldots,d$ are orthonormal inside $\RR^p$ as determined by local PCA around $x_i$. As in (\ref{hatQ}), we denote by $e_l(x_i)$ the $l$-th column of $Q_i$, where $Q_i$ is a $p\times d$ matrix whose columns form an orthonormal basis of $\iota_* T_{x_i}\MM$ so by Theorem \ref{localpcatheorem} $\|O_i^TQ_i-Id\|_{HS}=O(\epsilon_{\text{PCA}}^{3/2})$ for $\epsilon_{\text{PCA}} = O(n^{-\frac{2}{d+2}})$, which is the case of focus here (if $\epsilon_{\text{PCA}} = O(n^{-\frac{2}{d+1}})$ then $\|O_i^TQ_i-Id\|_{HS}=O(\epsilon_{\text{PCA}}^{5/4})$).

Fix $x_i$ and the normal coordinate $\{\partial_l\}_{l=1}^d$ around $x_i$ so that $\iota_*\partial_l(x_i)=e_l(x_i)$. Let $x_j=\exp_{x_i}t\theta$, where $\theta\in T_{x_i}\MM$, $\|\theta\|=1$ and $t=O(\sqrt{\epsilon})$. Then, by the definition of the parallel transport, we have
\begin{equation}\label{expandp}
P_{x_i,x_j}X(x_j)=\sum_{l=1}^dg(X(x_j),P_{x_j,x_i}\partial_l(x_i))\partial_l(x_i)
\end{equation}
and since the parallel transport and the embedding $\iota$ are isometric we have
\begin{equation}
g(P_{x_i,x_j}X(x_j),\partial_l(x_i))=g(X(x_j),P_{x_j,x_i}\partial_l(x_i))=\langle \iota_*X(x_j),\iota_*P_{x_j,x_i}\partial_l(x_i)\rangle.
\end{equation}

Local PCA provides an estimation of an orthonormal basis spanning $\iota_*T_{x_i}\MM$, which is free up to $O(d)$. Thus, there exists $R\in O(p)$ so that $\iota_*T_{x_j}\MM$ is invariant under $R$ and $e_l(x_j)=R\iota_*P_{x_j,x_i}\partial_l(x_i)$ for all $l=1,\ldots,d$. Hence we have the following relationship:
\begin{align}\label{p49}
\begin{split}
&\quad\langle \iota_*X(x_j),\iota_*P_{x_j,x_i}\partial_l(x_i)\rangle=\langle \sum_{k=1}^d \langle\iota_*X(x_j),e_k(x_j)\rangle e_k(x_j),\iota_*P_{x_j,x_i}\partial_l(x_i)\rangle\\
&=\sum_{k=1}^d\langle \iota_*X(x_j),e_k(x_j)\rangle\langle e_k(x_j),\iota_*P_{x_j,x_i}\partial_l(x_i)\rangle\\
&=\sum_{k=1}^d \langle R\iota_*P_{x_j,x_i}\partial_k(x_i),\iota_*P_{x_j,x_i}\partial_l(x_i)\rangle\langle \iota_*X(x_j),e_k(x_j)\rangle
\\&:=\sum_{k=1}^d \bar{R}_{l,k}\langle \iota_*X(x_j),e_k(x_j)\rangle :=\bar{R}X_j,
\end{split}
\end{align}
where $\bar{R}_{l,k}:=\langle R\iota_*P_{x_j,x_i}\partial_k(x_i),\iota_*P_{x_j,x_i}\partial_l(x_i)\rangle$, $\bar{R}:=[\bar{R}_{l,k}]_{l,k=1}^d$ and $X_j=(\langle \iota_*X(x_j),e_k(x_j)\rangle)_{k=1}^d$.

On the other hand, Lemma \ref{relatebasis} gives us
\begin{align}\label{oioj}
\begin{split}
Q^T_iQ_j&=\Big[\iota_*\partial_l(x_i)^TR\iota_*P_{x_j,x_i}\partial_k(x_i)\Big]_{l,k=1}^d\\
&=\Big[\iota_*P_{x_j,x_i}\partial_l(x_i)^TR\iota_*P_{x_j,x_i}\partial_k(x_i)\Big]_{l,k=1}^d\\
&\quad-t\Big[\Pi(\theta,\partial_l(x_i))^TR\iota_*P_{x_j,x_i}\partial_k(x_i)\Big]_{l,k=1}^d\\
&\quad-\frac{t^2}{6}\Big[2\nabla_{\partial_l(x_i)}\Pi(\theta,\theta)^TR\iota_*P_{x_j,x_i}\partial_k(x_i)
+\nabla_{\theta}\Pi(\theta,\partial_l(x_i))^T R\iota_*P_{x_j,x_i}\partial_k(x_i)\\
&\quad\quad\quad-(\iota_*P_{x_j,x_i}\mathcal{R}(\theta,\partial_l(x_i))\theta)^TR\iota_*P_{x_j,x_i}\partial_k(x_i)\Big]_{l,k=1}^d\\
&\quad+O(t^{3}).
\end{split}
\end{align}
We now analyze the right hand side of (\ref{oioj}) term by term.
Note that since $\iota_*T_{x_j}\MM$ is invariant under $R$, we have $R\iota_*P_{x_j,x_i}\partial_k(x_i)=\sum_{r=1}^d\bar{R}_{r,k}\iota_*P_{x_j,x_i}\partial_r(x_i)$. For the $O(t)$ term, we have
\begin{align}\label{firstorder}
\begin{split}
&\quad\Pi(\theta,\partial_l(x_i))^TR\iota_*P_{x_j,x_i}\partial_k(x_i)=\sum_{r=1}^d\bar{R}_{r,k}\Pi(\theta,\partial_l(x_i))^T\iota_*P_{x_j,x_i}\partial_r(x_i)\\
&=\sum_{r=1}^d\bar{R}_{r,k}\Pi(\theta,\partial_l(x_i))^T\left[\iota_*\partial_r(x_i)+t\Pi(\theta,\partial_r(x_i))+O(t^2)\right]\\
&=t\sum_{r=1}^d\bar{R}_{r,k}\Pi(\theta,\partial_l(x_i))^T\Pi(\theta,\partial_r(x_i))+O(t^2)
\end{split}
\end{align}
where the second equality is due to Lemma \ref{relatebasis} and the third equality holds since $\Pi(\theta,\partial_l(x_i))$ is perpendicular to $\iota_*\partial_r(x_i)$ for all $l,r=1,\ldots,d$. Moreover, Gauss equation gives us
\[
0=\langle \mathcal{R}(\theta,\theta)\partial_r(x_i),\partial_l(x_i)\rangle=\Pi(\theta,\partial_l(x_i))^T\Pi(\theta,\partial_r(x_i))-\Pi(\theta,\partial_r(x_i))^T\Pi(\theta,\partial_l(x_i)),
\]
which means the matrix $S_1:=\Big[\Pi(\theta,\partial_l(x_i))^T\Pi(\theta,\partial_r(x_i))\Big]_{l,r=1}^d$ is symmetric.

Fix a vector field $X$ on a neighborhood around $x_i$ so that $X(x_i)=\theta$. By definition we have
\begin{equation}\label{weingartenx}
\nabla_{\partial_l}\Pi(X,X)=\nabla_{\partial_l}(\Pi(X,X))-2\Pi(X,\nabla_{\partial_l}X)
\end{equation}
Viewing $T\MM$ as a subbundle of $T\RR^p$, we have the equation of Weingarten:
\begin{equation}\label{weingartenxx}
\nabla_{\partial_l}(\Pi(X,X))=-A_{\Pi(X,X)}\partial_l+\nabla^\perp_{\partial_l}(\Pi(X,X)),
\end{equation}
where $A_{\Pi(X,X)}\partial_l$ and $\nabla^\perp_{\partial_l}(\Pi(X,X))$ are the tangential and normal components of $\nabla_{\partial_l}(\Pi(X,X))$ respectively. Moreover, the following equation holds:
\begin{equation}\label{weingarten}
\langle A_{\Pi(X,X)}\partial_l,\iota_*\partial_k\rangle=\langle \Pi(\partial_l,\partial_k),\Pi(X,X)\rangle.
\end{equation}
By evaluating (\ref{weingartenx}) and (\ref{weingartenxx}) at $x_i$, we have
\begin{align}\label{firstorder3}
\begin{split}
&\quad\quad\nabla_{\partial_l(x_i)}\Pi(\theta,\theta)^TR\iota_*P_{x_j,x_i}\partial_k(x_i)
=\sum_{r=1}^d\bar{R}_{r,k}\nabla_{\partial_l(x_i)}\Pi(\theta,\theta)^T\iota_*P_{x_j,x_i}\partial_r(x_i)\\
&=\sum_{r=1}^d\bar{R}_{r,k}(-A_{\Pi(\theta,\theta)}\partial_l(x_i)+\nabla^\perp_{\partial_l(x_i)}(\Pi(X,X))-2\Pi(\theta,\nabla_{\partial_l(x_i)}X))^T\left[\iota_*\partial_r(x_i)+t\Pi(\theta,\partial_r(x_i))+O(t^2)\right]\\
&=-\sum_{r=1}^d\bar{R}_{r,k}(A_{\Pi(\theta,\theta)}\partial_l(x_i))^T\iota_*\partial_r(x_i)+O(t)
=-\sum_{r=1}^d\bar{R}_{r,k}\langle \Pi(\theta,\theta),\Pi(\partial_l(x_i),\partial_r(x_i))\rangle+O(t).
\end{split}
\end{align}
where the third equality holds since $\Pi(\theta,\nabla_{\partial_l(x_i)}X)$ and $\nabla^\perp_{\partial_l(x_i)}(\Pi(X,X))$ are perpendicular to $\iota_*\partial_l(x_i)$ and the last equality holds by (\ref{weingarten}). Due to the symmetry of the second fundamental form, we know the matrix $S_2=\Big[\langle \Pi(\theta,\theta),\Pi(\partial_l(x_i),\partial_r(x_i))\rangle\Big]_{l,r=1}^d$ is symmetric.

Similarly we have
\begin{align}\label{firstorder35}
\begin{split}
\nabla_{\theta}\Pi(\partial_l(x_i),\theta)^TR\iota_*P_{x_j,x_i}\partial_k(x_i)=\sum_{r=1}^d\bar{R}_{r,k}(A_{\Pi(\partial_l(x_i),\theta)}\theta)^T\iota_*\partial_r(x_i)+O(t).
\end{split}
\end{align}
Since
$(A_{\Pi(\partial_l(x_i),\theta)}\theta)^T\iota_*\partial_r(x_i)=\Pi(\theta,\partial_l(x_i))^T\Pi(\theta,\partial_r(x_i))$ by (\ref{weingarten}), which we denoted earlier by $S_1$ and used Gauss equation to conclude that it is symmetric.

To estimate the last term, we work out the following calculation by using the isometry of the parallel transport:
\begin{align}\label{firstorder2}
\begin{split}
&\quad\quad(\iota_*P_{x_j,x_i}\mathcal{R}(\theta,\partial_l(x_i))\theta)^TR\iota_*P_{x_j,x_i}\partial_k(x_i)
=\sum_{r=1}^d\bar{R}_{r,k}\langle\iota_*P_{x_j,x_i}(\mathcal{R}(\theta,\partial_l(x_i))\theta),\iota_*P_{x_j,x_i}\partial_r(x_i)\rangle\\
&=\sum_{r=1}^d\bar{R}_{r,k}g(P_{x_j,x_i}(\mathcal{R}(\theta,\partial_l(x_i))\theta),P_{x_j,x_i}\partial_r(x_i))
=\sum_{r=1}^d\bar{R}_{r,k}g(\mathcal{R}(\theta,\partial_l(x_i))\theta,\partial_r(x_i))
\end{split}
\end{align}
Denote $S_3=\Big[g(\mathcal{R}(\theta,\partial_l(x_i))\theta,\partial_r(x_i))\Big]_{l,r=1}^d$, which is symmetric by the definition of $\mathcal{R}$.

Substituting (\ref{firstorder}), (\ref{firstorder3}), (\ref{firstorder35}) and (\ref{firstorder2}) into (\ref{oioj}) we have
\begin{align}\label{oioj3}
\begin{split}
Q^T_iQ_j&=\bar{R}+t^2(-S_1-S_2/3+S_1/6-S_3/6)\bar{R}+O(t^3)=\bar{R}+t^2 S\bar{R}+O(t^3),
\end{split}
\end{align}
where $S:=-S_1-S_2/3+S_1/6-S_3/6$ is a symmetric matrix.

Suppose that both $x_i$ and $x_j$ are not in $\MM_{\sqrt{\epsilon}}$. To finish the proof, we have to understand the relationship between $O_i^TO_j$ and $Q_i^TQ_j$, which is rewritten as:
\begin{equation}\label{oiojqiqj}
O_i^TO_j= Q_i^TQ_j+(O_i-Q_i)^TQ_j+O_i^T(O_j-Q_j).
\end{equation}
From (\ref{minOiQi}) in Theorem \ref{localpcatheorem}, we know
\begin{equation*}
\|(O_i-Q_i)^TQ_i\|_{HS}=\|O_i^TQ_i-Id\|_{HS}=O(\epsilon_{\text{PCA}}^{3/2}),
\end{equation*}
which is equivalent to
\begin{equation}\label{differenceorder}
(O_i-Q_i)^TQ_i=O(\epsilon_{\text{PCA}}^{3/2}).
\end{equation}
Due to (\ref{oioj3}) we have $Q_j=Q_i\bar{R}+t^2 Q_iS\bar{R}+O(t^3)$, which together with (\ref{differenceorder}) gives
\begin{equation}\label{ojt}
(O_i-Q_i)^TQ_j=(O_i-Q_i)^T(Q_i\bar{R}+t^2 Q_iS\bar{R}+O(t^3))=O(\epsilon_{\text{PCA}}^{3/2}+\epsilon^{3/2}).
\end{equation}
Together with the fact that $Q_i^T=\bar{R}Q_j^T+t^2 S\bar{R}Q_j^T+O(t^3)$ derived from (\ref{oioj3}), we have
\begin{eqnarray}\label{ojt2}
O_i^T(O_j-Q_j)&=&Q_i^T(O_j-Q_j)+(O_i-Q_i)^T(O_j-Q_j)\label{oit}\\
&=&(\bar{R}Q_j^T+t^2 S\bar{R}Q_j^T+O(t^3))(O_j-Q_j)+(O_i-Q_i)^T(O_j-Q_j)\nonumber\\
&=&O(\epsilon_{\text{PCA}}^{3/2}+\epsilon^{3/2})+(O_i-Q_i)^T(O_j-Q_j)\nonumber
\end{eqnarray}
Recall that the following relationship between $O_i$ and $Q_i$ holds (\ref{uk-wk})
\begin{equation}\label{oiqi}
O_i=Q_i+\left[\begin{array}{c}O(\epsilon_{\text{PCA}}^{3/2})\\ O(\epsilon_{\text{PCA}})\end{array}\right]
\end{equation}
when the embedding $\iota$ is properly translated and rotated so that it satisfies $\iota(x_i)=0$, the first $d$ standard unit vectors $\{v_1,\ldots,v_d\}\subset\RR^p$ form the orthonormal basis of $\iota_*T_{x_i}\MM$, and the normal coordinates $\{\partial_k\}_{k=1}^d$ around $x_i$ satisfy $\iota_*\partial_k(x_i)=v_k$. Similarly, the following relationship between $O_j$ and $Q_j$ holds (\ref{uk-wk})
\begin{equation}\label{ojqj}
O_j=Q_j+\left[\begin{array}{c}O(\epsilon_{\text{PCA}}^{3/2})\\ O(\epsilon_{\text{PCA}})\end{array}\right]
\end{equation}
when the embedding $\iota$ is properly translated and rotated so that it satisfies $\iota(x_j)=0$, the first $d$ standard unit vectors $\{v_1,\ldots,v_d\}\subset\RR^p$ form the orthonormal basis of $\iota_*T_{x_j}\MM$, and the normal coordinates $\{\partial_k\}_{k=1}^d$ around $x_j$ satisfy $\iota_*\partial_k(x_j)=\bar{R}\iota_*P_{x_j,x_i}\partial_k(x_i)=v_k$. Also recall that $\iota_*T_{x_j}\MM$ is invariant under the rotation $R$ and from Lemma \ref{relatebasis}, $e_k(x_i)$ and $e_k(x_j)$ are related by $e_k(x_j)=e_k(x_i)+O(\sqrt{\epsilon})$. Therefore,
\begin{equation}\label{relatedback}
O_j-Q_j=(R+O(\sqrt{\epsilon}))\left[\begin{array}{c}O(\epsilon_{\text{PCA}}^{3/2})\\ O(\epsilon_{\text{PCA}})\end{array}\right]=\left[\begin{array}{c}O(\epsilon_{\text{PCA}}^{3/2})\\ O(\epsilon_{\text{PCA}})\end{array}\right]
\end{equation}
when expressed in the standard basis of $\RR^p$ so that the first $d$ standard unit vectors $\{v_1,\ldots,v_d\}\subset\RR^p$ form the orthonormal basis of $\iota_*T_{x_i}\MM$. Hence, plugging (\ref{relatedback}) into (\ref{ojt2}) gives
\begin{eqnarray}
O_i^T(O_j-Q_j)=O(\epsilon_{\text{PCA}}^{3/2})+(O_i-Q_i)^T(O_j-Q_j)=O(\epsilon_{\text{PCA}}^{3/2})\label{ojt3}
\end{eqnarray}
Inserting (\ref{oit}) and (\ref{ojt3}) into (\ref{oiojqiqj}) concludes
\begin{equation}
O_i^TO_j=Q_i^TQ_j+O(\epsilon_{\text{PCA}}^{3/2}+\epsilon^{3/2}).
\end{equation}

Recall that $O_{ij}$ is defined as $O_{ij}=UV^T$, where $U$ and $V$ comes from the singular value decomposition of $O_i^TO_j$, that is, $O_i^TO_j=U\Sigma V^T$. As a result,
\begin{align*}
\begin{split}
O_{ij}&=\argmin_{O\in O(d)}\left\|O_i^TO_j-O\right\|_{HS}=\argmin_{O\in O(d)}\left\|Q_i^TQ_j+O(\epsilon_{\text{PCA}}^{3/2}+\epsilon^{3/2})-O\right\|_{HS}\\
&=\argmin_{O\in O(d)}\|\bar{R}^TQ_i^TQ_j+O(\epsilon_{\text{PCA}}^{3/2}+\epsilon^{3/2})-\bar{R}^TO\|_{HS}\\
&=\argmin_{O\in O(d)}\|Id+t^2\bar{R}^TS\bar{R}+O(\epsilon_{\text{PCA}}^{3/2}+\epsilon^{3/2})-\bar{R}^TO\|_{HS}.
\end{split}
\end{align*}
Since $\bar{R}^TS\bar{R}$ is symmetric, we rewrite $\bar{R}^TS\bar{R}=U\Sigma U^T$, where $U$ is an orthonormal matrix and $\Sigma$ is a diagonal matrix with the eigenvalues of $\bar{R}^TS\bar{R}$ on its diagonal. Thus, $$Id+t^2\bar{R}^TS\bar{R}+O(\epsilon_{\text{PCA}}^{3/2}+\epsilon^{3/2})-\bar{R}^TO=U(Id+t^2\Sigma)U^T+O(\epsilon_{\text{PCA}}^{3/2}+\epsilon^{3/2})-\bar{R}^TO.$$ Since the Hilbert-Schmidt norm is invariant to orthogonal transformations, we have
\begin{eqnarray*}
\|Id+t^2\bar{R}^TS\bar{R}+O(\epsilon_{\text{PCA}}^{3/2}+\epsilon^{3/2})-\bar{R}^TO \|_{HS} &=& \|U(Id+t^2\Sigma)U^T+O(\epsilon_{\text{PCA}}^{3/2}+\epsilon^{3/2})-\bar{R}^TO \|_{HS} \\
&=& \|Id + t^2\Sigma + O(\epsilon_{\text{PCA}}^{3/2}+\epsilon^{3/2})-U^T\bar{R}^TOU \|_{HS}.
\end{eqnarray*}
Since $U^T\bar{R}^TOU$ is orthogonal, the minimizer must satisfy $U^T\bar{R}^TOU=Id + O(\epsilon_{\text{PCA}}^{3/2}+\epsilon^{3/2})$, as otherwise the sum of squares of the matrix entries would be larger. Hence we conclude $O_{ij}=\bar{R}+O(\epsilon_{\text{PCA}}^{3/2}+\epsilon^{3/2})$.

Applying (\ref{p49}) and (\ref{uk-wk}), we conclude
\begin{align*}
\begin{split}
O_{ij}\bar{X}_j&=\bar{R}\bar{X}_j+O(\epsilon_{\text{PCA}}^{3/2}+\epsilon^{3/2})=\bar{R}(\langle \iota_*X(x_j),u_l(x_j)\rangle)_{l=1}^d+O(\epsilon_{\text{PCA}}^{3/2}+\epsilon^{3/2})\\
&=\bar{R}(\langle \iota_*X(x_j),e_l(x_j)+O(\epsilon_{\text{PCA}}^{3/2})\rangle)_{l=1}^d+O(\epsilon_{\text{PCA}}^{3/2}+\epsilon^{3/2}) \\
&=(\langle \iota_*X(x_j),\iota_*P_{x_j,x_i}\partial_l(x_i)\rangle)_{l=1}^d+O(\epsilon_{\text{PCA}}^{3/2}+\epsilon^{3/2})\\
&=\left(\langle \iota_*P_{x_i,x_j}X(x_j),e_l(x_i)\rangle\right)^d_{l=1}+O(\epsilon_{\text{PCA}}^{3/2}+\epsilon^{3/2})\\
&=\left(\langle \iota_*P_{x_i,x_j}X(x_j),u_l(x_i)\rangle\right)^d_{l=1}+O(\epsilon_{\text{PCA}}^{3/2}+\epsilon^{3/2})
\end{split}
\end{align*}
This concludes the proof for points away from the boundary.

When $x_i$ and $x_j$ are in $\MM_{\sqrt{\epsilon_{\text{PCA}}}}$, by the same reasoning as above we get
\begin{align*}
\begin{split}
O_{ij}\bar{X}_j=\left(\langle \iota_*P_{x_i,x_j}X(x_j),u_l(x_i)\rangle\right)^d_{l=1}+O(\epsilon_{\text{PCA}}^{1/2}+\epsilon^{3/2}).
\end{split}
\end{align*}
This concludes the proof. We remark that similar results hold for $\epsilon_{\text{PCA}} = O(n^{-\frac{2}{d+1}})$ using the results in Theorem \ref{localpcatheorem}. 
\end{proof}

\subsection{[Proof of Theorem \ref{convergetokernel}]}
\begin{proof}
We demonstrate the proof for the case when the data is uniformly distributed over the manifold. The proof for the non-uniform sampling case is the same but more tedious. Note that when the data is uniformly distributed, $T_{\epsilon,\alpha}=T_{\epsilon,0}$ for all $0<\alpha\leq 1$, so in the proof we focus on analyzing $T_\epsilon:=T_{\epsilon,0}$. Denote $K_\epsilon:=K_{\epsilon,0}$. Fix $x_i \notin \mathcal{M}_{\sqrt{\epsilon_{\text{PCA}}}}$. We rewrite the left hand side of (\ref{conv}) as
\begin{equation}\label{estimator}
\frac{\sum_{j=1,j\neq i}^n K_\epsilon\left(x_i,x_j\right)O_{ij}\bar{X}_j}{\sum_{j=1,j\neq i}^n K_\epsilon\left(x_i,x_j\right)}=\frac{\frac{1}{n-1}\sum_{j=1,j\neq i}^n F_j}{\frac{1}{n-1}\sum_{j=1,j\neq i}^n G_j},
\end{equation}
where
\[
F_j=K_\epsilon\left(x_i,x_j\right)O_{ij}\bar{X}_j,\ \ \ \ G_j=K_\epsilon\left(x_i,x_j\right).
\]

Since $x_1,x_2,\ldots,x_n$ are i.i.d random variables, then $G_j$ for $j\neq i$ are also i.i.d random variables. However, the random vectors $F_j$ for $j\neq i$ are not independent, because the computation of $O_{ij}$ involves several data points which leads to possible dependency between $O_{ij_1}$ and $O_{ij_2}$. Nonetheless, Theorem \ref{relateOP} implies that the random vectors $F_j$ are well approximated by the i.i.d random vectors $F_j'$ that are defined as
\begin{equation}
F_j' := K_\epsilon\left(x_i,x_j\right)\left(\langle \iota_*P_{x_i,x_j}X(x_j),u_l(x_i)\rangle\right)^d_{l=1},
\end{equation}
and the approximation is given by
\begin{equation}
F_j= F_j'+K_\epsilon (x_i,x_j) O(\epsilon_{\text{PCA}}^{3/2}+\epsilon^{3/2}),
\end{equation}
where we use $\epsilon_{\text{PCA}} = O(n^{-\frac{2}{d+2}})$ (the following analysis can be easily modified to adjust the case $\epsilon_{\text{PCA}} = O(n^{-\frac{2}{d+1}})$).

Since $G_j$, when $j\neq i$, are identical and independent random variables and $F_j'$, when $j\neq i$, are identical and independent random vectors, we hereafter replace $F_j'$ and $G_j$ by $F'$ and $G$ in order to ease notation. By the law of large numbers we should expect the following approximation to hold
\begin{equation}
\label{EEF}
\frac{\frac{1}{n-1}\sum_{j=1,j\neq i}^n F_j}{\frac{1}{n-1}\sum_{j=1,j\neq i}^n G_j} = \frac{\frac{1}{n-1}\sum_{j=1,j\neq i}^n [F_j' + G_j O(\epsilon_{\text{PCA}}^{3/2}+\epsilon^{3/2})]}{\frac{1}{n-1}\sum_{j=1,j\neq i}^n G_j}\approx \frac{\EE F'}{\EE G} + O(\epsilon_{\text{PCA}}^{3/2}+\epsilon^{3/2}),
\end{equation}
where
\begin{equation}
\EE F'=\left(\left\langle\iota_*\int_\MM K_{\epsilon}(x_i,y) P_{x_i,y}X(y)\ud V(y),u_l(x_i))\right\rangle\right)_{l=1}^d,
\end{equation}
and
\begin{equation}
\EE G=\int_\MM K_{\epsilon}(x_i,y)\ud V(y).
\end{equation}

In order to analyze the error of this approximation, we make use of the result in \cite{Singer2006a} (equation (3.14), p. 132) to conclude a large deviation bound on each of the $d$ coordinates of the error. Together with a simple union bound we obtain the following large deviation bound:
\begin{equation}
\Pr \left\{\left\|\frac{\frac{1}{n-1}\sum_{j=1,j\neq i}^n F_j'}{\frac{1}{n-1}\sum_{j=1,j\neq i}^n G_j}- \frac{\EE F'}{\EE G}\right\| > \alpha \right\}  \leq C_1
\exp\left\{-\frac{C_2(n-1)\alpha^2
\epsilon^{d/2}\mbox{vol}(\mathcal M)}{2\epsilon\left[\|\nabla|_{y=x_i} \langle \iota_* P_{x_i,y}X(y),u_l(x_i)\rangle  \|^2 + O(\epsilon)\right]}
\right\},
\end{equation}
where $C_1$ and $C_2$ are some constants (related to $d$). This large deviation bound implies that w.h.p. the variance term is $O(\frac{1}{n^{1/2}\epsilon^{d/4-1/2}})$. As a result,
\begin{eqnarray*}
\frac{\sum_{j=1,j\neq i}^n K_{\epsilon,\alpha}\left(x_i,x_j\right)O_{ij}\bar{X}_j}{\sum_{j=1,j\neq i}^n K_{\epsilon,\alpha}\left(x_i,x_j\right)} &=& (\langle \iota_*T_{\epsilon,\alpha} X(x_i),u_l(x_i)\rangle )^d_{l=1} \\
&& + O\left(\frac{1}{n^{1/2}\epsilon^{d/4-1/2}} + \epsilon_{\text{PCA}}^{3/2}+\epsilon^{3/2}\right),
\end{eqnarray*}
which completes the proof for points away from the boundary. The proof for points inside the boundary is similar.

\end{proof}

\subsection{[Proof of Theorem \ref{Tepsexpansion}]}
\begin{proof}
We begin the proof by citing the following Lemma from \cite[Lemma 8]{Coifman20065}:
\begin{lem}\label{Kf}
Suppose $f\in \mathcal{C}^3(\MM)$ and $x\notin \MM_{\sqrt{\epsilon}}$, then
\[
\int_{B_{\sqrt{\epsilon}}(x)} \epsilon^{-d/2}K_\epsilon(x,y)f(y)\ud V(y)=m_0f(x)+\epsilon\frac{m_2}{d}\left[\frac{\Delta f(x)}{2}+ w(x)f(x)\right]+O(\epsilon^{2})
\]
where $w(x)=s(x) +\frac{m'_3z(x)}{24|S^{d-1}|}$, $s(x)$ is the scalar curvature of the manifold at $x$, $m_l=\int_{B_1(0)}\|x\|^lK(\|x\|)\ud x$, $B_1(0)=\{x\in\RR^d: \|x\|_{\RR^d}\leq 1\}$, $m'_l=\int_{B_1(0)}\|x\|^lK'(\|x\|)\ud x$, $z(x)=\int_{S^{d-1}}\|\Pi(\theta,\theta)\|\ud \theta$, and $\Pi$ is the second fundamental form of $\MM$ at $x$.
\end{lem}

Without loss of generality we may assume that $m_0=1$ for convenience of notation. By Lemma \ref{Kf}, we get
\begin{equation}\label{pepsilon}
p_\epsilon(y)=p(y)+\epsilon\frac{m_2}{d}\left(\frac{\Delta p(y)}{2}+w(y)p(y)\right)+O(\epsilon^2),
\end{equation}
which leads to
\begin{equation}\label{approximatepalpha}
\frac{p(y)}{p^{\alpha}_\epsilon(y)}=p^{1-\alpha}(y)\left[1-\alpha\epsilon\frac{m_2}{d}\left(w(y)+\frac{\Delta p(y)}{2p(y)}\right)\right]+O(\epsilon^{2}).
\end{equation}
Plug (\ref{approximatepalpha}) into the numerator of $T_{\epsilon,\alpha}X(x)$:
\begin{align*}
\begin{split}
&\ \ \ \ \int_{B_{\sqrt{\epsilon}}(x)} K_{\epsilon,\alpha}(x,y)P_{x,y} X(y)p(y)\ud V(y)\\
&=p_\epsilon^{-\alpha}(x)\int_{B_{\sqrt{\epsilon}}(x)} K_{\epsilon}(x,y)P_{x,y} X(y)p_\epsilon^{-\alpha}(y)p(y)\ud V(y)\\
&=p_\epsilon^{-\alpha}(x)\int_{B_{\sqrt{\epsilon}}(x)} K_{\epsilon}(x,y)P_{x,y} X(y)p^{1-\alpha}(y)\left[1-\alpha\epsilon\frac{m_2}{d}\left(w(y)+\frac{\Delta p(y)}{2p(y)}\right)\right]\ud V(y)+O(\epsilon^{d/2+2})\\
&= p_\epsilon^{-\alpha}(x)\int_{B_{\sqrt{\epsilon}}(x)} K_{\epsilon}(x,y)P_{x,y} X(y)p^{1-\alpha}(y)\ud V(y)\\
&\ \ \ \ -\frac{m_2\epsilon}{d}\alpha p_\epsilon^{-\alpha}(x)\int_{B_{\sqrt{\epsilon}}(x)} K_{\epsilon}(x,y)P_{x,y} X(y)p^{1-\alpha}(y)\left(w(y)+\frac{\Delta p(y)}{2p(y)}\right)\ud V(y)+O(\epsilon^{d/2+2})\\
&:=p_\epsilon^{-\alpha}(x)A-\frac{m_2\epsilon}{d}\alpha p_\epsilon^{-\alpha}(x)B+O(\epsilon^{d/2+2})
\end{split}
\end{align*}
where
\[\left\{\begin{array}{l}A:=\int_{B_{\sqrt{\epsilon}}(x)} K_{\epsilon}(x,y)P_{x,y} X(y)p^{1-\alpha}(y)\ud V(y),\\
B:=\int_{B_{\sqrt{\epsilon}}(x)} K_{\epsilon}(x,y)P_{x,y} X(y)p^{1-\alpha}(y)\left(w(y)+\frac{\Delta p(y)}{2p(y)}\right)\ud V(y).\end{array}\right.\]
Note that the $\alpha$-normalized integral operator (\ref{TepsalphaDEF}) is evaluated by changing the integration variables to the local coordinates, and the odd monomials in the integral vanish because the kernel is symmetric. Thus, applying Taylor's expansion to $A$ leads to:
{\allowdisplaybreaks\begin{align*}
\begin{split}
A&=\int_{S^{d-1}}\int_0^{\sqrt{\epsilon}} \left[K\left(\frac{t}{\sqrt{\epsilon}}\right)+K'\left(\frac{t}{\sqrt{\epsilon}}\right)\frac{\|\Pi(\theta,\theta)\|t^3}{24\sqrt{\epsilon}}+O\left(\frac{t^6}{\epsilon}\right)\right]\times\\
&\quad\left[X(x)+\nabla_\theta X(x)t+\nabla^2_{\theta,\theta}X(x)\frac{t^2}{2}+O(t^3)\right]\times \\
&\quad \left[p^{1-\alpha}(x)+\nabla_\theta(p^{1-\alpha})(x)t+\nabla^2_{\theta,\theta}(p^{1-\alpha})(x)\frac{t^2}{2}+O(t^3)\right]
\left[t^{d-1}+\Ric(\theta,\theta)t^{d+1}+O(t^{d+2})\right]\ud t\ud\theta\\
&=p^{1-\alpha}(x)X(x)\int_{S^{d-1}}\int_0^{\sqrt{\epsilon}} \left\{K\left(\frac{t}{\sqrt{\epsilon}}\right)\left[1+\Ric(\theta,\theta)t^2\right]t^{d-1}+ K'\left(\frac{t}{\sqrt{\epsilon}}\right)\frac{\|\Pi(\theta,\theta)\|t^{d+2}}{24\sqrt{\epsilon}}\right\}\ud t\ud\theta\\
&\quad+p^{1-\alpha}(x)\int_{S^{d-1}}\int_0^{\sqrt{\epsilon}} K\left(\frac{t}{\sqrt{\epsilon}}\right)\nabla^2_{\theta,\theta}X(x)\frac{t^{d+1}}{2}\ud t\ud\theta\\
&\quad+X(x)\int_{S^{d-1}}\int_0^{\sqrt{\epsilon}} K\left(\frac{t}{\sqrt{\epsilon}}\right)\nabla^2_{\theta,\theta}(p^{1-\alpha})(x)\frac{t^{d+1}}{2}\ud t\ud\theta\\
&\quad+\int_{S^{d-1}}\int_0^{\sqrt{\epsilon}} K\left(\frac{t}{\sqrt{\epsilon}}\right)\nabla_{\theta}X(x)\nabla_\theta(p^{1-\alpha})(x)t^{d+1}\ud t\ud\theta+O(\epsilon^{d/2+2})
\end{split}
\end{align*}
}
From the definition of $z(x)$ it follows that
\[
\int_{B_{\sqrt{\epsilon}}(0)} \frac{1}{\epsilon^{d/2}}K'\left(\frac{t}{\sqrt{\epsilon}}\right)\frac{\|\Pi(\theta,\theta)\|t^{d+2}}{24\sqrt{\epsilon}}\ud t\ud\theta=\frac{\epsilon^{d/2+1} m'_3z(x)}{24|S^{d-1}|}.
\]

Suppose $\{E_l\}_{l=1}^d$ is an orthonormal basis of $T_{x}\MM$, and express $\theta=\sum_{l=1}^d\theta_lE_l$. A direct calculation shows that
\begin{align*} \begin{split} \int_{S^{d-1}}\nabla^2_{\theta,\theta}X(x)\ud\theta
=\sum_{k,l=1}^d\int_{S^{d-1}}\theta_l\theta_k\nabla^2_{E_l,E_k}X(x)\ud\theta=\frac{|S^{d-1}|}{d}\nabla^2
X(x),
\end{split}
\end{align*}
and similarly
\begin{align*}
\begin{split}
\int_{S^{d-1}}\Ric(\theta,\theta)d\theta=\frac{|S^{d-1}|}{d}s(x).
\end{split}
\end{align*}
Therefore, the first three terms of $A$ become
\begin{equation}\label{FPEdenominatorresulta1}
\epsilon^{d/2}p^{1-\alpha}(x)\left\{\left(1+\frac{\epsilon m_2}{d}\frac{\Delta(p^{1-\alpha})(x)}{2p^{1-\alpha}(x)}+\frac{\epsilon m_2}{d}w(x)\right)X(x)+\frac{\epsilon m_2}{2d}\nabla^2X(x)\right\}.
\end{equation}
The last term is simplified to
\begin{align*}
\begin{split}
\int_{S^{d-1}}\int_0^{\sqrt{\epsilon}} K\left(\frac{t}{\sqrt{\epsilon}}\right)\nabla_{\theta}X(x)\nabla_\theta(p^{1-\alpha})(x)t^{d+1}\ud t\ud\theta=\epsilon^{d/2+1}\frac{m_2}{|S^{d-1}|}\int_{S^{d-1}} \nabla_{\theta}X(x)\nabla_\theta(p^{1-\alpha})(x)\ud\theta.
\end{split}
\end{align*}

Next, we consider $B$. Note that since there is an $\epsilon$ in front of $B$, we only need to consider the leading order term. Denote $Q(y)=p^{1-\alpha}(y)\left(w(y)+\frac{\Delta p(y)}{2p(y)}\right)$ to simplify notation. Thus, applying Taylor's expansion to each of the terms in the integrand of $B$ leads to:
\begin{align*}
\begin{split}
B&=\int_{B_{\sqrt{\epsilon}}(x)} K_{\epsilon}(x,y)P_{x,y} X(y)Q(y)\ud V(y)\\
&=\int_{S^{d-1}}\int_0^{\sqrt{\epsilon}} \left[K\left(\frac{t}{\sqrt{\epsilon}}\right)+K'\left(\frac{t}{\sqrt{\epsilon}}\right)\frac{\|\Pi(\theta,\theta)\|t^3}{24\sqrt{\epsilon}}+O\left(\frac{t^6}{\epsilon}\right)\right]\left[X(x)+\nabla_\theta X(x)t+O(t^2)\right]\\
&\quad\quad\quad\left[Q(x)+\nabla_\theta Q(x)t+O(t^2)\right]\left[t^{d-1}+\Ric(\theta,\theta)t^{d+1}+O(t^{d+2})\right]\ud t\ud\theta\\
&=\epsilon^{d/2}X(x)Q(x)+O(\epsilon^{d/2+1})\\
\end{split}
\end{align*}
In conclusion, the numerator of $T_{\epsilon,\alpha}X(x)$ becomes
\begin{align*}
\begin{split}
&\quad\epsilon^{d/2} p_\epsilon^{-\alpha}(x)p^{1-\alpha}(x)\left\{1+\frac{\epsilon m_2}{d}\left[\frac{\Delta(p^{1-\alpha})(x)}{2p^{1-\alpha}(x)}-\alpha\frac{\Delta p(x)}{2p(x)}\right]\right\}X(x)\\
&+\epsilon^{d/2+1} \frac{m_2}{2d} p_\epsilon^{-\alpha}(x)p^{1-\alpha}(x)\nabla^2X(x)+\epsilon^{d/2+1} \frac{m_2}{|S^{d-1}|} p_\epsilon^{-\alpha}(x)\int_{S^{d-1}}\nabla_{\theta}X(x)\nabla_\theta(p^{1-\alpha})(x)\ud\theta+O(\epsilon^{d/2+2})
\end{split}
\end{align*}
Similar calculation of the denominator of the $T_{\epsilon,\alpha} X(x)$ gives
\begin{align*}
\begin{split}
&\quad\int_{B_{\sqrt{\epsilon}}(x)} K_{\epsilon,\alpha}(x,y)p(y)\ud V(y)\\
&=p_\epsilon^{-\alpha}(x)\int_{B_{\sqrt{\epsilon}}(x)} K_{\epsilon}(x,y)p^{1-\alpha}(y)\left[1-\alpha\epsilon\frac{m_2}{d}\left(w(y)+\frac{\Delta p(y)}{2p(y)}\right)\right]\ud V(y)+O(\epsilon^{d/2+2})\\
&=p_\epsilon^{-\alpha}(x)\int_{B_{\sqrt{\epsilon}}(x)} K_{\epsilon}(x,y)p^{1-\alpha}(y)\ud V(y)\\
&\quad-\frac{m_2\epsilon}{d}\alpha p_\epsilon^{-\alpha}(x)\int_{B_{\sqrt{\epsilon}}(x)} K_{\epsilon}(x,y)p^{1-\alpha}(y)\left(w(y)+\frac{\Delta p(y)}{2p(y)}\right)\ud V(y)+O(\epsilon^{d/2+2})\\
&=p_\epsilon^{-\alpha}(x)C-\frac{\epsilon m_2}{d}\alpha p_\epsilon^{-\alpha}(x)D+O(\epsilon^{d/2+2})
\end{split}
\end{align*}
where
\[\left\{\begin{array}{l}C:=\int_{B_{\sqrt{\epsilon}}(x)} K_{\epsilon}(x,y)p^{1-\alpha}(y)\ud V(y),\\
D:=\int_{B_{\sqrt{\epsilon}}(x)} K_{\epsilon}(x,y)p^{1-\alpha}(y)\left(w(y)+\frac{\Delta p(y)}{2p(y)}\right)\ud V(y).\end{array}\right.\]
We apply Lemma \ref{Kf} to $C$ and $D$:
\begin{align*}
\begin{split}
C=\epsilon^{d/2}p^{1-\alpha}(x)\left[1+\frac{\epsilon m_2}{d}\left(w(x)+\frac{\Delta(p^{1-\alpha})(x)}{2p^{1-\alpha}(x)}\right)\right]+O(\epsilon^{d/2+2}),
\end{split}
\end{align*}
and
\begin{align*}
\begin{split}
D=\epsilon^{d/2}p^{1-\alpha}_\epsilon(x)\left(s(x)+\frac{\Delta p(x)}{2p(x)}\right)+O(\epsilon^{d/2+1}).
\end{split}
\end{align*}
In conclusion, the denominator of $T_{\epsilon,\alpha}X(x)$ is
\begin{align*}
\begin{split}
\epsilon^{d/2}p_\epsilon^{-\alpha}(x)p^{1-\alpha}(x)\left\{1+\epsilon\frac{ m_2}{d}\left(\frac{\Delta(p^{1-\alpha})(x)}{2p^{1-\alpha}(x)}-\alpha\frac{\Delta p(x)}{2p(x)}\right)\right\}+O(\epsilon^{d/2+2})
\end{split}
\end{align*}

Putting all the above together, we have
\begin{align*}
\begin{split}
&T_{\epsilon,\alpha} X(x)=X(x)+\epsilon\frac{m_2}{2d}\nabla^2X(x)+\epsilon\frac{m_2}{|S^{d-1}|}\frac{\int_{S^{d-1}}\nabla_{\theta}X(x)\nabla_\theta(p^{1-\alpha})(x)\ud\theta}{p^{1-\alpha}(x)} +O(\epsilon^2)
\end{split}
\end{align*}
In particular, when $\alpha=1$, we have:
\begin{align*}
\begin{split}
T_{\epsilon,1} X(x)=X(x)+\epsilon\frac{m_2}{2d}\nabla^2 X(x)+O(\epsilon^2)
\end{split}
\end{align*}
\end{proof}

\subsection{[Proof of Theorem \ref{connlapbdry}]}

\begin{proof}
Suppose $\min_{y\in\partial\MM}d(x,y)=\tilde{\epsilon}$. Choose a normal coordinate $\{\partial_1,\ldots,\partial_{d}\}$ on the geodesic ball $B_{\epsilon^{1/2}}(x)$ around $x$ so that $x_0=\exp_x(\tilde{\epsilon}\partial_d(x))$. Due to Gauss Lemma, we know $\mbox{span}\{\partial_1(x_0),\ldots,\partial_{d-1}(x_0)\}=T_{x_0}\partial\MM$ and $\partial_d(x_0)$ is outer normal at $x_0$.

We focus first on the integral appearing in the numerator of $T_{\epsilon,1}X(x)$: $$\int_{B_{\sqrt{\epsilon}}(x)\cap \MM} \frac{1}{\epsilon^{d/2}}K_{\epsilon,1}(x,y)P_{x,y} X(y)p(y)\ud V(y).$$
We divide the integral domain $\exp_{x}^{-1}(B_{\sqrt{\epsilon}}(x)\cap \MM)$ into slices $S_\eta$ defined by
\[
S_{\eta}=\{(\vu,\eta)\in\RR^d:\|(u_1,\ldots,u_{d-1},\eta)\|<\sqrt{\epsilon}\},
\]
where $\eta\in[-\epsilon^{1/2},\epsilon^{1/2}]$ and $\vu=(u_1,\ldots,u_{d-1})\in\RR^{d-1}$. By Taylor's expansion and (\ref{approximatepalpha}), the numerator of $T_{\epsilon,1}X$ becomes
\begin{align}\label{bdrycal}
\begin{split}
&\quad\quad\int_{B_{\sqrt{\epsilon}}(x)\cap \MM} K_{\epsilon,1}(x,y) P_{x,y}X(y)p(y)\ud V(y)\\
&=p_{\epsilon}^{-1}(x)\int_{S_\eta}\int^{\sqrt{\epsilon}}_{-\sqrt{\epsilon}} K\left( \frac{\sqrt{\|\vu\|^2+\eta^2}}{\sqrt{\epsilon}} \right)\left(X(x)+\sum_{i=1}^{d-1}u_i\nabla_{\partial_i}X(x)+\eta\nabla_{\partial_d} X(x)+O(\epsilon)\right)\\
&\quad\quad \left[1-\epsilon\frac{m_2}{d}\left(w(y)+\frac{\Delta p(y)}{2p(y)}\right)+O(\epsilon^{2})\right]\ud \eta\ud \vu\\
&=p^{-1}(x)\int_{S_\eta}\int^{\sqrt{\epsilon}}_{-\sqrt{\epsilon}} K\left( \frac{\sqrt{\|\vu\|^2+\eta^2}}{\sqrt{\epsilon}} \right)\left(X(x)+\sum_{i=1}^{d-1}u_i\nabla_{\partial_i}X(x)+\eta\nabla_{\partial_d} X(x)+O(\epsilon)\right)\ud \eta\ud \vu
\end{split}
\end{align}
Note that in general the integral domain $S_\eta$ is not symmetric with related to $(0,\ldots,0,\eta)$, so we will try to symmetrize $S_\eta$ by defining the symmetrized slices:
\[
\tilde{S}_\eta=\cap^{d-1}_{i=1}(R_iS_\eta\cap S_\eta),
\]
where $R_i(u_1,\ldots,u_i,\ldots,\eta)=(u_1,\ldots,-u_i,\ldots,\eta)$. Note that from (\ref{jacobirelation}) in Lemma \ref{relatebasis}, the orthonormal basis $\{P_{x_0,x}\partial_1(x),\ldots,P_{x_0,x}\partial_{d-1}(x)\}$ of $T_{x_0}\partial\MM$ differ from $\{\partial_1(x_0),\ldots,\partial_{d-1}(x_0)\}$ by $O(\epsilon)$. Also note that up to error of order $\epsilon^{3/2}$, we can express $\partial\MM\cap B_{\epsilon^{1/2}}(x)$ by a homogeneous degree 2 polynomial with variables $\{P_{x_0,x}\partial_1(x),\ldots,P_{x_0,x}\partial_{d-1}(x)\}$. Thus the difference between $\tilde{S}_\eta$ and $S_\eta$ is of order $\epsilon$ and (\ref{bdrycal}) can be reduced to:
\begin{align}\label{bdrycal2}
\begin{split}
p^{-1}(x)\int_{\tilde{S}_\eta}\int^{\sqrt{\epsilon}}_{-\sqrt{\epsilon}} K\left( \frac{\sqrt{\|\vu\|^2+\eta^2}}{\sqrt{\epsilon}} \right)\left(X(x)+\sum_{i=1}^{d-1}u_i\nabla_{\partial_i}X(x)+\eta\nabla_{\partial_d} X(x)+O(\epsilon)\right)\ud \eta\ud \vu\\
\end{split}
\end{align}
Next, we apply Taylor's expansion on $X(x)$:
\[
P_{x,x_0}X(x_0)=X(x)+\tilde{\epsilon}\nabla_{\partial_d}X(x)+O(\epsilon).
\]
Since
\[
\nabla_{\partial_d}X(x)=P_{x,x_0}(\nabla_{\partial_d}X(x_0))+O(\epsilon^{1/2}),
\]
the Taylor's expansion of $X(x)$ becomes:
\begin{equation}\label{taylor1}
X(x)=P_{x,x_0}(X(x_0)-\tilde{\epsilon}\nabla_{\partial_d}X(x_0)+O(\epsilon)),
\end{equation}
Similarly for all $i=1,\ldots,d$ we have
\begin{equation}\label{taylor2}
P_{x,x_0}(\nabla_{\partial_i}X(x_0))=\nabla_{\partial_i}X(x)+O(\epsilon^{1/2})
\end{equation}
Plugging (\ref{taylor1}) and (\ref{taylor2}) into (\ref{bdrycal2}) further reduce (\ref{bdrycal}) into:
\begin{align}\label{bdrycal3}
\begin{split}
p^{-1}(x)\int_{\tilde{S}_\eta}\int^{\sqrt{\epsilon}}_{-\sqrt{\epsilon}} K\left( \frac{\sqrt{\|\vu\|^2+\eta^2}}{\sqrt{\epsilon}} \right)P_{x,x_0}\left(X(x_0)+\sum_{i=1}^{d-1}u_i\nabla_{\partial_i}X(x_0)+(\eta-\tilde{\epsilon})\nabla_{\partial_d}X(x_0)+O(\epsilon)\right)\ud \eta\ud \vu\\
\end{split}
\end{align}
The symmetry of the kernel implies that for $i=1,\ldots,d-1$,
\begin{equation}\label{symmcancel}
\int_{\tilde{S}_\eta}K\left( \frac{\sqrt{\|\vu\|^2+\eta^2}}{\sqrt{\epsilon}}\right)u^i\ud \vu=0,
\end{equation}
and hence the numerator of $T_{1,\epsilon}X(x)$ becomes
\begin{equation}\label{bdryrslt1}
p^{-1}(x)P_{x,x_0}(m^\epsilon_0X(x_0)+m^\epsilon_1\nabla_{\partial_d}X(x_0))+O(\epsilon^{d/2+1})
\end{equation}
where
\begin{equation}\label{meps0}
m^\epsilon_0=\int_{\tilde{S}_\eta}\int^{\sqrt{\epsilon}}_{-\sqrt{\epsilon}} K\left( \frac{\sqrt{\|u\|^2+\eta^2}}{\sqrt{\epsilon}}\right)\ud \eta\ud x=O(\epsilon^{d/2})
\end{equation}
and
\begin{equation}\label{meps1}
m^\epsilon_1=\int_{\tilde{S}_\eta}\int^{\sqrt{\epsilon}}_{-\sqrt{\epsilon}} K\left( \frac{\sqrt{\|u\|^2+\eta^2}}{\sqrt{\epsilon}}\right)(\eta-\tilde{\epsilon})\ud \eta\ud x=O(\epsilon^{d/2+1/2}).
\end{equation}
Similarly, the denominator of $T_{\epsilon,1}X$ can be expanded as:
\begin{equation}\label{bdryrslt2}
\int_{B_{\sqrt{\epsilon}}(x)\cap \MM} K_{\epsilon,1}(x,y)p(y)\ud V(y)=p^{-1}(x)m^\epsilon_0+O(\epsilon^{d/2+1/2}),
\end{equation}
which together with (\ref{bdryrslt1}) gives us the following asymptotic expansion:
\begin{equation}\label{bdryrslt3}
T_{\epsilon,1} X(x)=P_{x,x_0}\left(X(x_0)+\frac{m^\epsilon_1}{m^\epsilon_0}\nabla_{\partial_d}X(x_0)\right)+O(\epsilon).
\end{equation}
Combining (\ref{bdryrslt3}) with (\ref{convbdry}) in Theorem \ref{convergetokernel}, we conclude the theorem.
\end{proof}

\subsection{[Proof of Theorem \ref{summaryheatkernel}]}
\begin{proof}
We denote the spectrum of $\nabla^2$ by $\{\lambda_l\}_{l=0}^\infty$, where $0\leq\lambda_0\leq\lambda_1\leq \ldots$, and the corresponding eigenspaces by $E_l:=\{X\in L^2(T\MM):~\nabla^2X=-\lambda_l X\}$, $l=0,1,\ldots$. The eigen-vector-fields are smooth and form a basis for $L^2(T\MM)$, that is,
\[
L^2(T\MM)=\overline{\oplus_{l\in\NN\cup\{0\}} E_l}.
\]
Thus we proceed by considering the approximation through eigen-vector-field subspaces. To simplify notation, we rescale the kenrel $K$ so that $\frac{m_2}{2dm_0}=1$.

Fix $X_l\in E_l$. When $x\notin\MM_{\sqrt{\epsilon}}$, from Corollary \ref{Tepsexpansioncor} we have uniformly
\[
\frac{T_{\epsilon,1}X_l(x)-X_l(x)}{\epsilon}=\nabla^2X_l(x)+O(\epsilon).
\]
When $x\in\MM_{\sqrt{\epsilon}}$, from Theorem \ref{connlapbdry} and the Neumann condition, we have uniformly
\begin{equation}\label{tepsbdry}
T_{\epsilon,1}X_l(x)=P_{x,x_0}X_l(x_0)+O(\epsilon).
\end{equation}
Note that we have
\[
P_{x,x_0}X_l(x_0)=X_l(x)+P_{x,x_0}\sqrt{\epsilon}\nabla_{\partial_d}X_l(x_0)+O(\epsilon),
\]
thus again by the Neumann condition at $x_0$, (\ref{tepsbdry}) becomes
\[
T_{\epsilon,1}X_l(x)=X_l(x)+O(\epsilon).
\]
In conclusion, when $x\in\MM_{\sqrt{\epsilon}}$ uniformly we have
\[
\frac{T_{\epsilon,1}X_l(x)-X_l(x)}{\epsilon}=O(1).
\]
Note that when the boundary of the manifold is smooth, the measure of $\MM_{\sqrt{\epsilon}}$ is $O(\epsilon^{1/2})$. We conclude that in the $L^2$ sense,
\begin{equation}\label{Ieps1}
\left\|\frac{T_{\epsilon,1}X_l-X_l}{\epsilon}-\nabla^2X_l\right\|_{L^2} = O(\epsilon^{1/4}),
\end{equation}

Next we show how $T_{\epsilon,1}^{t/\epsilon}$ converges to $e^{-t\nabla^2}$. We know $I+\epsilon\nabla^2$ is invertible on $E_l$ with norm $\frac{1}{2}\leq\|I+\epsilon\nabla^2\|<1$ when $\epsilon<\frac{1}{2\lambda_l}$. Next, note that if $B$ is a bounded operator with norm $\|B\|<1$, we have the following bound for any $s>0$ by the binomial expansion:
\begin{align}\label{binomialexp}
\begin{split}
\|(I+B)^s-I\|&=\left\|sB+\frac{s(s-1)}{2!}B^2+\frac{s(s-1)(s-2)}{3!}B^3+\ldots\right\|\\
&\leq s\|B\|+\frac{s(s-1)}{2!}\|B\|^2+\frac{s(s-1)(s-2)}{3!}\|B\|^3+\ldots\\
&= s\|B\|\left\{1+\frac{s-1}{2!}\|B\|+\frac{(s-1)(s-2)}{3!}\|B\|^2+\ldots\right\}\\
&\leq s\|B\|\left\{1+\frac{s-1}{1!}\|B\|+\frac{(s-1)(s-2)}{2!}\|B\|^2+\ldots\right\}\\
&=s\|B\|(1+\|B\|)^{s-1}
\end{split}
\end{align}
On the other hand, note that on $E_l$
\begin{equation}\label{Ieps2}
e^{t\nabla^2}=(I+\epsilon\nabla^2)^{\frac{t}{\epsilon}}+O(\epsilon).
\end{equation}
Indeed, for $X\in E_l$, we have $e^{t\nabla^2}X=(1-t\lambda_l+t^2\lambda_l^2/2+\ldots)X$ and $(I+\epsilon\nabla^2)^{\frac{t}{\epsilon}}X=(1-t\lambda_l+t^2\lambda_l^2/2-t\epsilon\lambda_l^2/2+\ldots)X$ by the binomial expansion. Thus we have the claim.

Put all the above together, over $E_l$, for all $l\geq 0$, when $\epsilon<\frac{1}{2\lambda_l}$ we have:
\begin{align*}
\begin{split}
\|T^{\frac{t}{\epsilon}}_{\epsilon,1}-e^{t\nabla^2}\|&=
\left\|(I+\epsilon\nabla^2+O(\epsilon^{\frac{5}{4}}))^{\frac{t}{\epsilon}}-\left(I+\epsilon\nabla^2\right)^{\frac{t}{\epsilon}}+O(\epsilon)\right\|\\ &\leq \left\|\left(I+\epsilon\nabla^2\right)^{\frac{t}{\epsilon}}\right\|\left\|\left[I+\left(I+\epsilon\nabla^2\right)^{-1}O(\epsilon^{5/4})\right]^{\frac{t}{\epsilon}}-I+O(\epsilon)\right\|\\
&=(1+t+O(\epsilon))(\epsilon^{1/4}t+O(\epsilon))
=O(\epsilon^{\frac{1}{4}}),
\end{split}
\end{align*}
where the first equality comes from (\ref{Ieps1}) and (\ref{Ieps2}), the third inequality comes from (\ref{binomialexp}). Thus we have $\|T^{\frac{t}{\epsilon}}_{\epsilon,1}-e^{t\nabla^2}\|\leq O(\epsilon^{1/4})$ on $\overline{\oplus_{l: \lambda_l<\frac{1}{2\epsilon}} E_l}$. By taking $\epsilon\rightarrow 0$, the proof is completed.
\end{proof}

%% file: Snspectrum.tex
All results and proofs in this section can be found in \cite{fultonharris} and \cite{taylor}. Consider the following setting:
\begin{align*}
\begin{split}
G=SO(n+1),~~K=SO(n),~~\MM=G/K=S^{n},~~\fg=so(n+1),~~\fk=so(n)
\end{split}
\end{align*}
Denote $\Omega^1(S^{n})$ the complexified smooth $1$ forms, which is a $G$-module by $(g\cdot s)(x)=g\cdot s(g^{-1}x)$ for $g\in G$, $s\in \Omega^1(S^{n})$, and $x\in S^{n}$. Over $\mathcal{M}$ we have Haar measure $d\mu$ and Hodge Laplacian operator $\Delta=d\delta+\delta d$. Since $\Delta$ is a self-adjoint and uniform second order elliptic operator on $\Omega^p(S^{n})$, the eigenvalues $\lambda_i$ are discrete and non-negative real numbers, with only accumulation point at $\infty$, and their related eigenspaces $E_i$ are of finite dimension. We also know $\oplus_{i=1}^{\infty} E_i$ is dense in $\Omega^1(S^{n})$ in the topology defined by the inner product $(f,g)_{S^{n}}\equiv \int_{S^{n}}\langle f,g\rangle d\mu$, where $\langle\cdot,\cdot\rangle$ is the left invariant hermitian metric defined on $S^{n}$.

Since $\Delta$ is a $G$-invariant differential operator, its eigenspaces $E_\lambda$ are $G$-modules. We will count the multiplicity of $E_\lambda$ by first counting how many copies of $E_\lambda$ are inside $\Omega^1(S^{n})$ through the Frobenius reciprocity law, followed by the branching theorem and calculating $\dim E_\lambda$. On the other hand, since $S^{n}$ is a symmetric space, we know $\Delta=-C$ over $\Omega^1(S^{n})$, where $C$ is the Casimir operator on $G$, and we can determine the eigenvalue of $C$ over any finite dimensional irreducible submodule of $\Omega^1(S^{n})$ by Freudenthal's Formula. Finally we consider the relationship between real forms of $g$ and complex forms of $g$.

Note that $\fg/\fk\otimes_\RR\CC\cong \CC^n$ when $G=SO(n+1)$ and $K=SO(n)$. Denote $V=\CC^n=\Lambda^p(\fg/\fk)\otimes_\RR \CC$ as the standard representation of $SO(n)$.

There are two steps toward calculating the multiplicity of eigenforms over $S^{n}$.

\textbf{Step 1}
Clearly $\Omega^1(S^{n})$ is a reducible $G$-module. For $\lambda\in \mbox{Irr}(G,\CC)$, construct a $G$-homomorphism
\[
\Hom_G(\Gamma_\lambda,\Omega^1(S^{n}))\otimes_\CC \Gamma_\lambda\rightarrow \Omega^1(S^{n})
\]
by $\phi\otimes v\mapsto \phi(v)$. We call the image the $V_\lambda$-isotypical summand in $\Omega^1(S^{n})$ with multiplicity $\dim_\CC \Hom_G(\Gamma_\lambda,\Omega^1(S^{n}))$. Then we apply Frobenius reciprocity law:
\[
\Hom_G(\Gamma_\lambda,\Omega^1(S^{n}))\cong \Hom_K(\res^G_K \Gamma_\lambda,\Lambda^1V)
\]
Thus if we can calculate $\dim\Hom_K(\res^G_K \Gamma_\lambda,\Lambda^1V)$, we know how many copies of the irreducible representation $\Gamma_\lambda$ inside $\Omega^1(S^{n})$. To calculate it, we apply the following facts. When $V_i$ and $W_j$ are irreducible representations of $G$, we have by Schur's lemma:
\[
\Hom_G(\oplus^N_{i=1} V_i,\oplus^M_{j=1} W_j)\cong \oplus^{N,M}_{i=1,j=1} \Hom_G(V_i,W_j),
\]

Denote $L_1...L_n$ the basis for the dual space of Cartan subalgebra of $so(2n)$ or $so(2n+1)$. Then $L_1...L_n$, together with $\frac{1}{2}\sum^n_{i=1}L^i$ generate the weight lattice. The Weyl chamber of $SO(2n+1)$ is
\[\mathcal{W}=\left\{\sum a_iL_i: a_1\geq a_2\geq...\geq a_n\geq 0\right\},
\]
and the edges of the $\mathcal{W}$ are thus the rays generated by the vectors $L_1,L_1+L_2,...,L_1+L_2...+L_n$; for $SO(2n)$, the Weyl chamber is
\[\mathcal{W}=\left\{\sum a_iL_i: a_1\geq a_2\geq...\geq |a_n|\right\},
\]
and the edges are thus the rays generated by the vectors $L_1,L_1+L_2,...,L_1+L_2...+L_{n-2}, L_1+L_2...+L_n$ and $L_1+L_2...-L_n$.

To keep notations unified, we denote the fundamental weights $\omega_i$ separately. When $G=SO(2n)$, denote
\begin{equation}
\begin{array}{ll}
\omega_0=0& \mbox{when~~} p=0\\
\omega_p=\sum^p_{i=1}\lambda_i& \mbox{when~~} 1\leq p\leq n-1\\
\omega_n=\frac{1}{2}\sum^n_{i=1}\lambda_i& \mbox{when~~} p=n;
\end{array}
\end{equation}

when $G=SO(2n+1)$, denote
\begin{equation}
\begin{array}{ll}
\omega_0=0& \mbox{when~~} p=0\\
\omega_p=\sum^p_{i=1}\lambda_i& \mbox{when~~} 1\leq p\leq n-2\\
\omega_{n-1}=\frac{1}{2}\left(\sum^{n-1}_{i=1}\lambda_i-\lambda_n\right)& \mbox{when~~} p=n\\
\omega_n=\frac{1}{2}\left(\sum^{n-1}_{i=1}\lambda_i+\lambda_n\right)& \mbox{when~~} p=n
\end{array}
\end{equation}

\begin{thm}

(1) When $m=2n+1$, the exterior powers $\Lambda^pV$ of the standard representation $V$ of $so(2n+1)$ is the irreducible representation with the highest weight $\omega_p$, when $p< n$ and $2\omega_n$ when $p=n$.

(2) When $m=2n$, the exterior powers $\Lambda^pV$ of the standard representation $V$ of $so(2n)$ is the irreducible representation with the highest weight $\omega_p$, when $p\leq n-1$; when $p=n$, $\Lambda^nV$ splits into two irreducible representations with the highest weight $2\omega_{m-1}$ and $2\omega_m$.
\end{thm}

\begin{proof}
Please see \cite{fultonharris} for details.
\end{proof}

\begin{thm} (Branching theorem)

When $G=SO(2n)$ or $G=SO(2n+1)$, the restriction of the irreducible representations of $G$ will be decomposed as direct sum of the irreducible representations of $K=SO(2n-1)$ or $K=SO(2n)$ in the following way. Let $\Gamma_\lambda$ be an irreducible $G$-module over $\CC$ with the highest weight $\lambda=\sum^n_{i=1}\lambda_iL_i\in\mathcal{W}$

then as a $K$-module, $\Gamma_\lambda$ decomposes into $K$-irreducible modules as follows:

(1) if $m=2n$ ($G=SO(2n)$ and $K=SO(2n-1)$),
\[ \Gamma_\lambda=\oplus \Gamma_{\sum^n_{i=1} \lambda'_iL_i}\]
where $\oplus$ runs over all $\lambda'_i$ such that
\[\lambda_1\geq \lambda'_1\geq \lambda_2\geq \lambda'_2\geq...\geq \lambda'_{n-1}\geq |\lambda_n| \]
with $\lambda'_i$ and $\lambda_i$ simultaneously all integers or all half integers. Here $\Gamma_{\sum^n_{i=1} \lambda'_iL_i}$ is the irreducible $K$-module with the highest weight $\sum^n_{i=1} \lambda'_iL_i$

(2) if $m=2n+1$ ($G=SO(2n+1)$ and $K=SO(2n)$),
\[ \Gamma_\lambda=\oplus \Gamma_{\sum^{n}_{i=1} \lambda'_iL_i}\]
where $\oplus$ runs over all $\lambda'_i$ such that
\[\lambda_1\geq \lambda'_1\geq \lambda_2\geq \lambda'_2\geq...\geq \lambda'_{n-1}\geq \lambda_n\geq |\lambda'_n| \]
with $\lambda'_i$ and $\lambda_i$ simultaneously all integers or all half integers. Here $\Gamma_{\sum^{n}_{i=1} \lambda'_iL_i}$ is the irreducible $K$-module with the highest weight $\sum^{n}_{i=1} \lambda'_iL_i$

\end{thm}

\begin{proof}
Please see \cite{fultonharris} for details.
\end{proof}

Based on the above theorems, we know how to calculate $\dim_\CC \Hom_G(\Gamma_\lambda,\Omega^1(M))$. To be more precise, since the right hand side of $\Hom_K(\res^G_K \Gamma_\lambda,\Lambda^1V)$ is the irreducible representation of $K$ with the highest weight $\omega_1$ (or splits in the low dimension case), we know the $\dim\Hom_K(\res^G_K \Gamma_\lambda,\Lambda^1V)$ can be $1$ only if $\res^G_K \Gamma_\lambda$ has the same highest weight by Schur's lemma and classification theorem. Please see \cite{fultonharris} for details.

\textbf{Step 2}
In this step we relate the irreducible representation $\Gamma_\lambda\subset\Omega^1(S^n)$ to the eigenvalue of eigenforms. Consider the Laplace operator on $SO(n+1)$, which is endowed with a bi-invariant Riemannian metric. Since $SO(n+1)$ is semi-simple, we can take the metric on $\fg$ to be given by the negative of the Killing form: $(X,Y)=-\tr(adXadY)$. The Laplace operator $\Delta$ related to this metric is the Hodge-Laplace operator which enjoys the following relationship between its eigenvalue and its related highest weight.

\begin{thm}
Suppose $\Gamma_\mu\subset \Omega^1(S^n)$ is an irreducible $G$-module with the highest weight $\mu$, then we have
\[
\Delta=\langle \mu+2\rho,\mu\rangle Id_{\Gamma_\mu}
\]
where $f\in \Gamma_\mu$, $\rho$ is the half sum of all positive roots, and $\langle\cdot,\cdot\rangle$ is induced inner product on the dual Cartan subalgebra of $\fg$ from the Killing form $B$.
\end{thm}
\begin{proof}
Please see \cite{taylor} for details.
\end{proof}

Note that for $SO(2n+1)$, $\rho=\sum^n_{i=1}\left(n+\frac{1}{2}-i\right)L_i$ since $R^+=\left\{L_i-L_j,L_i+L_j,L_i: i<j\right\}$; for $SO(2n)$, $\rho=\sum^n_{i=1}\left(n-i\right)L_i$ since $R^+=\left\{L_i-L_j,L_i+L_j: i<j\right\}$.

Combining these theorems, we know if $\Gamma_\mu$ is an irreducible representation of $G$, then it is an eigenspace of $\Delta$ with eigenvalue $\lambda=\langle \mu+2\rho,\mu\rangle$. In particular, if we can decompose the eigenform space $E_\lambda\subset \Omega^1(S^n)$ into irreducible $G$-module, we can not only determine the eigenvalue $\lambda$ but also its multiplicity. Indeed, we know $E_\lambda=\Gamma_\mu^{\oplus N}$, where $\lambda=\langle \mu+2\rho,\mu\rangle$, is the isotypical summand of $\Gamma_\mu$ inside $\Omega^1(S^n)$.

\textbf{Step 3}
Now we apply Weyl Character Formula to calculate the dimension of $\Gamma_\lambda$ for $G$.

\begin{thm}

(1) When $m=2n+1$, consider $\lambda=\sum_{i=1}^n \lambda_iL_i$, where $\lambda_1\geq...\geq\lambda_n\geq 0$, the highest weight of an irreducible representation $\Gamma_\lambda$. Then $\dim \Gamma_\lambda=\Pi_{i<j} \frac{l_i-l_j}{j-i}\Pi_{i\leq j}\frac{l_i+l_j}{2n+1-i-j}$, where $l_i=\lambda_i+n-i+1/2$. In particular, when $\lambda=\omega_p$, $p\leq n$, $\dim \Gamma_\lambda=C^{2n+1}_p$.

(2) When $m=2n$, consider $\lambda=\sum_{i=1}^n \lambda_iL_i$, where $\lambda_1\geq...\geq|\lambda_n|$, the highest weight of an irreducible representation $\Gamma_\lambda$. Then $\dim \Gamma_\lambda=\Pi_{i<j} \frac{l_i-l_j}{j-i}\frac{l_i+l_j}{2n-i-j}$, where $l_i=\lambda_i+n-1$. In particular, when $\lambda=\omega_p$, $\dim \Gamma_\lambda=C^{2n}_p$ when $p<n$ and $\dim V_\lambda=C^{2n}_n/2$.
\end{thm}
\begin{proof}
Please see \cite{fultonharris} for details.
\end{proof}

\textbf{Step 4}
We need the following theorem about real representations of $G$ to solve the original problem:
\begin{thm}
(1) When $n$ is odd. Let $\omega_i$ be the highest weight of the representation $\Lambda^iV$ os $so_{2n+1}\CC$. For any weight $\lambda=a_1\omega_1+...+a_{n-1}\omega_{n-1}+a_n\omega_n/2$ or $so_{2n+1}\CC$, the irreducible representation $\Gamma_\lambda$ with highest weight $\lambda$ is real if $a_n$ is even, or if $n\cong 0$ or $3$ mod $4$; if $a_n$ is odd and $n\equiv 1$ or $2$ mod $4$, then $\Gamma_\lambda$ is quaternionic.

(2) When $n$ is even. The representation $\Gamma_\lambda$ of $so_{2n}\RR$ with highest weight $\lambda=a_1\omega_1+...+a_{n-2}\omega_{n-2}+a_{n-1}\omega_{n-1}+a_n\omega_n$ will be complex if $n$ is odd and $a_{n-1}\neq a_n$; it will be quaternionic if $n\equiv 2$ mod $4$ and $a_{n-1}+a_n$ is odd; and it will be real otherwise.
\end{thm}

\begin{proof}
Please see \cite{fultonharris} for details.
\end{proof}

Combining the following Tables \ref{table:2n} and \ref{table:2n+1} and this Theorem we know all the eigen-1-form spaces are real form.

\textbf{Step 5}
Now we put all the above together. All the eigen-1-form spaces of $S^n$ as an irreducible representation of $SO(n+1)$, $n\geq 4$, are listed in the table \ref{table:2n} and \ref{table:2n+1}. $\lambda$ is the highest weight. The $S^3$ and $S^2$ cases are listed in \ref{table:3} and \ref{table:2}. $S^3$ is separated since it has a different highest weight $kL_1-L_2$, which happens in general for $n$-forms in $S^{2n+1}$.

\begin{table}[H]
\centering \caption{Eigenvalues and their multiplicity of $S^{2n}$, where $n\geq2$. $l_i=\lambda_i+n-i$ and $m_i=n-i$.}
\label{table:2n} \begin{tabular}{c|c|c}
$\lambda$                                   & eigenvalues           & multiplicity \tabularnewline
\hline
$kL_1$, $k\geq1$                         & $k(k+2n-1) $           & $\Pi_{i<j} \frac{l^2_i-l^2_j}{m^2_i-m^2_j}$  \tabularnewline
$kL_1+L_2$, $k\geq1$                        & $(k+1)(k+2n-2)$     & $\Pi_{i<j} \frac{l^2_i-l^2_j}{m^2_i-m^2_j}$  \tabularnewline
\end{tabular}
\end{table}

\begin{table}[H]
\centering \caption{Eigenvalues and their multiplicity of $S^{2n+1}$, where $n\geq2$. $l_i=\lambda_i+n-i+1/2$ and $m_i=n-i+1/2$.}
\label{table:2n+1} \begin{tabular}{c|c|c}
$\lambda$                                   & eigenvalues           & multiplicity \tabularnewline
\hline
$kL_1$, $k\geq1$                         & $k(k+2n)$           & $\Pi_{i<j} \frac{l^2_i-l^2_j}{m^2_i-m^2_j}\Pi_i\frac{l_i}{m_i}$  \tabularnewline
$kL_1+L_2$, $k\geq1$                       & $(k+1)(k+2n-1)$     & $\Pi_{i<j} \frac{l^2_i-l^2_j}{m^2_i-m^2_j}\Pi_i\frac{l_i}{m_i}$  \tabularnewline
\end{tabular}
\end{table}

\begin{table}[H]
\centering \caption{Eigenvalues and their multiplicity of $S^3$.}
\label{table:3} \begin{tabular}{c|c|c}
$\lambda$                                   & eigenvalues           & multiplicity \tabularnewline
\hline
$kL_1$, $k\geq1$                         & $k(k+2)$           & $(k+1)^2$  \tabularnewline
$kL_1+L_2$, $k\geq1$                       & $(k+1)^2$     & $k(k+2)$  \tabularnewline
$kL_1-L_2$, $k\geq1$                       & $(k+1)^2$     & $k(k+2)$  \tabularnewline
\end{tabular}
\end{table}

\begin{table}[H]
\centering \caption{Eigenvalues and their multiplicity of $S^2$.}
\label{table:2} \begin{tabular}{c|c|c}
$\lambda$                                   & eigenvalues           & multiplicity \tabularnewline
\hline
$kL_1$ $k\geq 1$                 & $k(k+1)$           & $2(2k+1)$  \tabularnewline
\end{tabular}
\end{table}

Consider $SO(3)$ for example. In this case, $n=1$ and $\mathcal{M}=S^2$. From the analysis in cryo-EM \cite{hadani20091}, we know that the multiplicities of eigenvectors are $6,10,...$, which echoes the above analysis.

In conclusion, we can see the following compatible first few multiplicities:\\
$S^2$: 6, 10, 14\\
$S^3$: 4, 6, 9, 16, 16\\
$S^4$: 5, 10, 14\\
$S^5$: 6, 15, 20\\
$S^6$: 7, 21, 27